\documentclass[reqno,11pt,letterpaper]{amsart}

\usepackage{lipsum}
\usepackage{amsmath}
\usepackage{amssymb}
\usepackage{amsthm}
\usepackage{mathrsfs}
\usepackage{accents}
\usepackage{calc}
\usepackage{arydshln}
\usepackage{upgreek}
\usepackage{slashed}
\usepackage{xifthen}
\usepackage{graphicx}
\usepackage{longtable}
\usepackage[inline]{enumitem}

\usepackage{xcolor}
\definecolor{winered}{rgb}{0.6,0,0}
\definecolor{lessblue}{rgb}{0,0,0.7}

\usepackage[pdftex,colorlinks=true,linkcolor=winered,citecolor=lessblue,urlcolor=lessblue,breaklinks=true,bookmarksopen=true]{hyperref}

\makeatletter
\newcommand{\myitem}[3]{\item[#2]\def\@currentlabel{#3}\label{#1}}
\makeatother

\usepackage{titletoc}

\makeatletter

\def\@tocline#1#2#3#4#5#6#7{
\begingroup
  \par
    \parindent\z@ \leftskip#3 \relax \advance\leftskip\@tempdima\relax
                  \rightskip\@pnumwidth plus 4em \parfillskip-\@pnumwidth
    \ifcase #1 
       \vskip 0.6em \hskip 0em 
       \or
       \or \hskip 0em 
       \or \hskip 1em 
    \fi%
    #6
    \nobreak\relax{\leavevmode\leaders\hbox{\,.}\hfill}
    \hbox to\@pnumwidth {\@tocpagenum{#7}}
  \par
\endgroup
}

 \def\l@section{\@tocline{0}{0pt}{0pc}{}{}}

\renewcommand{\tocsection}[3]{%
  \indentlabel{\@ifnotempty{#2}{ 
    \ignorespaces\bfseries{#2. #3}}}
  \indentlabel{\@ifempty{#2}{\ignorespaces\bfseries{#3}}{}} 
    \vspace{1.5pt}}

\renewcommand{\tocsubsection}[3]{%
  \indentlabel{\@ifnotempty{#2}{
    \ignorespaces#2. #3}}
  \indentlabel{\@ifempty{#2}{\ignorespaces #3}{}}
    \vspace{1.5pt}}

\renewcommand{\tocsubsubsection}[3]{%
  \indentlabel{\@ifnotempty{#2}{
    \ignorespaces#2. #3}}
  \indentlabel{\@ifempty{#2}{\ignorespaces #3}{}}
    \vspace{1.5pt}}

\makeatother

\makeatletter
\def\@nomenstarted{0}
\newlength{\@nomenoldtabcolsep}

\newcommand{\nomenstart}
  {%
    \def\@nomenstarted{1}%
    \setlength{\@nomenoldtabcolsep}{\tabcolsep}%
    \setlength{\tabcolsep}{3.5pt}%
    \begin{longtable}{p{0.11\textwidth} p{0.86\textwidth}}
  }

\newcommand{\nomenitem}[2]{%
    \ifcase\@nomenstarted%
      \or 
      \or \\ 
    \fi%
    #1\,{\leavevmode\leaders\hbox{\,.}\hfill} & #2%
    \def\@nomenstarted{2}%
  }%
\newcommand{\nomenend}
  {\\%
      \end{longtable}%
      \setlength{\tabcolsep}{\@nomenoldtabcolsep}%
      \def\@nomenstarted{0}%
  }
\makeatother

\makeatletter
\newcommand{\vast}{\bBigg@{4}}
\newcommand{\Vast}{\bBigg@{5}}
\newcommand{\VAST}[1]{\bBigg@{#1}}
\makeatother

\allowdisplaybreaks

\numberwithin{equation}{section}
\numberwithin{figure}{section}
\newtheorem{thm}{Theorem}[section]

\newtheorem{prop}[thm]{Proposition}
\newtheorem{lemma}[thm]{Lemma}
\newtheorem{cor}[thm]{Corollary}
\newtheorem*{thm*}{Theorem}
\newtheorem*{prop*}{Proposition}
\newtheorem*{cor*}{Corollary}
\newtheorem*{conj*}{Conjecture}

\theoremstyle{definition}
\newtheorem{definition}[thm]{Definition}

\theoremstyle{remark}
\newtheorem{rmk}[thm]{Remark}
\newtheorem{example}[thm]{Example}

\makeatletter
\newcommand{\fakephantomsection}{%
  \Hy@MakeCurrentHref{\@currenvir.\the\Hy@linkcounter}
  \Hy@raisedlink{\hyper@anchorstart{\@currentHref}\hyper@anchorend}%
  \Hy@GlobalStepCount\Hy@linkcounter%
}
\makeatother

\newcommand{\mc}{\mathcal}
\newcommand{\cA}{\mc A}

\newcommand{\cC}{\mc C}
\newcommand{\cD}{\mc D}
\newcommand{\cE}{\mc E}
\newcommand{\cF}{\mc F}
\newcommand{\cG}{\mc G}

\newcommand{\cK}{\mc K}
\newcommand{\cL}{\mc L}
\newcommand{\cM}{\mc M}
\newcommand{\cN}{\mc N}
\newcommand{\cO}{\mc O}
\newcommand{\cP}{\mc P}

\newcommand{\cR}{\mc R}
\newcommand{\cS}{\mc S}

\newcommand{\cU}{\mc U}
\newcommand{\cV}{\mc V}

\newcommand{\cX}{\mc X}
\newcommand{\cY}{\mc Y}

\newcommand{\ms}{\mathscr}

\newcommand{\sD}{\ms D}

\newcommand{\C}{\mathbb{C}}
\newcommand{\N}{\mathbb{N}}
\newcommand{\R}{\mathbb{R}}
\newcommand{\Z}{\mathbb{Z}}

\newcommand{\Sph}{\mathbb{S}}

\newcommand{\sfs}{\mathsf{s}}

\newcommand{\bfw}{\mathbf{w}}

\newcommand{\bfB}{\mathbf{B}}

\newcommand{\fa}{\mathfrak{a}}

\newcommand{\fe}{\mathfrak{e}}

\newcommand{\fm}{\mathfrak{m}}
\newcommand{\fp}{\mathfrak{p}}

\newcommand{\ft}{\mathfrak{t}}

\newcommand{\ran}{\operatorname{ran}}

\newcommand{\End}{\operatorname{End}}

\newcommand{\Hom}{\operatorname{Hom}}
\renewcommand{\Re}{\operatorname{Re}}
\renewcommand{\Im}{\operatorname{Im}}
\newcommand{\Id}{\operatorname{Id}}

\newcommand{\mathspan}{\operatorname{span}}
\newcommand{\supp}{\operatorname{supp}}

\newcommand{\tr}{\operatorname{tr}}

\newcommand{\diag}{\operatorname{diag}}

\newcommand{\eps}{\epsilon}

\newcommand{\hra}{\hookrightarrow}
\newcommand{\la}{\langle}

\newcommand{\extcup}{\mathrel{\ol\cup}}

\newcommand{\ol}{\overline}
\newcommand{\pa}{\partial}
\newcommand{\dd}{{\mathrm d}}
\newcommand{\ra}{\rangle}

\newcommand{\ul}[1]{\underline{#1}{}}

\newcommand{\wh}{\widehat}
\newcommand{\wt}{\widetilde}
\newcommand{\xra}{\xrightarrow}
\newcommand{\ubar}[1]{\underaccent{\bar}#1}
\newcommand{\pfstep}[1]{$\bullet$\ \underline{\textit{#1}}}

\newcommand{\bop}{{\mathrm{b}}}

\newcommand{\qop}{{\mathrm{q}}}

\newcommand{\scop}{{\mathrm{sc}}}

\newcommand{\cl}{{\mathrm{cl}}}

\newcommand{\cp}{{\mathrm{c}}}

\newcommand{\Diff}{\mathrm{Diff}}

\DeclareMathOperator{\Op}{Op}

\newcommand{\Vb}{\cV_\bop}

\newcommand{\Vq}{\cV_\qop}

\newcommand{\Diffb}{\Diff_\bop}

\newcommand{\Diffq}{\Diff_\qop}

\newcommand{\Vsc}{\cV_\scop}

\newcommand{\Omegab}{{}^{\bop}\Omega}

\newcommand{\Tb}{{}^{\bop}T}

\newcommand{\Tq}{{}^\qop T}
\newcommand{\Tsc}{{}^{\scop}T}

\newcommand{\sigmab}{{}^\bop\upsigma}

\newcommand{\sigmaq}{{}^\qop\upsigma}

\newcommand{\loc}{{\mathrm{loc}}}
\newcommand{\CI}{\cC^\infty}
\newcommand{\CIdot}{\dot\cC^\infty}

\newcommand{\CIc}{\cC^\infty_\cp}

\newcommand{\Hb}{H_{\bop}}

\newcommand{\Hbext}{\bar H_{\bop}}

\newcommand{\phg}{{\mathrm{phg}}}

\newcommand{\Ric}{\mathrm{Ric}}
\newcommand{\Ein}{\mathrm{Ein}}
\newcommand{\bhm}{\fm}
\newcommand{\bha}{\fa}

\newcommand{\openbigpmatrix}[1]
  {%
    \def\@bigpmatrixsize{#1}%
    \addtolength{\arraycolsep}{-#1}%
    \begin{pmatrix}%
  }
\newcommand{\closebigpmatrix}
  {%
    \end{pmatrix}%
    \addtolength{\arraycolsep}{\@bigpmatrixsize}%
  }

\newlength{\enummargin}

\newcommand{\usref}[1]{{\upshape\ref{#1}}}

\DeclareGraphicsExtensions{.mps}

\makeatletter
\newcommand*{\fwbw}[1]{\expandafter\@fwbw\csname c@#1\endcsname}
\newcommand*{\@fwbw}[1]{\ifcase #1 \or {\rm fw}\or {\rm bw}\fi}
\AddEnumerateCounter{\fwbw}{\@fwbw}
\makeatother

\begin{document}

\title{Gluing small black holes into initial data sets}

\date{\today}


\subjclass[2010]{Primary 83C05, 35B25, Secondary 35C20, 35N10, 83C57}

\author{Peter Hintz}
\address{Department of Mathematics, ETH Z\"urich, R\"amistrasse 101, 8092 Z\"urich, Switzerland}
\email{peter.hintz@math.ethz.ch}

\begin{abstract}
  We prove a strong localized gluing result for the general relativistic constraint equations (with or without cosmological constant) in $n\geq 3$ spatial dimensions. We glue an $\eps$-rescaling of an asymptotically flat data set $(\hat\gamma,\hat k)$ into the neighborhood of a point $\fp\in X$ inside of another initial data set $(X,\gamma,k)$, under a local genericity condition (non-existence of KIDs) near $\fp$. As the scaling parameter $\eps$ tends to $0$, the rescalings $\frac{x}{\eps}$ of normal coordinates $x$ on $X$ around $\fp$ become asymptotically flat coordinates on the asymptotically flat data set; outside of any neighborhood of $\fp$ on the other hand, the glued initial data converge back to $(\gamma,k)$. The initial data we construct enjoy polyhomogeneous regularity jointly in $\eps$ and the (rescaled) spatial coordinates.

  Applying our construction to unit mass black hole data sets $(X,\gamma,k)$ and appropriate boosted Kerr initial data sets $(\hat\gamma,\hat k)$ produces initial data which conjecturally evolve into the extreme mass ratio inspiral of a unit mass and a mass $\eps$ black hole.

  The proof combines a variant of the gluing method introduced by Corvino and Schoen with geometric singular analysis techniques originating in Melrose's work. On a technical level, we present a fully geometric microlocal treatment of the solvability theory for the linearized constraints map.
\end{abstract}

\maketitle

\tableofcontents

\section{Introduction}
\label{SI}

Let $n\geq 3$. For a smooth Riemannian $n$-manifold $(X,\gamma)$ and a symmetric 2-tensor $k$ on $X$, the \emph{constraint equations} for $\gamma,k$ are
\begin{equation}
\label{EqICE}
  \begin{cases}
    R_\gamma - |k|_\gamma^2 + (\tr_\gamma k)^2 = 2\Lambda, \\
    \delta_\gamma k + \dd(\tr_\gamma k) = 0.
  \end{cases}
\end{equation}
Here, $R_\gamma$ is the scalar curvature of $\gamma$, $\delta_\gamma$ is the negative divergence, and $\Lambda\in\R$ is the cosmological constant. We say that $(X,\gamma,k)$ is an \emph{initial data set} if~\eqref{EqICE} holds. Given a Lorentzian manifold $(M,g)$ of dimension $(n+1)$ and signature $(-,+,\ldots,+)$, and given an embedded spacelike hypersurface $X\subset M$, the first and second fundamental form $\gamma$ and $k$ of $X$ satisfy~\eqref{EqICE} provided $g$ satisfies the Einstein vacuum equations
\begin{equation}
\label{EqIEVE}
  \Ein(g)+\Lambda g=0,\qquad \Ein(g):=\Ric(g)-\frac12 R_g g.
\end{equation}
The fundamental theorems of Choquet-Bruhat and Geroch \cite{ChoquetBruhatLocalEinstein,ChoquetBruhatGerochMGHD} show, conversely, that given $(X,\gamma,k)$ satisfying~\eqref{EqICE}, there exists a maximal globally hyperbolic spacetime $(M,g)$, unique up to isometries, which solves the Einstein vacuum equations and into which $X$ embeds with first and second fundamental form given by $\gamma$ and $k$, respectively.

We consider the problem of gluing the initial data of a small asymptotically flat black hole (such as a Schwarzschild or Kerr black hole with small mass), or more generally of a rescaled asymptotically flat initial data set $(\R^n\setminus\hat K^\circ,\hat\gamma,\hat k)$ with cosmological constant $0$, into a neighborhood of a point $\fp\in X$ in a given generic (near $\fp$) smooth initial data set $(X,\gamma,k)$ with arbitrary cosmological constant $\Lambda\in\R$. Here $\hat K\subset\R^n$ is compact (and possibly empty); and the asymptotic flatness condition, in standard coordinates $\hat x\in\R^n$, means that $\hat\gamma(\pa_{\hat x^i},\pa_{\hat x^j})\to\delta_{i j}$ and $k(\pa_{\hat x^i},\pa_{\hat x^j})\to 0$ as $|\hat x|\to\infty$, with appropriate rates of convergence. (See Definition~\ref{DefCEAf} for the precise definition used in this paper.)

\begin{thm}[Main result, rough version]
\label{ThmI}
  Assume that $X$ is generic in a connected neighborhood $\cU^\circ$ of $\fp$ (in the sense that it does not admit any KIDs in $\cU^\circ$, see Definition~\usref{DefCEPctKID}). Then there exist $\eps_\sharp>0$ and a family $(X_\eps,\gamma_\eps,k_\eps)$, $\eps\in(0,\eps_\sharp)$, of initial data sets with cosmological constant $\Lambda$ with the following properties.
  \begin{enumerate}
  \item On $X\setminus\cU^\circ\subset X_\eps$, we have $(\gamma_\eps,k_\eps)=(\gamma,k)$.
  \item In geodesic normal coordinates $x=(x^1,\ldots,x^n)\in\R^n$ around $\fp\in X$, the manifold $X_\eps$ is equal to $B(0,1)\setminus\eps\hat K^\circ$, and we have smooth convergence $(\gamma_\eps,k_\eps)\to(\gamma,k)$ as $\eps\searrow 0$ in $|x|>\delta$ for any $\delta>0$. (That is, the matrix coefficients $(\gamma_\eps)_{i j}=\gamma_\eps(\pa_{x^i},\pa_{x^j})$ and $k_\eps(\pa_{x^i},\pa_{x^j})$ converge to those of $\gamma$ and $k$.)
  \item\label{ItIhat} The tensors $\gamma_\eps|_{\eps\hat x}=(\gamma_\eps|_{\eps\hat x}(\pa_{x^i},\pa_{x^j}))_{i,j=1,\ldots,n}$ and $\eps k_\eps|_{\eps\hat x}$ converge, smoothly and locally uniformly in $\hat x\in\R^n\setminus\hat K^\circ$, to $\hat\gamma|_{\hat x}=(\hat\gamma|_{\hat x}(\pa_{\hat x^i},\pa_{\hat x^j}))$ and $\hat k|_{\hat x}$, respectively, as $\eps\searrow 0$.
  \end{enumerate}
\end{thm}

To explain part~\eqref{ItIhat}, we first note that $\eps^2(\hat\gamma,\hat k)$ is a rescaled asymptotically flat data set: asymptotically flat coordinates for it are $\eps\hat x$, which as $\eps\searrow 0$ become \emph{local} coordinates $x$ on $X$ near $\fp$. Note that the ADM mass of $\eps^2(\hat\gamma,\hat k)$ is $\eps$ times that of $(\hat\gamma,\hat k)$. Since in the coordinates $\hat x=\frac{x}{\eps}$ we have $\pa_{\hat x^i}=\eps\pa_{x^i}$, part~\eqref{ItIhat} states that $(\gamma_\eps,\eps k_\eps)\approx\eps^2(\hat\gamma,\hat k)$ in $x$-coordinates when $|x|\lesssim\eps$.\footnote{The $\eps^{-1}$-scaling of the second fundamental form in $(\gamma_\eps,k_\eps)\approx(\eps^2\hat\gamma,\eps^{-1}\eps^2\hat k)$ arises from the scaling properties of~\eqref{EqICE}, see Lemma~\ref{LemmaCETotScale}; heuristically, it follows from the fact that the future unit normal for the embedding of $(X,\gamma,k)$ into the spacetime $M=\R_t\times X$ is, near $\fp$, scaled by $\eps^{-1}$ relative to the future unit normal of the embedding of $(\hat X,\hat\gamma,\hat k)$, $\hat X=\R_{\hat x}^n$, into $\R_{\hat t}\times\hat X$; that is, the time scale $\hat t$ of the small black hole is related to the time scale $t$ of the ambient spacetime $M$ by $\hat t=\frac{t}{\eps}$.} See Figure~\ref{FigI}. In the region $\eps\lesssim|x|\lesssim 1$, the family $(\gamma_\eps,k_\eps)$ transitions from the $\eps$-rescaling $\eps^2(\hat\gamma,\hat k)$ to the original data set $(\gamma,k)$ in an appropriate manner; we describe this more precisely in~\S\ref{SsIG} below.

\begin{figure}[!ht]
\centering
\includegraphics{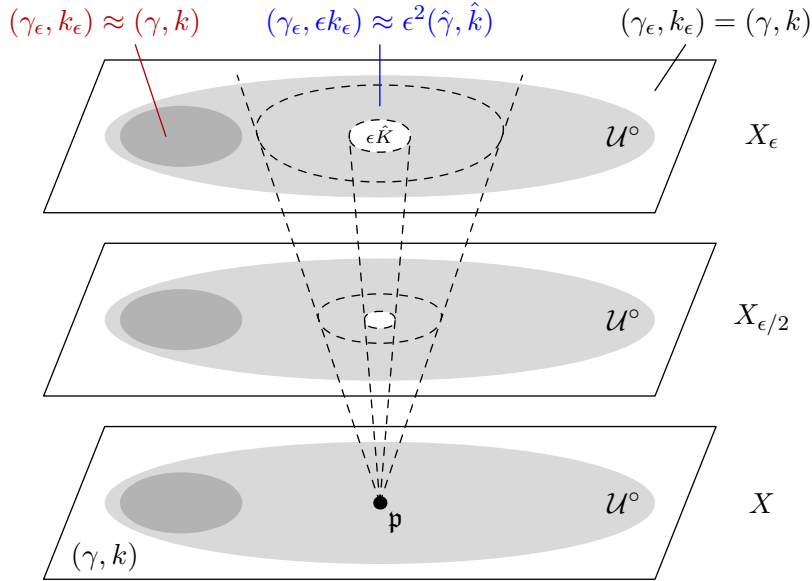}
\caption{Illustration of Theorem~\ref{ThmI}. The data set $(\gamma,k)$ is unchanged outside of $\cU^\circ$, and in $\ol{\cU^\circ}\setminus\{\fp\}$ it converges smoothly and locally uniformly to $(\gamma,k)$ as $\eps\searrow 0$. Near $\fp$ on the other hand, the glued initial data are close to the $\eps$-rescaling $\eps^2(\hat\gamma,\hat k)$ of the asymptotically flat data set $(\hat\gamma,\hat k)$.}
\label{FigI}
\end{figure}

A natural choice for $(\hat\gamma,\hat k)$ for $n=3$ is the initial data set of a (boosted) Schwarzschild or Kerr black hole \cite{SchwarzschildPaper,KerrKerr}; see~\S\ref{SsAK}. To illustrate Theorem~\ref{ThmI} in this case, consider the initial data $(\hat\gamma,\hat k)$ of an unboosted mass $\hat\bhm>0$ Schwarzschild black hole in polar coordinates $\hat x=\hat r\omega$, $\omega\in\Sph^2$,
\[
  \hat\gamma=\Bigl(1-\frac{\hat\bhm}{\hat r}\Bigr)\dd\hat r^2 + \hat r^2 g_{\Sph^2},\qquad \hat k=0,
\]
on $\R^3\setminus\hat K^\circ$ where $\hat K$ is a bounded closed ball of radius $>\max(0,2\hat\bhm)$. Then
\[
  \eps^2\hat\gamma = \Bigl(1-\frac{\eps\hat\bhm}{r}\Bigr)\dd r^2+r^2 g_{\Sph^2},\qquad r=\eps\hat r,
\]
is the metric of a mass $\eps\hat\bhm$ Schwarzschild black hole. The initial data set $(\gamma_\eps,k_\eps)$ thus describes a mass $\eps\hat\bhm$ black hole glued into the given data set $(X,\gamma,k)$.

We do not concern ourselves here with the construction of appropriate initial data sets $(X,\gamma,k)$. We recall that Beig--Chru\'sciel--Schoen \cite{BeigChruscielSchoenKIDs} demonstrated the genericity of the assumption on the absence of KIDs in a large number of settings; moreover, Moncrief \cite{MoncriefLinStabI} showed that the absence of KIDs is equivalent to the absence of Killing vector fields in the evolving spacetime. KIDs on $(X,\gamma,k)$ near $\fp$ are nonzero elements, defined in a neighborhood of $\fp$, of the cokernel of the linearization of the constraint equations around $(\gamma,k)$. While this cokernel is necessarily finite-dimensional, it may be non-trivial. However, a non-trivial cokernel is typically an obstruction for localized gluing constructions or deformations of initial data sets; for instance, the rigidity part of the Positive Mass Theorem \cite{SchoenYauPMT} in the time-symmetric setting ($k=0$, and $\gamma$ is scalar-flat) implies that one cannot compactly perturb the Euclidean metric to a non-isometric scalar-flat metric. (See however the work by Czimek--Rodnianski \cite{CzimekRodnianskiGluing} on how to overcome the presence of a cokernel in characteristic gluing problems by taking advantage of the nonlinear nature of the constraint equations.) In this paper, we impose the local genericity condition on $(X,\gamma,k)$ in order to obtain a local gluing result in Theorem~\ref{ThmI} via an appropriate solvability theory for the linearized constraints map.

The initial data of a subextremal Kerr (or Kerr--de~Sitter or Kerr--anti de~Sitter) black hole do not satisfy the genericity condition required in Theorem~\ref{ThmI}. Nonetheless, we show in~\S\ref{SsAK} how to prove Theorem~\ref{ThmI} also in this case, provided $\cU^\circ$ intersects the black hole interior; see Theorem~\ref{ThmAK}. (The main idea in the proof is to eliminate the cokernel by allowing for a violation of the constraint equations deep inside the black hole.) In particular, we are thus able to glue a small black hole into a unit mass black hole initial data set.

\subsection{Context}
\label{SsIC}

Starting with the construction by Majumdar--Papapetrou \cite{MajumdarSolution,PapapetrouSolution} of electrovacuum spacetimes via the superposition of extremally charged black holes (with Kastor--Traschen \cite{KastorTraschenManyBH} performing a similar construction in $\Lambda>0$), several constructions of initial data sets containing several black hole regions have been proposed. By solving the constraint equations using explicit ansatzes mainly involving superpositions of scalings and translations of the harmonic function $|x|^{-1}$, Brill--Lindquist \cite{BrillLindquist} constructed initial data for the Einstein--Maxwell equations with any finite number $N$ of charged wormholes (Einstein--Rosen bridges) at arbitrary points in $\R^3$ with arbitrary masses; the resulting data have $N+1$ asymptotically flat regions. Misner's time-symmetric `matched throat' vacuum initial data identify all but $2$ asymptotically flat ends; see Lindquist \cite{LindquistIVP} for the Einstein--Maxwell case. There also exist hybrid approaches, such as the one developed by Brandt--Bruegmann \cite{BrandtBruegmannMultipleBH} which involves an explicit prescription for the conformal class of $k$ (with $\gamma$ conformally Euclidean) together with a numerical scheme for finding the conformal factor.

The key mathematical technique allowing for flexible and localized gluing constructions for solutions of the constraint equations was introduced by Corvino \cite{CorvinoScalar} with Schoen \cite{CorvinoSchoenAsymptotics}; see also \cite{ChruscielDelayMapping} and~\S\ref{SsIG} below. It is based on the observation that the adjoint of the linearized constraint equations is overdetermined and permits coercive estimates on function spaces with very strong weights at the boundary of the gluing region. Concretely, Corvino--Schoen prove that an asymptotically flat data set can be perturbed near infinity to an exact Kerr data set for a suitable choice of Kerr parameters (mass and angular momentum). The relevant linear operator in this setting (namely, the linearization of the constraints map around the trivial Minkowski data) has a nontrivial cokernel, which is accounted for by appropriately choosing the parameters of the Kerr data set.

This technique was generalized and used by Chru\'sciel--Delay \cite[\S8.9]{ChruscielDelayMapping} to construct initial data containing many Kerr black holes. (See \cite{ChruscielDelaySimple} for the time-symmetric case of data containing several Schwarzschild black holes.) The initial data of \cite{ChruscielDelayMapping} are symmetric under the parity map $x\mapsto -x$; given pairwise disjoint balls $B(x_i,4 r_i)$, the data are equal to Kerr data (with arbitrary parameters, subject to the parity condition) in each of the $B(x_i,2 r_i)$, and also in $\R^3\setminus\bigcup B(x_i,4 r_i)$. The gluing procedure succeeds when all black hole masses are sufficiently small relative to the radii $r_i$ and the pairwise distances of the $x_i$. We also recall that Isenberg--Mazzeo--Pollack \cite[\S9]{IsenbergMazzeoPollackWormholes} constructed many-black-hole initial data by connecting Euclidean spaces to the neighborhood of any finite number of points on a given asymptotically flat maximal (i.e.\ $\tr_\gamma k=0$) data set via wormholes. Chru\'sciel--Mazzeo \cite{ChruscielMazzeoManyBH} proved the presence of multiple black holes in the spacetime development of the data produced in \cite{ChruscielDelaySimple,ChruscielDelayMapping,IsenbergMazzeoPollackWormholes} under suitable smallness conditions. See the hypotheses in \cite[\S3]{ChruscielMazzeoManyBH} for details; in particular, the first fundamental form must globally be sufficiently close to the Euclidean metric.

Another construction of many-black-hole initial data was given by Chru\'sciel--Corvino--Isenberg \cite{ChruscielCorvinoIsenbergNbody}. Phrasing their main result in a manner which relates more directly to the present work, \cite{ChruscielCorvinoIsenbergNbody} glues sufficiently small rescalings of large compact subsets of $N$ given asymptotically flat initial data sets (modified to be exact Kerr data near their respective asymptotically flat ends) into $N$ disjoint balls in $\R^3$; the glued data are exact Kerr data, with carefully chosen small parameters, also near infinity in $\R^3$.

The many-black-hole initial data sets obtained by (repeated\footnote{One can extend Theorem~\ref{ThmI} to simultaneously glue any finite number of asymptotically flat data sets into neighborhoods of a matching number of distinct points in $X$, under a local genericity condition near each of these points. One may also apply Theorem~\ref{ThmI} to a data set $(X_\eps,\gamma_\eps,k_\eps)$, for some fixed small $\eps>0$, in place of $(X,\gamma,k)$, or more generally iterate a combination of these two procedures finitely many times.}) application of Theorem~\ref{ThmI} combine features of those of \cite{IsenbergMazzeoPollackWormholes} (in that the background geometry is arbitrary) with those of \cite{ChruscielDelayMapping,ChruscielCorvinoIsenbergNbody} (in that one glues in rescaled asymptotically flat data sets into neighborhoods of points in a background data set). We stress, however, that the background data $(X,\gamma,k)$ and the asymptotically flat data $(\hat\gamma,\hat k)$ in Theorem~\ref{ThmI} are unrestricted except for the genericity assumption on $X$ (which can be relaxed in some circumstances, cf.\ \S\ref{SsAK}). Moreover, and importantly for intended future applications (see~\S\ref{SsIO}), our construction gives detailed control on the $\eps$-dependence of the glued initial data; see~\S\ref{SsIG}. (Moreover, this control cannot be obtained by a naive application of Corvino-type gluing methods; see Remark~\ref{RmkIGCorvino}.)

In the presence of a positive cosmological constant $\Lambda>0$, the author showed how to glue Schwarzschild-- or Kerr--de~Sitter black holes into neighborhoods of points on the future conformal boundary of de~Sitter space \cite{HintzGluedS}, thus producing examples of spacetimes solving~\eqref{EqIEVE} with precisely controlled structure at future timelike infinity and the future conformal boundary. A fortiori, this gives initial data sets on $\Sph^3$ (or its punctured version $\R^3$) containing any finite number of (large) regions with exact Kerr--de~Sitter data, while away from these regions the data are close to de~Sitter data. Apart from \cite{HintzGluedS,ChruscielMazzeoManyBH}, there do not appear to be any results (beyond simple domain-of-dependence type considerations) on the structure of the future evolution of the many-black-hole initial data sets obtained by any of the aforementioned methods. The author conjectures that a subclass of initial data of the type constructed by Theorem~\ref{ThmI} allows for such results; see~\S\ref{SsIO}.

Further results on initial data gluing which are not directly tied to many-black-hole settings include Cortier's gluing into Kerr--de~Sitter data sets \cite{CortierKdSGluing}, and the Carlotto--Schoen gluing in asymptotic cones \cite{CarlottoSchoenData} which produces asymptotically flat initial data containing any finite number of non-trivial conic regions which do not interact for any desired amount of time. We also mention the problem of filling in a given asymptotically flat data set in a controlled singularity-free manner, as studied in \cite[\S5]{BieriChruscielBondi}. In more recent work, Aretakis--Czimek--Rodnianski \cite{AretakisCzimekRodnianskiGluing,AretakisCzimekRodnianskiGluingAnalysis,AretakisCzimekRodnianskiGluingSpacelike} developed a gluing procedure for the \emph{characteristic} initial value problem, in which case the constraint equations become a coupled system of transport equations along the null generators of the incoming and outgoing null cones emanating from a spacelike 2-sphere, see e.g.\ \cite[\S2.3]{LukCharacteristic}. Applications include a sharp (as far as decay is concerned) improvement of the Carlotto--Schoen result as well as an alternative proof of the Corvino--Schoen gluing results. For non-vacuum initial data sets, there are further types of gluing problems one may consider; we mention in particular the results by Corvino--Huang \cite{CorvinoHuangDEC} on promoting the dominant energy condition to a strict inequality (again under suitable genericity conditions).

Besides the problem of gluing given or known initial data sets, one may wish to construct initial data sets \textit{ab initio}. This is typically done via variants of the \emph{conformal method} \cite{LichnerowiczConstraints,YorkConstraints} in which $\gamma,k$ are expressed in terms of suitable seed data and a conformal factor which satisfies a semilinear elliptic PDE. We refer the reader to Carlotto's excellent review article \cite{CarlottoConstraints} for a detailed discussion and further references.

\subsection{Gluing and asymptotic expansions}
\label{SsIG}

We proceed to set up a more precise version of Theorem~\ref{ThmI}. To begin, we say that a function $u$ on $\R^n$ is \emph{conormal} with weight $\alpha\in\R$ if
\[
  |(\la r\ra\nabla_x)^j u(x)|\lesssim\la r\ra^{-\alpha}
\]
for all $j\in\N_0$, where $r=|x|$.\footnote{Later on, we use the notation $\cA^\alpha(\ol{\R^n})$ for the space of such functions.} (In other words, every derivative in $x$ gains a power of $\la r\ra^{-1}$.) \emph{Polyhomogeneity} of $u$ is the stronger requirement that in $|x|>1$, we can write
\begin{equation}
\label{EqIGPhg}
  u(x)\sim\sum r^{-z_\ell}(\log r)^{k_\ell} u_\ell\Bigl(\frac{x}{|x|}\Bigr),
\end{equation}
where $z_\ell\in\C$, $k_\ell\in\N_0$, $u_\ell\in\CI(\Sph^{n-1})$, with $\Re z_\ell\to\infty$; the symbol `$\sim$' means that for any $N$, the difference of $u$ and the truncation of the sum to all $\ell$ with $\Re z_\ell\leq N$ is conormal with weight $N$. Thus, polyhomogeneous functions have generalized Taylor expansions at the boundary at infinity\footnote{We make this more precise using the radial compactification $\ol{\R^n}$ of $\R^n$ in~\S\ref{SsBb}.} $r^{-1}=0$ of $\R^n$.

We then call a pair $(\hat\gamma,\hat k)$ of smooth symmetric 2-tensors on $\R^n\setminus\hat K^\circ$ (with $\hat K\Subset\R^n$) \emph{$\delta$-asymptotically flat} if the following holds: writing coordinates on $\R^n$ as $\hat x=(\hat x^1,\ldots,\hat x^n)$, the coefficients $\hat\gamma_{\hat i\hat j}=\hat\gamma(\pa_{\hat x^i},\pa_{\hat x^j})$ and $\hat k_{\hat i\hat j}$ are polyhomogeneous conormal at infinity, with $\hat\gamma_{\hat i\hat j}-\delta_{i j}$ conormal with weight\footnote{That is, the exponents $z_\ell$ in the expansion~\eqref{EqIGPhg} for $\hat\gamma_{\hat i\hat j}$ satisfy $\Re z_\ell\geq\delta$ (with strict inequality if $k_\ell\geq 1$), except for a single term $(z_0,k_0)=(0,0)$ with coefficient $u_0=\delta_{i j}$. The definition used in the main part of the paper records the set $\{(z_\ell,k_\ell)\}$, see Definition~\ref{DefCEAf}.} $\delta>0$ and $\hat k$ conormal with weight $1+\delta$.

Following a standard procedure in geometric singular analysis (see e.g.\ \cite{MazzeoMelroseAdiabatic}), we construct the glued metrics $(\gamma_\eps,k_\eps)$ for all sufficiently small $\eps>0$ in one fell swoop by working on an appropriate total space $\wt X\setminus\wt K^\circ$, of dimension $n+1$, which is a manifold with corners; here $\wt X$ is a resolution of $[0,\eps_\sharp)\times X$ at $\{0\}\times\{\fp\}$, and $\wt K=\bigsqcup\{\eps\}\times\eps\hat K$ in geodesic normal coordinates on $X$ around $\fp$. See~\S\ref{SG} for details. We only note here that the space $\wt X$ is equipped with a map down to $[0,\eps_\sharp)$ which is a fibration over $(0,\eps_\sharp)$, and the fibers of $\wt X\setminus\wt K^\circ\to(0,\eps_\sharp)$ are the manifolds $X_\eps$ in Theorem~\ref{ThmI}. The fibers become singular as $\eps\searrow 0$ however, and the fiber over $\eps=0$ is the disjoint union of two manifolds: one is the compactification at infinity of $\R^n\setminus\hat K^\circ$ which carries the asymptotically flat data set $(\hat\gamma,\hat k)$, and the other is a compactification of $X\setminus\{\fp\}$ (namely, the real blow-up of $X$ at $\{\fp\}$) which carries the original data set $(\gamma,k)$. Local coordinates on this space are illustrated in Figure~\ref{FigIG}. In particular, the transitional region $\eps\lesssim|x|\lesssim 1$ omitted in the statement of Theorem~\ref{ThmI} is simply a neighborhood of the codimension $2$ corner of $\wt X$.

\begin{figure}[!ht]
\centering
\includegraphics{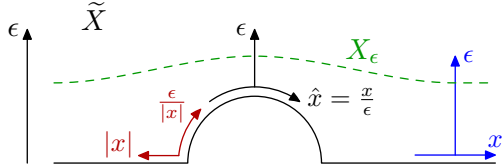}
\caption{Illustration of the total space $\wt X$ and of three local coordinate systems on this manifold with corners. Also shown is a fiber $X_\eps$ of $\wt X$ over $\eps>0$ (with $\hat K=\emptyset$).}
\label{FigIG}
\end{figure}

We may now state a more precise version of Theorem~\ref{ThmI}, in which we construct a single pair $(\wt\gamma,\wt k)$ of sections of the bundle $S^2\wt T^*\wt X$ which is defined as the pullback of the second symmetric power of the cotangent bundle of $X$ to $\wt X$:

\begin{thm}[Main result, more precise version]
\label{ThmIG}
  Let $\hat\gamma$ and $\hat k$ be $\delta$-asymptotically flat on $\R^n\setminus\hat K^\circ$ for some $\delta>0$. Suppose that the initial data set $(X,\gamma,k)$ does not admit any KIDs in a smoothly bounded connected open neighborhood $\cU^\circ\subset X$ of $\fp\in X$. Denote by $x\in\R^n$ geodesic normal coordinates around $\fp\in X$, and write $\hat x=\frac{x}{\eps}$. Then there exist polyhomogeneous sections $\wt\gamma,\wt k$ of $S^2\wt T^*\wt X$ with the following properties.
  \begin{enumerate}
  \item the restriction $(\gamma_\eps,k_\eps)$ of $(\wt\gamma,\wt k)$ to each fiber $X_\eps$ of $\wt X\setminus\wt K^\circ\to(0,\eps_\sharp)$ satisfies the constraint equations~\eqref{EqICE} with the same cosmological constant as $(\gamma,k)$.
  \item On $[0,\eps_\sharp)\times(X\setminus\cU^\circ)$, we have $(\wt\gamma,\wt k)=(\gamma,k)$.
  \item The leading order terms of $(\wt\gamma,\wt k)$ at the $X\setminus\{\fp\}$ component of the singular fiber over $\eps=0$ are given by $(\gamma,k)$; the leading order terms of $(\wt\gamma_{i j},\eps\wt k_{i j})$ at the $\R^n\setminus\hat K^\circ$ component are $(\hat\gamma_{\hat i\hat j},\hat k_{\hat i\hat j})$. Near the codimension $2$ corner of $\wt X$, and writing $\omega=\frac{x}{|x|}$, this means the following: the tensors
    \[
      \wt\gamma = \wt\gamma_{i j}\Bigl(|x|,\frac{\eps}{|x|},\omega\Bigr)\dd x^i\,\dd x^j,\quad
      \wt k = \wt k_{i j}\Bigl(|x|,\frac{\eps}{|x|},\omega\Bigr)\dd x^i\,\dd x^j,
    \]
    which are jointly polyhomogeneous in $|x|\in[0,1)$ and $\frac{\eps}{|x|}\in[0,1)$, have the following leading order behavior at $\frac{\eps}{|x|}=0$, resp.\ $|x|=0$:
    \begin{equation}
    \label{EqIGBdyData}
    \begin{alignedat}{2}
      \wt\gamma_{i j}(|x|,0,\omega)&=\gamma_{i j}(|x|\omega),&\qquad
      \wt\gamma_{i j}\Bigl(0,\frac{1}{|\hat x|},\omega\Bigr)&=\hat\gamma_{\hat i\hat j}(|\hat x|\omega), \\
      \wt k_{i j}(|x|,0,\omega)&=k_{i j}(|x|\omega), &\qquad
      \eps\wt k_{i j}\Bigl(0,\frac{1}{|\hat x|},\omega\Bigr)&=\hat k_{\hat i\hat j}(|\hat x|\omega).
    \end{alignedat}
    \end{equation}
  \end{enumerate}
\end{thm}

See Theorem~\ref{ThmGlT} for the full statement. We stress that the glued initial data $(\wt\gamma,\wt k)$ have exactly two singular limits, namely the two given data sets $(\hat\gamma,\hat k)$ and $(\gamma,k)$; we do not need to introduce any sort of intermediate gluing region.

\begin{rmk}[Limited regularity]
\label{RmkIGLimited}
  One can construct the families $(\wt\gamma,\wt k)$, and thus the glued data $(\gamma_\eps,k_\eps)$, under weaker assumptions on regularity and asymptotics. Note first that for $\hat\rho=(\eps^2+|x|^2)^{1/2}$, the rescaling $\hat\rho\wt k_{i j}$ can be restricted to $\frac{\eps}{|x|}=0$ and $|x|=0$, with leading order terms $|x|k_{i j}$ and $\la\hat x\ra\hat k_{\hat i\hat j}$, respectively. When $(\hat\gamma,\hat k)$ is merely conormal, then one can construct a conormal solution $(\wt\gamma,\wt k)$ of the gluing problem so that $(\wt\gamma,\hat\rho\wt k)$ has boundary data as in~\eqref{EqIGBdyData}. Less restrictively still, it suffices to assume that $\hat\gamma$ and $\hat k$ have weighted H\"older regularity of degree $(s,\alpha)$ with $\N_0\ni s\geq 2$, $\alpha\in(0,1)$, with weights $\delta$ and $1+\delta$ as above (see also \cite[Appendix~A]{ChruscielDelayMapping}); the solution $(\wt\gamma,\hat\rho\wt k)$, with boundary data~\eqref{EqIGBdyData}, given by our methods then has matching regularity on an appropriate scale of weighted H\"older spaces. One can similarly weaken the regularity requirements on $\gamma,k$. We shall omit a detailed treatment of such finite regularity versions of our main results in this paper.
\end{rmk}

The construction of $(\wt\gamma,\wt k)$ proceeds in three steps.
\begin{enumerate}
\item\label{ItIG1} The first step is to define sections $(\wt\gamma_0,\wt k_0)$ on the total space which have the desired boundary data as in~\eqref{EqIGBdyData}; such sections already solve the constraint equations (with cosmological constant $\Lambda$) to leading order at both components of the singular fiber at $\eps=0$. We note here that at the $\R^n\setminus\hat K^\circ$ component, the appropriately rescaled cosmological constant vanishes to leading order (see Lemma~\ref{LemmaCETotScale} and Proposition~\ref{PropCETot}\eqref{ItCETotPhg}).
\item\label{ItIG2} In a second step,\footnote{If one is satisfied with less than full polyhomogeneity of $(\wt\gamma,\wt k)$, one skip all, or part, of this step.} one corrects the lower order terms of $(\wt\gamma_0,\wt k_0)$ in turns at the $\R^n\setminus\hat K^\circ$, resp.\ the $X\setminus\{\fp\}$ components of the singular fiber (denoted $\hat X\setminus\hat K^\circ$, resp.\ $X_\circ$ in the bulk of the paper); this requires right-inverting the linearized constraints maps on the asymptotically flat space $(\R^n\setminus\hat K^\circ,\hat\gamma,\hat k)$, resp.\ on the punctured space $(X\setminus\{\fp\},\gamma,k)$ with suitable control of the solution at infinity, resp.\ at $\fp$. This is accomplished in~\S\S\ref{SsCEAf} and \ref{SsCEPct}. The genericity condition on $(X,\gamma,k)$ is used to effect the right inversion on the punctured space with control on supports; this uses that on appropriate function spaces on the punctured space the triviality of the cokernel of the linearized constraints map can be inferred from the triviality on the original space $X$. We do not require any genericity for $(\R^n\setminus\hat K^\circ,\hat\gamma,\hat k)$, as the relevant cokernel is trivial if we permit the solution of the linearized constraints map to have less decay than $|\hat x|^{-n+2}$. (See Proposition~\ref{PropCEAfSolv}.)

  Using an asymptotic summation argument, we can then construct a formal solution $(\wt\gamma',\wt k')$ of the gluing problem: it satisfies the constraints modulo rapidly vanishing errors as $\eps\searrow 0$. Steps~\eqref{ItIG1} and \eqref{ItIG2} are completed in~\S\ref{SsGlFo} (the formal solution being denoted $(\wt\gamma,\wt k)$ there).
\item\label{ItIG3} Finally, we correct the formal solution to a true solution $(\wt\gamma,\wt k)$ by solving the nonlinear constraint equations with rapidly vanishing (as $\eps\searrow 0$) right hand side. This is done using a standard contraction mapping argument on a nonstandard scale of Sobolev spaces with $\eps$-dependent norms. The correction term vanishes to infinite order at $\eps=0$ as well. See~\S\ref{SsGlT} (the true solution being denoted $(\wt\gamma',\wt k')$ there).
\end{enumerate}

\begin{rmk}[Difficulties of a bare interpolation approach]
\label{RmkIGCorvino}
  We contrast the above construction with a more traditional approach as used e.g.\ in \cite{CorvinoScalar,ChruscielDelayMapping}. For brevity, we only consider the time-symmetric case, so $\hat k=0$ and $k=0$. We use geodesic normal coordinates $x\in\R^n$ on $X$ near $\fp$. For small parameters $\eps,\eta>0$, we define a bare interpolation of $\hat\gamma$ and $\gamma$ by
  \begin{equation}
  \label{ItIGCorvino}
    \gamma^{(\eps,\eta)}_{i j}(x) := \chi\Bigl(\frac{|x|}{\eta}\Bigr)\hat\gamma_{\hat i\hat j}\Bigl(\frac{x}{\eps}\Bigr) + \Bigl(1-\chi\Bigl(\frac{|x|}{\eta}\Bigr)\Bigr)\gamma_{i j}(x),
  \end{equation}
  where the non-negative function $\chi\in\CIc([0,2))$ is identically $1$ on $[0,1]$. The metric $\gamma^{(\eps,\eta)}$ violates the (time-symmetric) constraint equations only in the gluing region $U_\eta=\{\eta\leq|x|\leq 2\eta\}$. Taking $\eta$ small and then $\eps\ll\eta$ ensures that $\gamma^{(\eps,\eta)}$ is Riemannian (and includes a large piece of the asymptotically flat metric $\hat\gamma$); moreover, in $U_\eta$ we have $\hat\gamma_{\hat i\hat j}(\frac{x}{\eps}),\gamma_{i j}(x)\approx\delta_{i j}$ and thus also $\gamma_{i j}^{(\eps,\eta)}(x)\approx\delta_{i j}$. However, the violation of the constraint equations in $U_\eta$ is typically of size $\cO(1+\eta^{-1}(\frac{\eps}{\eta})^{1+\delta})$,\footnote{Indeed, the scalar curvature of $\gamma^{(\eps,\eta)}$ is schematically of the form $\sum_{j=1}^2\pa^j(\chi(\frac{|x|}{\eta}))\pa^{2-j}(\gamma(x)-\hat\gamma(\frac{x}{\eps}))$, and thus---using that $\pa^k(\gamma_{i j}(x)-\delta_{i j})=\cO(|x|^{2-k})$ and $\pa^k(\hat\gamma_{i j}(\frac{x}{\eps})-\delta_{i j})\lesssim|\frac{x}{\eps}|^{-\delta-k}$ for $k=0,1,2$---of size $\sum_{j=1}^2\eta^{-j}(\eta^{2-(2-j)}+(\frac{\eta}{\eps})^{-\delta-(2-j)})\simeq 1+\eps^{1+\delta}\eta^{-2-\delta}$ for $|x|\simeq\eta$ indeed.} and therefore is not small; hence, it is not clear how one can re-impose the constraints in $U_\eta$ with perturbative methods, regardless of the smallness of $\eps\ll\eta\ll 1$. This discussion serves to indicate that the deformation of $\gamma^{(\eps,\eta)}$ is a delicate task.\footnote{Passing to the rescaled coordinates $\tilde x=\frac{x}{\eta}$ (and regarding $\gamma_{i j}(\eta\tilde x)$ and $\hat\gamma_{\hat i\hat j}(\eta\eps^{-1}\tilde x)$ as the coefficients of a metric in the frame $\dd\tilde x^1,\ldots,\dd\tilde x^n$) does not eliminate this issue. The violation of the constraints in the gluing region $\tilde U=\{1\leq|\tilde x|\leq 2\}$ is now of size $\cO(\eta^2+(\frac{\eps}{\eta})^\delta)$, which is small for $\eps\ll\eta\ll 1$. However, the metric $\gamma^{(\eps,\eta)}(\eta\tilde x)$ converges to the Euclidean metric in $\tilde U$ as $\eta\to 0$, and hence one faces an approximate cokernel of the linearized constraints map when attempting to deform $\gamma^{(\eps,\eta)}$ to a solution of the constraints (cf.\ \cite{CorvinoScalar}). If one relaxed the localization requirements of the deformation and attempted to take advantage of the local genericity of $X$ in a fixed neighborhood of $\fp$, one would have to work on a domain with diameter $\simeq\eta^{-1}\to\infty$ in $\tilde x$-coordinates, which would lead to different analytic subtleties.} But even if one can correct $\gamma^{(\eps,\eta)}$ in $U_\delta$ to a solution of the constraint equations in some way, the glued data set would transition from $\gamma$ to the rescaling of $\hat\gamma$ in the region $|x|\simeq\eta$, with $\eps\ll\eta\ll 1$. Thus, one would produce an intermediate gluing region $\{\eps\ll|x|\ll 1\}$, which asymptotes to the codimension $2$ corner of $\wt X$. --- By contrast to~\eqref{ItIGCorvino}, the interpolation $\wt\gamma_0$ of $\hat\gamma$ and $\gamma$ which is the starting point of our construction (see Step~\eqref{ItIG1} above) can be defined as
  \begin{equation}
  \label{ItIGCorvinoNo}
    (\wt\gamma_0^{(\eps)})_{i j} = \gamma_{i j}(x)-\delta_{i j}+\hat\gamma_{\hat i\hat j}\Bigl(\frac{x}{\eps}\Bigr)
  \end{equation}
  on the $\eps$-level set of $\wt X$ near $x=0$ (while globally on $X$, one may take $\sum_{i,j}\chi(|x|)\cdot\eqref{ItIGCorvinoNo}\,\dd x^i\,\dd x^j+(1-\chi(|x|))\gamma$). Thus, there is only a single gluing region, which is of the rough form $\eps\lesssim|x|\lesssim 1$; in particular, the lower gluing radius shrinks to $0$. Moreover, carefully note that the modification of $\wt\gamma_0$ to a solution of the constraint equations uses deformations which are global in a fixed punctured neighborhood of $\fp$ in $X\setminus\{\fp\}$ (i.e.\ all the way down to $\fp$) and global on $\R^n_{\hat x}$ near $|\hat x|=\infty$ (i.e.\ all the way up to infinity).
\end{rmk}

\begin{rmk}[Approximate solutions]
\label{RmkIGApprox}
  The construction of $(\wt\gamma,\wt k)$ in (generalized) Taylor series at the two boundary hypersurfaces at $\eps=0$ in Step~\eqref{ItIG2} produces increasingly accurate approximate solutions $(\gamma_\eps,k_\eps)$ for small $\eps$; thus, one may be able to use part of the above procedure to produce rather accurate initial data for use in numerical evolutions.
\end{rmk}

The right inversion of the linear operators in Steps~\eqref{ItIG2} and \eqref{ItIG3} is accomplished much as in the original works \cite{CorvinoScalar,CorvinoSchoenAsymptotics,ChruscielDelayMapping} using exponential weights at the boundary of the gluing region. Our discussion of the coercivity of the adjoint of the linearized constraints map in the asymptotically flat setting takes inspiration from \cite{DelayCompact} in that we combine standard elliptic theory with a prolongation argument (Lemma~\ref{LemmaCEProlong}), cf.\ the `Kernel Restriction Condition' and Lemma~8.2 in \cite{DelayCompact}.

We briefly comment on the choice of function spaces. In the asymptotically flat setting, one measures regularity with respect to the weighted operator $\la\hat x\ra\hat\nabla$ (differentiation in the $\hat x$-coordinates); this is also called \emph{b-regularity} on $\ol{\R^n}$ in the language of \cite{MelroseAPS}. Similarly, near the puncture of $X$ at $\fp$, one needs to measure regularity with respect to the weighted operator $|x|\nabla$ (which is the same as b-regularity on the blow-up $[X;\{\fp\}]$). Indeed, this is the notion of regularity which is compatible with polynomial weights or polyhomogeneous expansions at $|x|=0$; moreover, it arises naturally by computing the form of $\la\hat x\ra\hat\nabla$ near the asymptotically flat end in $\hat x$-coordinates in the local coordinates $x$; see~\eqref{EqIGNabla} below. On the total space $\wt X$, these two notions can be merged into the notion of \emph{q-regularity}, introduced by the author in a geometrically related context in \cite{HintzKdSMS}: this amounts to measuring regularity with respect to
\begin{equation}
\label{EqIGNabla}
  \hat\rho\nabla := (\eps^2+|x|^2)^{1/2}\nabla = \la|\hat x|^{-1}\ra|x|\nabla = \la\hat x\ra\hat\nabla\qquad \Bigl(\hat\rho=(\eps^2+|x|^2)^{1/2},\ \hat x=\frac{x}{\eps}\Bigr).
\end{equation}
The corresponding q-Sobolev spaces are standard Sobolev spaces on $X_\eps$ as sets, but their norms depend on $\eps$. The contraction mapping argument takes place on tensors on $\wt X$ measured in such q-Sobolev spaces\footnote{H\"older spaces would work just as well; we use Sobolev spaces for convenience only.} (and thus the corrections to the formal solution are constructed for all small $\eps>0$ simultaneously).

Melrose and Singer, in unpublished work \cite{MelroseSingerKahlerGluing}, have developed a singular-geometric and analytic point of view for the problem of gluing constant scalar curvature K\"ahler (cscK) metrics into neighborhoods of points of a given compact cscK manifold which is closely related to the point of view adopted in the present paper. The gluing constructions for magnetic monopoles by Kottke--Singer \cite{KottkeSingerMonopoles} and for gravitational instantons by Schroers--Singer \cite{SchroersSingerInstantons} are also closely related, though either the total space or the relevant rescaled tangent bundles differ from (and are in fact more delicate than) those used in the present work.

\subsection{Exponential weights and geometric singular analysis}
\label{SsI00}

We briefly comment on a novel technical aspect of our analysis, which is due to Mazzeo \cite{Mazzeo00}. Our discussion here shall take place in the half space $[0,\infty)_x\times\R^{n-1}_y\subset\R^n_{(x,y)}$. If one wishes to construct a smooth solution $u$ of some underdetermined PDE $P u=f$ (with $x>0$ on $\supp f$) with support in $x\geq 0$, one can use coercivity estimates for the overdetermined operator $P^*$ on spaces with exponential weights $e^{\beta/x}$, $\beta>0$ (as done in all of the aforementioned gluing papers), or equivalently for $e^{-\beta/x}P^* e^{\beta/x}$ on unweighted spaces. Consider, for example, the case $P=\delta$, the (negative) divergence on 1-forms. Then $P^*=\dd=(\pa_x,\pa_{y^1},\ldots,\pa_{y^{n-1}})$ in the basis $\dd x,\dd y^1,\ldots,\dd y^{n-1}$; this is left elliptic (and indeed overdetermined). Conjugating $P^*$ by $e^{-\beta/x}$ produces singular terms of size $\cO(x^{-2})$. We thus consider the rescaling $x^2 P^*$ as a left elliptic operator built from the 00-vector fields (`double 0-vector fields') $x^2\pa_x$ and $x^2\pa_{y^j}$; conjugation of such a \emph{00-differential operator} by $e^{\beta/x}$ merely produces smooth lower order terms.

A 00-differential operator has a principal symbol which is an extension of the usual principal symbol map down to $x=0$; this symbol takes the $x^2$-degeneration into account by passing to an appropriate rescaling of the cotangent bundle. That is, the principal symbol is a fiberwise polynomial in the momentum variables $\xi,\eta$ defined by writing 1-forms as $\xi\frac{\dd x}{x^2}+\eta\frac{\dd y}{x^2}$.

The crucial observation is that 00-differential operators have yet another \emph{commutative} symbol which captures their leading order behavior at any boundary point $(0,y)\in[0,\infty)\times\R^{n-1}$; for example, the boundary principal symbol of $e^{-\beta/x}P^* e^{\beta/x}=(x^2\pa_x-\beta,\pa_{y^1},\ldots,\pa_{y^{n-1}})$ is $(i\xi-\beta,i\eta)$, and this is well-defined for \emph{all} $(\xi,\eta)\in\R\times\R^{n-1}$. The existence of this second commutative symbol is structurally due to the fact that 00-vector fields commute to leading order at $x=0$.

The left ellipticity of both symbols then suffices to prove a semi-Fredholm estimate for $e^{-\beta/x}P^* e^{\beta/x}$ on a scale of 00-Sobolev spaces. Complemented with a direct analysis of the kernel, one obtains the solvability of $P u=f$ on 00-Sobolev spaces, with $u$ having the rapidly decaying weight $e^{-\beta/x}$ near $x=0$. The extension of $u$ by $0$ to $x<0$ furnishes a solution with the desired support property. See~\S\S\ref{SsB00} and \ref{SsBAn} for a detailed discussion, and Examples~\ref{ExB00Ext} and \ref{ExBAnExt} for the case of the exterior derivative mentioned here.

In~\S\ref{SsBAn}, we explain the relationship of the part of 00-analysis which only uses the standard principal symbol to bounded geometry analysis as discussed in \cite{ShubinBounded}; this is also related to the \emph{scaling property} of \cite[Appendix~B]{ChruscielDelayMapping}. The existence and importance of the second, boundary, principal symbol on the other hand appears to be noted here for the first time.

\subsection{Outlook: evolution of glued initial data}
\label{SsIO}

We make the conjecture in~\S\ref{SsIC} about the evolution of glued initial data more precise:

\begin{conj*}
  Let $(\hat\gamma,\hat k)$ denote boosted subextremal Kerr black hole initial data with parameters $\hat\bhm$ (mass) and $\hat\bha$ (specific angular momentum). If $(M,g)$ is a compact globally hyperbolic subset of the maximal globally hyperbolic development of $(X,\gamma,k)$, then the maximal globally hyperbolic development of $(X_\eps,\gamma_\eps,k_\eps)$, for a `suitable' family $(\gamma_\eps,k_\eps)$ of initial data with boundary data $(\gamma,k)$ and $(\hat\gamma,\hat k)$, contains a region $(M_\eps,g_\eps)$ with the following properties:
  \begin{enumerate}
  \item $M_\eps$ is obtained from $M$ by excising a size $\eps$-neighborhood of the geodesic $\cC\subset M$ whose initial conditions at $X\subset M$ are determined by the point $\fp\in X$ and the boost parameter of the Kerr data set;
  \item $g_\eps$ tends to $g$ away from $\cC$, while in an $\cO(\eps)$-neighborhood of $\cC$, the rescaling $\eps^{-2}g_\eps$ tends to a family (depending on the point in $\cC$) of Kerr black hole metrics with parameters $(\hat\bhm,\hat\bha)$.
  \item The total family $\eps\mapsto g_\eps$ is polyhomogeneous on a total space obtained by resolving $[0,1)\times M$ at $\{0\}\times\cC$.
  \end{enumerate}
\end{conj*}

In particular, when $(M,g)$ is a unit mass Kerr--de~Sitter black hole, and the gluing data are chosen so that $\cC$ starts in the exterior region but crosses the event horizon in finite proper time, then $(M_\eps,g_\eps)$ describes the merger of a mass $\eps$ black hole with a unit mass black hole. (Moreover, once the black holes have merged, one may restrict to a suitable neighborhood of the Kerr--de~Sitter exterior region and apply \cite{HintzVasyKdSStability} to conclude that the merged black hole settles down to a Kerr--de~Sitter black hole.) We hope to address this conjecture in future work. At this point, we merely remark that the conjecture in the stated form cannot hold without restricting to `suitable' initial data; roughly speaking, the initial data must be such that the evolving spacetime metric is adiabatic, i.e.\ does not vary much on the fast time scale $\hat t$ of the small black hole.

\subsection*{Outline of the paper}

In~\S\ref{SB}, we recall the notions of geometric singular analysis which are relevant for the present paper, namely b-vector fields and operators (and their normal operators), scattering vector fields, and real blow-ups; we further discuss 00-geometry and the associated scales of Sobolev spaces, and develop elliptic theory in this context. In~\S\ref{SG}, we define the manifold with corners $\wt X$ on which the gluing procedure will take place, and we describe the class of q-differential operators on $\wt X$ which will precisely capture the behavior of the (linearized) constraints map.

In~\S\ref{SCE}, we study the (linearized) constraints map in detail. Estimates and solvability results for the (linearized) constraints map on asymptotically flat and punctured manifolds are combined into a uniform solvability theory on $\wt X$. In~\S\ref{SGl}, we apply these results to construct the glued initial data set first on the level of (generalized) Taylor series; we then correct this formal solution to a true solution using a suitable contraction mapping principle.

Applications and variants of the main result are discussed in~\S\ref{SA}.

\subsection*{Acknowledgments}

I would like to thank Richard Melrose and Andr\'as Vasy for a number of conversations about black hole gluing related to the present work. I am also grateful to Rafe Mazzeo for suggesting the use of his 00-calculus \cite{Mazzeo00} to facilitate a geometric singular analysis style treatment of localized gluing. Many thanks are also due to Alessandro Carlotto and Piotr Chru\'sciel for a number of helpful comments and suggestions on content and exposition, and to Michael Singer for important pointers to existing literature.

\section{Geometric and analytic background}
\label{SB}

\subsection{b- and scattering structures; blow-ups}
\label{SsBb}

For manifolds with boundary or with corners, we demand that all boundary hypersurfaces be embedded. A defining function of a boundary hypersurface $H\subset X$ of a manifold with corners is a function $\rho\in\CI(X)$ for which $H=\rho^{-1}(0)$, further $\rho>0$ on $X\setminus H$, and finally $\dd\rho\neq 0$ along $H$. In the case that $X$ is a manifold with boundary $H=\pa X$, we called such a function $\rho$ simply a \emph{boundary defining function}.

The radial compactification
\[
  \ol{\R^n} := \Bigl(\R^n \sqcup \bigl( [0,\infty)_\rho\times\Sph^{n-1}_\omega \bigr) \Bigr) / \sim,
\]
where we identify $0\neq x=r\omega\in\R^n$ with $(\rho,\omega)=(r^{-1},\omega)$, is a manifold with boundary `at infinity' given by $\pa\ol{\R^n}\cong\Sph^{n-1}$; the interior is $\R^n$. Invertible linear maps on $\R^n$ extend by continuity to diffeomorphisms of $\ol{\R^n}$; thus, the radial compactification of a finite-dimensional real vector space is well-defined.

On a manifold with corners $X$, we write $\cV(X)$ for the Lie algebra of all smooth vector fields (sections of $T X\to X$), and $\Vb(X)\subset\cV(X)$ for the Lie algebra of b-vector fields \cite{MelroseAPS}, i.e.\ vector fields which are tangent to $\pa X$. In local coordinates $x^1,\ldots,x^k\geq 0$, $x^{k+1},\ldots,x^n\in\R$, near a codimension $k$ corner of $X$, the space $\Vb(X)$ is spanned over $\CI(X)$ by the vector fields $x^i\pa_{x^i}$ ($i=1,\ldots,k$), $\pa_{x^j}$ ($j=k+1,\ldots,n$), and therefore these vector fields are a frame of the \emph{b-tangent bundle} $\Tb X\to X$. When $X$ is a manifold with boundary and boundary defining function $\rho\in\CI(X)$, then $\Vsc(X):=\rho\Vb(X)$ is the Lie algebra of \emph{scattering vector fields} \cite{MelroseEuclideanSpectralTheory}; in local coordinates $x\geq 0$, $y=(y^1,\ldots,y^{n-1})\in\R^{n-1}$, this space is spanned over $\CI(X)$ by the vector fields $x^2\pa_x$ and $x\pa_{y^j}$ ($j=1,\ldots,n-1$), which are a frame of the \emph{scattering tangent bundle} $\Tsc X\to X$. The dual bundles $\Tb^*X\to X$ and $\Tsc^*X\to X$ have frames $\frac{\dd x^i}{x^i}$ ($i=1,\ldots,k$), $\dd x^j$ ($j=k+1,\ldots,n$) and $\frac{\dd x}{x^2}$, $\frac{\dd y^j}{x}$ ($j=1,\ldots,n-1$), respectively. In the special case $X=\ol{\R^n}$, a change of coordinates calculation shows that $\CI(\ol{\R^n})=S^0_{\rm cl}(\R^n)$ consists of classical symbols on $\R^n$ (i.e.\ smooth functions which are in addition smooth in $(r^{-1},\omega)$), and $\Vsc(\ol{\R^n})$ is spanned over $\CI(\ol{\R^n})$ by the (translation-invariant) coordinate vector fields $\pa_{x^i}$ ($i=1,\ldots,n$). In particular, the Euclidean metric on $\R^n$ is a smooth Riemannian scattering metric, i.e.\ a positive definite section of $S^2\,\Tsc^*\ol{\R^n}$.

When $X$ is a manifold with corners, we write
\[
  \Diffb^k(X)
\]
for the space of all $k$-th order b-differential operators, i.e.\ locally finite sums of up to $k$-fold compositions of elements of $\Vb(X)$. If $H_1,\ldots,H_N$ is a list of the boundary hypersurfaces of $X$, and if $\rho_j\in\CI(X)$ denotes a defining function of $H_j$ for $j=1,\ldots,N$, then we furthermore write
\[
  \rho_1^{\alpha_1}\cdots\rho_N^{\alpha_N}\Diffb^k(X) = \{ \rho_1^{\alpha_1}\cdots\rho_N^{\alpha_N}P \colon P\in\Diffb^k(X) \}
\]
for the space of weighted b-differential operators; here $\alpha_j\in\R$, $j=1,\ldots,N$. We note that conjugation by $\rho_j^{\alpha_j}$ preserves the space $\Diffb^k(X)$ and its weighted analogues. These spaces are $\CI(X)$-modules; spaces $\Diffb^k(X;E,F)$ of operators acting between sections of vector bundles $E,F\to X$ can thus be defined in the usual manner, likewise for weighted operators.

The b-principal symbol of $A\in\Diffb^k(X)$ is an element $\sigmab^k(A)\in P^k_{\rm hom}(\Tb^*X)$, the space of smooth functions on $\Tb^*X$ which are fiberwise homogeneous polynomials of degree $k$; it is multiplicative, given via pullback along the base projection for $k=0$, and given by $\sigmab^1(V)(z,\zeta)=i \zeta(V|_z)$ for $V\in\Vb(X)$. We make this explicit in the case that $X$ is a manifold with boundary: in local coordinates $x\geq 0$, $y\in\R^{n-1}$ as above, and writing b-covectors as $\xi\frac{\dd x}{x}+\eta\,\dd y$,
\[
  \sigmab^k(A)(x,y,\xi,\eta) \colon A=\sum_{j+|\alpha|\leq k} a_{j\alpha}(x,y)(x D_x)^j D_y^\alpha \mapsto \sum_{j+|\alpha|=k} a_{j\alpha}(x,y)\xi^j\eta^\alpha.
\]
The principal symbol captures $A\in\Diffb^k(X)$ to leading order in the differential order sense, in that $\sigmab^k(A)=0$ implies $A\in\Diffb^{k-1}(X)$. In order to capture $A$ to leading order at $\pa X$, one fixes a collar neighborhood $[0,x_0)_x\times\pa X$ of $\pa X$, and defines, in local coordinates as above, the \emph{b-normal operator} of $A$ by
\[
  N(A) := \sum_{j+|\alpha|\leq k} a_{j\alpha}(0,y)(x D_x)^j D_y^\alpha \in \Diffb^k([0,\infty)_x\times\pa X).
\]
This operator is dilation-invariant in $x$, and it differs from $A$ on $[0,x_0)\times\pa X$ by an element of $x\Diffb^k$. Formally conjugating $N(A)$ by the Mellin transform in $x$ (see also~\S\ref{SsBFn}) gives the \emph{Mellin transformed normal operator family}
\begin{equation}
\label{EqBNormMT}
  N(A,\lambda) := \sum_{j+|\alpha|\leq k} a_{j\alpha}(0,y)\lambda^j D_y^\alpha \in \Diff^k(\pa X).
\end{equation}

We next recall the notion of (real) blow-up; see \cite{MelroseDiffOnMwc} for a detailed account. We are given a manifold with corners $X$ and a p-submanifold $S\subset X$, i.e.\ a submanifold so that near all $p\in S$ there exist local coordinates $x^1,\ldots,x^k\geq 0$, $x^{k+1},\ldots,x^n\in\R$ on $X$ so that $S$ is given by the vanishing of a subset of these coordinates. (If this subset always contains at least one of the $x^1,\ldots,x^k$, we call $S$ a boundary p-submanifold.) The blow-up of $X$ along $S$ is defined as
\[
  [X;S] := (X\setminus S) \sqcup S N^+S,
\]
where $S N^+S=(N^+S\setminus o)/\R^+$ is the spherical inward pointing normal bundle of $S$; here, we set $N^+S=T^+_S X/T S$, with $T^+_p X\subset T_p X$ consisting of all (non-strictly) inward pointing tangent vectors on $X$ at $p$. The space $[X;S]$ is equipped with a blow-down map $\upbeta\colon[X;S]\to X$, which is the identity on $X\setminus S$ and the base projection on the \emph{front face} $S N^+S$. The space $[X;S]$ can be given the structure of a smooth manifold by declaring polar coordinates around $S$ to be smooth down to the origin. To make this concrete, say $S$ is given by $x^1=\ldots=x^p=0$, $x^{k+1}=\ldots=x^{k+q}=0$ for some $0\leq p\leq k$ and $0\leq q\leq n-k$ with $(p,q)\neq(0,0)$; setting $R=(\sum_{i=1}^p (x^i)^2+\sum_{j=k+1}^{k+q}(x^j)^2)^{1/2}$ and $\Omega=R^{-1}(x^1,\ldots,x^p,x^{k+1},\ldots,x^{k+q})\in\Sph_p^{p+q-1}$ (the unit sphere in $\R^{p+q}$ intersected with $[0,\infty)^p\times\R^q$), the map $(x^1,\ldots,x^n)\mapsto(R,\Omega,x^{p+1},\ldots,x^k,x^{k+q+1},\ldots,x^n)$ extends from $X\setminus S$ to a diffeomorphism from a neighborhood of the front face of $[X;S]$ to $[0,\infty)\times\Sph_+^{p+q-1}\times\R^{(k-p)+(n-k-q)}$. The blow-down map is the product of the map $(R,\Omega)\mapsto R\Omega$ and the identity map in the remaining coordinates. Finally, if $T\subset X$ is another p-submanifold so that for each point in $S\cap T$ there exist local coordinates in which both $S$ and $T$ are of the above product form, then the lift of $T$ to $[X;S]$ is defined by $\upbeta^*T:=\upbeta^{-1}(T)$ when $T\subset S$, and $\upbeta^*T:=\cl(\upbeta^{-1}(T\setminus S))$ otherwise (where `$\cl$' denotes the closure in $[X;S]$). One can then define the iterated blow-up $[X;S;T]:=[[X;S];\upbeta^*T]$. (This construction can be iterated.)

When using index notation, we use the summation convention; that is, indices appearing twice are summed over (unless otherwise noted).

\subsection{00-structures}
\label{SsB00}

Mazzeo~\cite{Mazzeo00} defines spaces of 00-(pseudo)differential operators and studies their symbolic properties and parametrices; in the present paper, we restrict ourselves to the case of 00-\emph{differential} operators, and pursue a hands-on approach to their analysis. Such operators are tailored to the analysis of overdetermined operators on spaces which feature exponential weights at hypersurfaces. Thus, on a manifold $X$ with boundary $\pa X$, denote by $\rho\in\CI(X)$ a boundary defining function and recall from~\cite{MazzeoMelroseHyp} the Lie algebra $\cV_0(X)$ of 0-vector fields consisting of all $V\in\cV(X)$ which vanish at $\pa X$. We then set
\[
  \cV_{0 0}(X) = \rho\cV_0(X).
\]
In local coordinates $x\geq 0$, $y\in\R^{n-1}$, near a boundary point of $X$, elements of $\cV_{0 0}(X)$ are of the form $a(x,y)x^2\pa_x+b^j(x,y)x^2\pa_{y^j}$ for smooth $a,b$. Using that $\cV_0(X)$ is a Lie algebra (or by direct computation), one finds that commutators of 00-vector fields gain a factor of the boundary defining function,
\[
  [\cV_{0 0}(X),\cV_{0 0}(X)] \subset \rho\cV_{0 0}(X).
\]
Moreover, since $\cV_{0 0}(X)\subset\rho\Vb(X)$, we also have $\rho^{-\alpha}V(\rho^\alpha)\in\rho\CI(X)$ for $\alpha\in\R$. For the corresponding spaces $\Diff_{0 0}^k(X)$ and $\rho^{-\alpha}\Diff_{0 0}^k(X)$ of (weighted) 00-differential operators, this implies
\[
  A_j\in\rho^{-\alpha_j}\Diff_{0 0}^{k_j}(X),\ j=1,2 \implies
  [A_1,A_2]\in\rho^{-\alpha_1-\alpha_2+1}\Diff_{0 0}^{k_1+k_2-1}(X),
\]
i.e.\ composition of 00-operators is commutative to leading order in the differential order sense as usual, and also to leading order at $\pa X$. We can use more singular weights as well: for $\beta\in\R$, we have $e^{-\beta/\rho}V(e^{\beta/\rho})\in\CI(X)$, as is easily checked in local coordinates, and therefore conjugation by $e^{\beta/\rho}$ preserves the space $\rho^{-\alpha}\Diff_{0 0}^k(X)$

The space $\cV_{0 0}(X)$ is the space of smooth sections of the 00-tangent bundle ${}^{0 0}T X\to X$, with local frame $x^2\pa_x$, $x^2\pa_{y^j}$ ($j=1,\ldots,n-1$); the dual bundle is the 00-cotangent bundle ${}^{0 0}T^*X\to X$. Letting $P^k({}^{0 0}T^*X)$ denote the space of smooth functions on ${}^{0 0}T^*X$ which are fiberwise polynomials of degree $k$, we can define a surjective (full\footnote{This map combines the two symbols alluded to in~\S\ref{SsI00} into a single object.}) \emph{principal symbol} map
\begin{equation}
\label{EqB00Symb}
  {}^{0 0}\upsigma^k \colon \Diff_{0 0}^k(X) \to (P^k/\rho P^{k-1})({}^{0 0}T^*X)
\end{equation}
with kernel $\rho\Diff_{0 0}^{k-1}(X)$ as follows: in local coordinates $x\geq 0$, $y\in\R^{n-1}$, and writing 00-covectors as $\xi\frac{\dd x}{x^2}+\eta_j\frac{\dd y^j}{x^2}$ (so $\xi\in\R$, $\eta\in\R^{n-1}$), we assign
\begin{equation}
\label{EqB00Def}
  {}^{0 0}\upsigma^k \colon A = \sum_{j+|\alpha|\leq k} a_{j\alpha}(x,y)(x^2 D_x)^j (x^2 D_y)^\alpha \mapsto \sum_{j+|\alpha|\leq k} a_{j\alpha}(x,y)\xi^j\eta^\alpha.
\end{equation}
(The map ${}^{0 0}\upsigma^{k,\alpha}\colon\rho^{-\alpha}\Diff^k_{0 0}\to\rho^{-\alpha}P^k/\rho^{-\alpha+1}P^{k-1}$ is defined analogously.) Thus, ${}^{0 0}\upsigma^k(A)$ is the usual principal symbol over $X^\circ$ (where ${}^{0 0}T^*X$ and $T^*X$ are equal) and can be identified there with a \emph{homogeneous} polynomial, whereas over $x=0$, the 00-principal symbol captures the \emph{full} operator $A$ modulo 00-operators with decaying coefficients.

\begin{lemma}[Well-definedness of the principal symbol]
\label{LemmaB00Well}
  The map ${}^{0 0}\upsigma^k$, defined by~\eqref{EqB00Def}, is well-defined as a map~\eqref{EqB00Symb}, i.e.\ it is independent of the choice of local coordinates. It is multiplicative in the sense that ${}^{0 0}\upsigma^{k_1+k_2}(A_1\circ A_2)={}^{0 0}\upsigma^{k_1}(A_1)\cdot{}^{0 0}\upsigma^{k_2}(A_2)$ for $A_j\in\Diff_{0 0}^{k_j}(X)$, $j=1,2$, and it maps adjoints to adjoints (complex conjugates).
\end{lemma}
\begin{proof}
  We can write coordinates $\tilde x\geq 0$, $\tilde y\in\R^{n-1}$ near $(0,0)$ as $\tilde x=x b(x,y)$ and $\tilde y=\Phi(x,y)$ with $0<b$ and $\Phi(x,-)$ a local diffeomorphism for small $x$. Then
  \[
    x^2 D_x = a_{1 1}\tilde x^2 D_{\tilde x} + a_{1 2}\tilde x^2 D_{\tilde y},\quad
    x^2 D_y = a_{2 1}\tilde x^2 D_{\tilde x} + a_{2 2}\tilde x^2 D_{\tilde y};
  \]
  where $a_{1 1}=b^{-2}(b + x b'_x)$, $a_{1 2}=b^{-2}\Phi'_x$, $a_{2 1}=x b^{-2}b'_y$, and $a_{2 2}=b^{-2}\Phi'_y$ are smooth. We likewise have
  \[
    \tilde\xi\frac{\dd\tilde x}{\tilde x^2} + \tilde\eta\frac{\dd\tilde y}{\tilde x^2} = \xi\frac{\dd x}{x^2}+\eta\frac{\dd y}{x^2},
  \]
  with $\xi=a_{1 1}\tilde\xi+a_{1 2}\tilde\eta$ and $\eta=a_{2 1}\tilde\xi+a_{2 2}\tilde\eta$. (Thus, the putative principal symbols of $x^2 D_x$ and $x^2 D_y$ are indeed well-defined.) Expanding $(x^2 D_x)^j=\sum_{p+|\beta|\leq j} f_{p\beta}(\tilde x,\tilde y)(\tilde x^2 D_{\tilde x})^p (\tilde x^2 D_{\tilde y})^\beta$, all terms with $p+|\beta|<j$ involve at least one derivative $\tilde x^2 D_{\tilde x}$ or $\tilde x^2 D_{\tilde y}$ falling onto a coefficient $a_{1 1},a_{1 2}$, or $\tilde x^2 D_{\tilde x}$ falling onto $\tilde x^2$; any such term thus comes with at least one additional power of $\tilde x$. Therefore, modulo $\tilde x\Diff_{0 0}^{j-1}$, only the terms with $p+|\beta|=j$ survive, and the putative principal symbols of the two sides are thus $\xi^j$ and $(a_{1 1}\tilde\xi+a_{1 2}\tilde\eta)^j$, which are equal. The same argument applies more generally to $(x^2 D_x)^j(x^2 D_y)^\alpha$.

  The multiplicativity of the 00-principal symbol follows similarly: upon writing
  \[
    a^{1,j\alpha}(x,y)(x^2 D_x)^j(x^2 D_y)^\alpha \circ a_{2,k\beta}(x,y)(x^2 D_x)^k(x^2 D_y)^\beta
  \]
  in the form~\eqref{EqB00Def} by commuting $(x^2 D_x)^j(x^2 D_y)^\alpha$ through $a_{2,k\beta}$, and further commuting $(x^2 D_y)^\alpha$ through $(x^2 D_x)^k$, each commutator drops one derivative and gains one power of $x$.
\end{proof}

The setting in which 00-operators arise in the present paper takes the following general form: let $X^\circ$ be an open manifold which is separated into two connected components $X_+$ and $X_-$ by a hypersurface $H$; let $\rho\in\CI(X^\circ)$ be a defining function of $H$ inside $X_+$. Let $L\in\Diff^m(X^\circ)$; then
\[
  L_+ := \rho^{2 m}L \in \Diff_{0 0}^m(X_+).
\]
Indeed, in local coordinates $x\in\R$, $y\in\R^{n-1}$ in which $X_+$ is given by $x\geq 0$, the operator $L$ is a linear combination (with $\CI$ coefficients) of $D_x^j D_y^\alpha$, $j+|\alpha|\leq m$, and therefore
\[
  x^{2 m}D_x^j D_y^\alpha=x^{2(m-j-|\alpha|)}\bigl(x^{2 j}D_x^j (x^2 D_y)^\alpha - [x^{2 j}D_x^j,x^{2|\alpha|}]x^{-2|\alpha|}(x^2 D_y)^\alpha\bigr)\in\Diff_{0 0}^m
\]
indeed. This calculation also implies ${}^{0 0}\upsigma^m(L_+)=M^*\upsigma^m(L)$, where $\upsigma^m(L)\in P^m_{\rm hom}(T^*X^\circ)$ denotes the standard (homogeneous) principal symbol of $L$, and $M\colon{}^{0 0}T^*X_+\to T^*_{X_+}X^\circ$ is the bundle isomorphism given by multiplication by $\rho^2$. In particular, the boundary principal symbol ${}^{0 0}\upsigma^m(L_+)|_{{}^{0 0}T^*_H X_+}$ is a homogeneous polynomial, and correspondingly $L_+$ cannot have an injective principal symbol (except in the case $m=0$). A conjugation by an exponential weight can change this, since
\begin{equation}
\label{EqB00SymbConj}
  {}^{0 0}\sigma^m(e^{-\beta/\rho}L_+ e^{\beta/\rho})(z,\zeta) = {}^{0 0}\sigma^m(L_+)\Bigl(\zeta+i\beta\frac{\dd\rho}{\rho^2}\Bigr)=\upsigma^m(L)(\rho^2\zeta+i\beta\,\dd\rho),\quad \zeta\in{}^{0 0}T^*_z X_+.
\end{equation}
(This follows in local coordinates by multiplicativity from the easy cases $L_+=\rho^2 D_\rho$---with $e^{-\beta/\rho}\rho^2 D_\rho e^{\beta/\rho}=\rho^2 D_\rho+i\beta$---and $L_+=\rho^2 D_y$.) Note that the right hand side of~\eqref{EqB00SymbConj} is not homogeneous in $\zeta$ anymore (unless $m=0$ and $\beta=0$).

\begin{example}[Exterior derivative]
\label{ExB00Ext}
  Consider $X^\circ=\R^n$, with coordinates $(x,y)\in\R\times\R^{n-1}$, and $H=\{x=0\}$, $X_\pm=\{\pm x\geq 0\}$, and $\rho=x$; let $L=\dd\colon\CI(\R^n)\to\CI(\R^n;T^*\R^n)$, with principal symbol $\upsigma^1(L)(z,\zeta)=i\zeta$. Then $L_+=x^2\,\dd=(x^2\pa_x,x^2\pa_y)$, and its conjugation
  \[
    e^{-\beta/x}L_+ e^{\beta/x}=x^2\,\dd - \beta\,\dd x \in \Diff_{0 0}^1(X_+;\ul\R,T^*_{X_+}\R^n),
  \]
  has principal symbol ${}^{0 0}T^*X_+ \ni (\xi,\eta) \mapsto (i x^2\xi-\beta,i\eta)$. This is injective for $\beta\neq 0$; note that $x^2\colon{}^{0 0}T^*X_+\to T^*_{X_+}\R^n$ is an isomorphism.
\end{example}

\subsection{Function spaces}
\label{SsBFn}

Let $X$ be a manifold with corners and boundary hypersurfaces $H_1,\ldots,H_N$, and denote a boundary defining function of $H_j$ by $\rho_j$. We write $\CIdot(X)\subset\CI(X)$ for the space of smooth functions which vanish to infinite order at each $H_j$. For $\alpha_j\in\R$ ($j=1,\ldots,N$), set $w=\prod_{j=1}^N \rho_j^{\alpha_j}$; we then define the space of \emph{conormal functions}
\[
  \cA^{\alpha_1,\ldots,\alpha_N}(X) = \{ u\in\CI(X^\circ) \colon L u \in w L^\infty_\loc(X)\ \forall L\in\Diffb(X) \}.
\]
We refer to $\alpha_j$ as the weight at the $j$-th boundary hypersurface of $X$. Polyhomogeneous spaces refine the mere boundedness to generalized Taylor expansions. \emph{Index sets} capture the exponents in the Taylor expansion; recall here that an index set is a subset $\cE\subset\C\times\N_0$ so that $(z,k)\in\cE$ implies $(z+1,k)\in\cE$, and also $(z,k-1)\in\cE$ when $k\geq 1$, and so that $\Re z_j\to\infty$ whenever $(z_j,k_j)\in\cE$ tends to infinity. (Thus, the subset of elements $(z,k)\in\cE$ with $\Re z<C$ is finite for any choice of $C\in\R$.) Frequently occurring index sets are denoted somewhat imprecisely as
\begin{alignat*}{2}
  &\N_0&\quad&\text{(for the index set $\N_0\times\{0\}$)}, \\
  &\N_0+z\ \text{or simply}\ z&\quad&\text{(for the index set $\{(z+j,0)\colon j\in\N_0\}$)}, \\
  &(z,k)&\quad&\text{(for the index set $\{(z+j,l)\colon j\in\N_0,\ 0\leq l\leq k$)}.
\end{alignat*}
We write $\min\Re\cE=\min\{\Re z\colon(z,k)\in\cE\}$; for $a\in\R$, and we write $\Re\cE>a$ as an abbreviation for $\min\Re\cE>a$. If now $\cE=(\cE_1,\ldots,\cE_N)$, $\cE_j\subset\C\times\N_0$, is a collection of index sets, and $\Re\cE_j>a_j$, then
\[
  \cA_\phg^{\cE_1,\ldots,\cE_N}(X) = \cA_\phg^\cE(X) \subset \cA^{a_1,\ldots,a_N}(X)
\]
consists of all conormal functions $u$ so that for all $j=1,\ldots,N$, the following holds in a collar neighborhood $[0,\eps)_{\rho_j}\times H_j$ of $H_j\subset X$: letting $\alpha^{(j)}=(\alpha_1,\ldots,\alpha_{j-1},\alpha_{j+1},\ldots,\alpha_N)$, there exist $a_{(z,k)}\in\cA^{\alpha^{(j)}}(H_j)$ for $(z,k)\in\cE_j$ so that
\[
  u(\rho_j,x)-\sum_{\genfrac{}{}{0pt}{}{(z,k)\in\cE_j}{\Re z\leq C}}\rho_j^z(\log\rho_j^{-1})^k a_{(z,k)}(x)\in\cA^{\alpha_1,\ldots,\alpha_{j-1},C,\alpha_{j+1},\ldots,\alpha_N}(X)
\]
for all $C\in\R$. Necessarily, then, the $a_{(z,k)}$ are themselves polyhomogeneous on $H_j$, with index set $\cE_i$ at the boundary hypersurface $H_j\cap H_i$ of $H_j$ for those $i$ for which $H_j\cap H_i\neq\emptyset$. See \cite[Chapter~4]{MelroseDiffOnMwc} for a detailed treatment of polyhomogeneous distributions. If $u\in\cA_\phg^\cE(X)$ and $v\in\cA_\phg^\cF(X)$, then $u+v\in\cA_\phg^{\cE\cup\cF}(X)$ and $u\cdot v\in\cA_\phg^{\cE+\cF}(X)$, where we set $\cE+\cF=\{(z_1+z_2,k_1+k_2)\colon (z_1,k_1)\in\cE,\ (z_2,k_2)\in\cF\}$. For $k\in\N$, we also write
\[
  k\cE = \cE+\ldots+\cE\quad\text{($k$ summands)};
\]
thus $u^k\in\cA_\phg^{k\cE}(X)$ when $u\in\cA_\phg^\cE(X)$. We say that an index set $\cE$ is \emph{nonlinearly closed} if $\cE=\bigcup_{j\in\N}j\cE$. The \emph{nonlinear closure} of an index set $\cE$ with $\Re\cE>0$ is defined by $\cF:=\bigcup_{j\in\N}j\cE$; it is the smallest index set with the property that $u^k\in\cA_\phg^\cF(X)$ for all $k\in\N_0$ and for all $u\in\cA_\phg^\cE(X)$. We finally recall the \emph{extended union} of index sets
\[
  \cE\extcup\cF := \cE\cup\cF \cup \{(z,k_1+k_2+1)\colon(z,k_1)\in\cE,\ (z,k_2)\in\cF \}.
\]

We next recall \emph{weighted b-Sobolev spaces} on a compact manifold $X$ with boundary $\pa X\neq\emptyset$. Fix a smooth positive b-density on $X$, i.e.\ a positive section of the density bundle $\Omegab X\to X$ associated to $\Tb X\to X$; in local coordinates $x\geq 0$, $y\in\R^{n-1}$, such a b-density takes the form $a(x,y)|\frac{\dd x}{x}\,\dd y|$ where $0<a\in\CI$. Then $L^2(X)$ denotes the corresponding $L^2$-space on $X^\circ$. For $s\in\N_0$, we define more generally the b-Sobolev space $\Hb^s(X)$ to consist of all $u\in L^2(X)=\Hb^0(X)$ so that $L u\in\Hb^0(X)$ for all $L\in\Diffb^s(X)$. If $\{V_1,\ldots,V_N\}\subset\Vb(X)$ is a collection of b-vector fields which spans $\Vb(X)$, then $\|u\|_{\Hb^s(X)}^2:=\sum\|V_{i_1}\ldots V_{i_m}u\|_{L^2(X)}^2$, with the sum taken over all $m\leq s$ and $1\leq i_1,\ldots,i_m\leq N$, gives $\Hb^s(X)$ the structure of a Hilbert space. Weighted b-Sobolev spaces are denoted
\[
  \Hb^{s,\alpha}(X)=\rho^\alpha\Hb^s(X)=\{\rho^\alpha u\colon u\in\Hb^s(X)\},
\]
with $\rho\in\CI(X)$ denoting a boundary defining function; this is a Hilbert space with norm $\|u\|_{\Hb^{s,\alpha}(X)}=\|\rho^{-\alpha}u\|_{\Hb^s(X)}$. For $s\in\R$, one can define $\Hb^{s,\alpha}(X)$ via interpolation and duality. Since conjugation by $\rho^\alpha$ preserves the space $\Diffb^k(X)$, an operator $A\in\rho^{-\beta}\Diffb^k(X)$ defines a bounded linear map $\Hb^{s,\alpha}(X)\to\Hb^{s-k,\alpha-\beta}(X)$ for all $s,\alpha\in\R$. In a completely analogous fashion, one can define 00-Sobolev spaces $H_{0 0}^s(X)$ (relative to a positive 00-density, in local coordinates $a(x,y)|\frac{\dd x}{x^2}\frac{\dd y^1}{x^2}\cdots\frac{\dd y^{n-1}}{x^2}|$) and their weighted analogues\footnote{The sign convention is such that increasing any one of $\alpha,\beta,s$ gives a smaller space.}
\[
  e^{-\beta/\rho}\rho^\alpha H_{0 0}^s(X),
\]
with (weighted) 00-operators giving bounded maps between such spaces with appropriately shifted orders. When one order, say $s$, takes the value $+\infty$ (resp.\ $-\infty$), the corresponding function space is defined as the intersection (resp.\ union) over all $s\in\R$, so e.g.\ $H_{0 0}^\infty(X)=\bigcap_{s\in\R}H_{0 0}^s(X)$; likewise in the cases $\alpha=\infty$ and $\beta=\infty$.

Returning to b-Sobolev spaces, we recall their interaction with the Mellin-transform from \cite[\S3.1]{VasyMicroKerrdS} (though we opt to use large parameter spaces here instead of semiclassical ones). To wit, we work on $[0,\infty)_x\times\pa X$, with $\pa X$ a compact manifold without boundary; this is the model for a manifold with boundary near its boundary. The Mellin transform of a function $u=u(x,y)$ with compact support in $x^{-1}((0,\infty))$ is then the Fourier transform in logarithmic coordinates,
\[
  (\cM u)(\lambda,y) = \int_0^\infty x^{-i\lambda}u(x,y)\,\frac{\dd x}{x}.
\]
The range can be characterized using the Paley--Wiener theorem; the inverse Mellin transform is $(\cM_\alpha^{-1}\hat u)(x,y) = \int_{\Im\lambda=-\alpha} x^{i\alpha}\hat u(\lambda,y)\,\dd\lambda$ for any $\alpha\in\R$. The Mellin transform diagonalizes dilation-invariant operators: if $N(A)\in\Diffb^k([0,\infty)\times\pa X)$ is the (dilation-invariant) b-normal operator of a b-differential operator $A$, then $\cM(N(A)u)(\lambda)=N(A,\lambda)(\cM u)(\lambda)$.

Equipping $\pa X$ with a volume density, and taking its product with $|\frac{\dd x}{x}|$ to define a b-density on $[0,\infty)\times\pa X$, Plancherel's Theorem gives an isomorphism
\[
  \cM \colon x^\alpha\Hb^0([0,\infty)\times\pa X) \xra{\cong} L^2(\{\Im\lambda=-\alpha\};L^2(\pa X)).
\]
Upon defining higher order b-Sobolev spaces via testing with dilation-invariant vector fields, i.e.\ with $x\pa_x$ and $x$-independent vector fields on $\pa X$, this extends to an isomorphism
\begin{equation}
\label{EqBFnMIso}
  \cM \colon x^\alpha\Hb^s([0,\infty)\times\pa X) \to L^2(\{\Im\lambda=-\alpha\};H^s_{\la\lambda\ra}(\pa X))
\end{equation}
for $s\in\N_0$, where we introduced the \emph{large parameter Sobolev space} $H_{\la\lambda\ra}^s(\pa X)$. As a set, the space $H_{\la\lambda\ra}^s(\pa X)$ is equal to $H^s(\pa X)$, but its norm depends on $\lambda$: if $V_1,\ldots,V_N\in\cV(\pa X)$ is a collection of vector fields on $X$ which spans $\cV(\pa X)$, we set
\[
  \|u\|_{H^s_{\la\lambda\ra}(\pa X)}^2 = \sum_{\genfrac{}{}{0pt}{}{\alpha\in\N_0^{N+1}}{|\alpha|\leq s}} \|(V_1,\ldots,V_N,\la\lambda\ra)^\alpha u\|_{L^2(\pa X)}^2.
\]
Via interpolation and duality, one can extend the isomorphism~\eqref{EqBFnMIso} to $s\in\R$.

We also recall the action of the Mellin-transform on conormal or polyhomogeneous distributions; see \cite[\S2A]{MazzeoEdge} for an overview, and \cite[\S4]{MelroseDiffOnMwc} for a detailed account. To wit, if $u\in\cA^\alpha([0,\infty)_x\times H)$ is supported in $x\leq x_0\in(0,\infty)$, then $(\cM u)(\lambda)$ is holomorphic in $\Im\lambda>-\alpha$, and for any $\eps>0$ we have the Paley--Wiener type bound
\[
  \|(\cM u)(\lambda)\|_{\cC^N(H)} \leq C_{N,\eps} x_0^{\Im\lambda}\la\Re\lambda\ra^{-N},\qquad N\in\N,\ \Im\lambda\geq-\alpha+\eps.
\]
Conversely, if $\cM u$ satisfies these bounds, then $u$ is supported in $x\leq x_0$, and $u\in\cA^{\alpha-\eps}$ for all $\eps>0$.

When $\cE$ is an index set and $u\in\cA^\cE_\phg([0,\infty)_x\times H)$, with $x\leq x_0$ on $\supp u$, then $(\cM u)(\lambda)$ is meromorphic in $\lambda\in\C$ with values in $\CI(H)$, with poles of order $\leq k+1$ at $\lambda=-i z$ for $(z,k)\in\cE$; and the above Paley--Wiener type bound holds for any $\Im\lambda\geq -C$ for sufficiently large $|\Re\lambda|$ (depending on $\cE$ and $C$). Conversely, if $\cM u$ satisfies these conditions, then $u\in\cA^\cE_\phg$ (together with the support property).

\subsection{Estimates for b- and 00-differential operators}
\label{SsBAn}

In this paper, we only work with differential operators $L$ that have injective, surjective, or invertible (elliptic) principal symbols $\upsigma(L)$ in the appropriate sense; note that if $\upsigma(L)$ is injective, resp.\ surjective, then $L^*L$, resp.\ $L L^*$ is elliptic, and therefore (approximate) left, resp.\ right inverses of $L$ can be constructed from (approximate) inverses of an appropriate elliptic operator. In the b-setting, the (large) calculus of b-pseudodifferential operators \cite{MelroseAPS} is a very precise tool for this purpose; one can similarly construct an equally useful algebra of 00-pseudodifferential operators (which is rather similar to the scattering ps.d.o.\ algebra) for the (approximate) inversion of elliptic 00-differential operators. Here, in order to keep the paper more easily accessible, we opt for a direct approach based on a bounded geometry perspective, cf.\ \cite{ShubinBounded}, and (for the 00-boundary symbol) on direct Fourier inversion.

We begin with the b-setting, and indeed work in a local coordinate chart $X=[0,\infty)_x\times\R^{n-1}_y$ which we equip with the b-density $|\frac{\dd x}{x}\,\dd y|$; we moreover define b-Sobolev spaces via testing with the vector fields $x\pa_x$ and $\pa_y$. For $p\in\Z^{n-1}$ and $j\in\Z$, we define the diffeomorphism
\begin{equation}
\label{EqBBdd}
\begin{split}
  \Phi_{p,j}&\colon \cD:=(-2,2)_w\times B(0,2)_z\xra{\cong} \cD_{p,j}:=(2^{-j-2},2^{-j+2})\times B(p,2) \subset X, \\
  \Phi_{p,j}(w,z)&=(2^{-j-w},p+z).
\end{split}
\end{equation}
Note that $\Phi_{p,j}^*((\log 2)x\pa_x)=-\pa_w$ and $\Phi_{p,j}^*(\pa_y)=\pa_z$, so $\cD_{p,j}$ is a unit size cell for b-vector fields. See Figure~\ref{FigBBdd}.

\begin{figure}[!ht]
\centering
\includegraphics{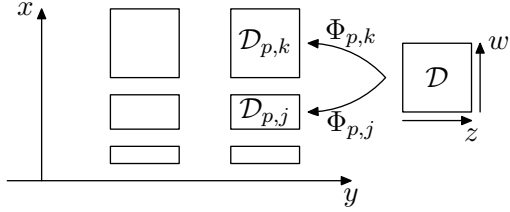}
\caption{Illustration of the unit cells in the b-setting; see~\eqref{EqBBdd}.}
\label{FigBBdd}
\end{figure}

Fix a nonnegative function $\chi\in\CIc(\cD)$ so that $\chi=1$ on $[-1,1]\times\{|z|\leq 1\}$, and put $\chi_{p,j}:=(\Phi_{p,j})_*\chi$; then the sum of all $\chi_{p,j}$, $p\in\Z^{n-1}$, $j\in\Z$, is locally finite and everywhere positive, and there exists some number $J\in\N$ so that the intersection of $\supp\chi_{p,j}$ for more than $J$ pairwise distinct pairs $(p,j)\in\Z^{n-1}\times\Z$ is empty.

\begin{lemma}[b-Sobolev norms]
\label{LemmaBAnHb}
  Let $s,\alpha\in\R$. Then we have an equivalence of norms
  \[
    \|u\|_{\Hb^{s,\alpha}(X)}^2 \sim \sum_{\genfrac{}{}{0pt}{}{p\in\Z^{n-1}}{j\in\Z}} 2^{2 j\alpha}\|\Phi_{p,j}^*(\chi_{p,j}u)\|_{H^s(\R^n)}^2,
  \]
  in the sense that there exists a constant $C>1$ only depending on $s,\alpha$ so that the left hand side is bounded by $C$ times the right hand side and vice versa.
\end{lemma}
\begin{proof}
  Since $x/2^{-j}\in[\frac14,4]$ on $\cD_{p,j}$, we only need to consider the case $\alpha=0$. By interpolation and duality, it moreover suffices to show, for $s\in\N_0$, the boundedness of the maps
  \begin{alignat*}{2}
    \Xi &\colon \Hb^s(X) \to \ell^2(\Z^{n-1}\times\Z;H^s(\R^n)),&\qquad u &\mapsto (\Phi_{p,j}^*(\chi_{p,j}u))_{(p,j)\in\Z^{n-1}\times\Z}, \\
    \Xi' & \colon \ell^2(\Z^{n-1}\times\Z;H^s(\R^n)) \to \Hb^s(X), &\qquad (v_{p,j}) &\mapsto \sum_{(p,j)\in\Z^{n-1}\times\Z} (\Phi_{p,j})_*(\chi v_{p,j}).
  \end{alignat*}
  (For $s=0$, the map $\Xi'$ is the adjoint of $\Xi$.) The boundedness of $\Xi$ is equivalent to the estimate $\sum\|\chi_{p,j}u\|_{H^s_\bop}^2\lesssim\|u\|_{H^s_\bop}^2$; this is clear for $s=0$, and follows for $s\in\N$ by induction, noting that the pullbacks of $[x\pa_x,\chi_{p,j}]$ and $[\pa_y,\chi_{p,j}]$ along $\Phi_{p,j}$ are uniformly bounded in $\CI(\R^n)$. The boundedness of $\Xi'$ follows for $s=0$ from the Cauchy--Schwarz inequality, using the finite intersection property of $\supp((\Phi_{p,j})_*\chi)=\supp\chi_{p,j}$; for $s\in\N$, one uses induction.
\end{proof}

\begin{cor}[b-Sobolev embedding]
\label{CorBAnSob}
  On an $n$-dimensional compact manifold $X$ with boundary, equipped with a smooth positive b-density, we have $\Hb^{s,\alpha}(X)\hra\rho^\alpha L^\infty(X)$ for $s>\frac{n}{2}$, where $\rho\in\CI(X)$ is a boundary defining function. More generally, $\Hb^{s,\alpha}(X)\hra\rho^\alpha\cC_\bop^k(X)$ for $s>\frac{n}{2}+k$, where $\cC_\bop^k(X)$ denotes the space of $\cC^k$-functions on $X^\circ$ all of whose b-derivatives of order up to $k$ are bounded. Finally, $\Hb^{\infty,\alpha}(X)\hra\cA^\alpha(X)$.
\end{cor}
\begin{proof}
  This follows from Lemma~\ref{LemmaBAnHb} and Sobolev embedding on $\R^n$.
\end{proof}

As a consequence of elliptic regularity on bounded subsets of $\R^n$, we obtain:

\begin{lemma}[Elliptic b-estimate]
\label{LemmaBAnEll}
  Let $L\in\rho^{-\beta}\Diffb^k(X)$. Let $\chi_1,\chi_2\in\CIc(X)$, with $\chi_2\equiv 1$ near $\supp\chi_1$, and with $\rho^\beta L$ having an elliptic principal symbol near $\supp\chi_2$. Then for all $s,s_0\in\R$, there exists a constant $C$ so that
  \begin{equation}
  \label{EqBAnEllEst}
    \|\chi_1 u\|_{\Hb^{s,\alpha}(X)} \leq C\Bigl( \|\chi_2 L u\|_{\Hb^{s-k,\alpha-\beta}(X)} + \|\chi_2 u\|_{\Hb^{s_0,\alpha}(X)}\Bigr).
  \end{equation}
  This holds in the strong sense that if the right hand side is finite, then so is the left hand side and the estimate holds.
\end{lemma}

Via a partition of unity argument, we also obtain the estimate~\eqref{EqBAnEllEst} for elliptic b-differential operators on manifolds.

\begin{proof}[Proof of Lemma~\usref{LemmaBAnEll}]
  We can reduce to the case $\alpha=\beta=0$. Then, for $\tilde\chi\in\CIc(\cD)$ identically $1$ near $\supp\chi$, the pullback $\tilde\chi\Phi_{p,j}^*L\in\Diff^k(\cD)$ is uniformly bounded and uniformly elliptic on $\tilde\chi^{-1}(1)$. Thus, for $s_0\leq s-1$, we have a uniform bound
  \begin{align*}
    \|\Phi_{p,j}^*(\chi_{p,j}\chi_1 u)\|_{H^s(\R^n)} &= \|\chi \Phi_{p,j}^*(\chi_1 u)\|_{H^s(\R^n)} \\
      &\leq C\Bigl( \|\tilde\chi \Phi_{p,j}^*(L\chi_1 u) \|_{H^s(\R^n)} + \|\tilde\chi\Phi_{p,j}^*(\chi_1 u)\|_{H^{s_0}(\R^n)}\Bigr) \\
      &\leq C\Bigl( \|\tilde\chi\Phi_{p,j}^*(\chi_1 L u)\|_{H^s(\R^n)} + \|\tilde\chi\Phi_{p,j}^*(\chi_2 u)\|_{H^{s-1}(\R^n)}\Bigr),
  \end{align*}
  which gives~\eqref{EqBAnEllEst} for $s_0=s-1$. Iterating this estimate on the error term (upon enlarging the supports of the cutoffs) finitely many times gives~\eqref{EqBAnEllEst} in general.
\end{proof}

Consider now the case $L\in\Diffb^k(X)$. Even when $X$ is compact and one takes $\chi_1=\chi_2=1$ in~\eqref{EqBAnEllEst}, the error term is not relatively compact (i.e.\ the inclusion $\Hb^{s,\alpha}(X)\hra\Hb^{s_0,\alpha}(X)$ fails to be a compact map); in fact, the Fredholm property of elliptic b-operators requires the invertibility of its normal operator \cite[\S5.17]{MelroseAPS}. Note then that when $L\in\Diffb^k(X)$ is elliptic, the Mellin-transformed normal operator family $N(L,\lambda)\in\Diff^k(\pa X)$, regarded as an operator with large parameter $\Re\lambda$, is elliptic when $|\Im\lambda|$ is bounded by an arbitrary but fixed constant $C_1$, as follows from the explicit expression~\eqref{EqBNormMT} and the fact that the contributions from $\Im\lambda$ (when writing $\lambda=\Re\lambda+i\Im\lambda$) are subprincipal. A parametrix $Q(\lambda)$ for $N(L,\lambda)$ can then be constructed within the class of pseudodifferential operators with large parameter \cite{ShubinSpectralTheory}, and as such, $Q(\lambda)$ has order $-k$; it has the property that $Q(\lambda)\circ N(L,\lambda)=I+R(\lambda)$ where $R(\lambda)$ is residual, in the sense that the Schwartz kernel of $R(\lambda)$ is smooth and rapidly vanishing as $|\Re\lambda|\to\infty$, $|\Im\lambda|<C_1$. Thus, $I+R(\lambda)$ is invertible for large $|\Re\lambda|$, with inverse on $L^2(\pa X)$ given by a convergent Neumann series whose limit is again of the form $I+\tilde R(\lambda)$ where $\tilde R(\lambda)$ is residual; therefore, $N(L,\lambda)^{-1}=(I+\tilde R(\lambda))Q(\lambda)$ is, for large $|\Re\lambda|$ (depending on the bound $|\Im\lambda|<C_1$) a large parameter ps.d.o.\ of order $-k$. As such, it is uniformly bounded as a map
\[
  N(L,\lambda)^{-1} \colon H_{\la\lambda\ra}^s(\pa X)\to H_{\la\lambda\ra}^{s-k}(\pa X)
\]
between large parameter Sobolev spaces (see~\S\ref{SsBFn}). The Fredholm property of an elliptic operator $L$ as a map $L\colon\Hb^{s,\alpha}(X)\to\Hb^{s-k,\alpha}(X)$ is equivalent to the invertibility of $N(L,\lambda)$ for \emph{all} $\lambda\in\C$ with $\Im\lambda=-\alpha$. See \cite[\S6]{VasyMinicourse} for details. (A particular instance of this is given in the proof of Lemma~\ref{LemmaCEAfCoker} below.)

\bigskip

We now turn to the 00-setting on $X=[0,\infty)_x\times\R^{n-1}_y$. The unit cells being of size $\sim x^2$ in the $x$- and $y$-directions at a point $(x,y)\in X^\circ$, we now define, for $p\in\Z^{n-1}$ and $3\leq j\in\N$, the diffeomorphism\footnote{We repurpose the notation previously used in the b-setting.}
\begin{equation}
\label{EqBAn00Cells}
\begin{split}
  \Phi_{p,j} \colon \cD=(-2,2)_w\times B(0,2)_z &\xra{\cong} \cD_{p,j} = \Bigl(\frac{1}{j+2},\frac{1}{j-2}\Bigr)_x \times B\Bigl(\frac{p}{j^2},\frac{2}{j^2}\Bigr), \\
  \Phi_{p,j}(w,z)&=\Bigl(\frac{1}{j+w},\frac{p+z}{j^2}\Bigr),
\end{split}
\end{equation}
where $\cD=(-2,2)_w\times B(0,2)_z$ as before. See Figure~\ref{FigBdd00} for an illustration. The inverse map is $(x,y)\mapsto(\frac{1}{x}-j,j^2 y-p)$; thus, the 00-vector fields $x^2\pa_x$ and $x^2\pa_y$ pull back to the uniformly bounded vector fields $-\pa_w$ and $(\frac{j}{j+w})^2\pa_z$ on $\cD$. Fixing $\chi\in\CIc(\cD)$, with $\chi\equiv 1$ on $[-1,1]\times\{|z|\leq 1\}$, the pushforwards $\tilde\chi_{p,j}=(\Phi_{p,j})_*\chi$ have similar properties as above: their sum is locally finite and everywhere positive for $x\leq\frac12$, and the intersections of more than a suitable fixed number $J\in\N$ of supports are empty.

\begin{figure}[!ht]
\centering
\includegraphics{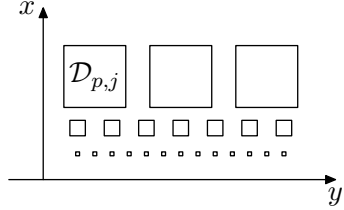}
\caption{Illustration of the unit cells in the 00-setting, cf.\ \eqref{EqBAn00Cells}.}
\label{FigBdd00}
\end{figure}

The restriction to $x\leq\frac12$ above is due to the fact that the scaling of the `natural' 00-vector fields $x^2\pa_x$ and $x^2\pa_y$ is artificial for $x\gtrsim 1$. For technical convenience, we thus fix $\chi_\pa\in\CIc([0,\frac12))$ to be identically $1$ on $[0,\frac14]$, and we fix $\rho=\rho(x)$ so that $\rho(x)=x$ for $x\leq\frac12$, $\rho$ is increasing, and $\rho(x)=1$ for $x\geq 1$; we then set
\[
  \|u\|_{e^{-\beta/\rho}\rho^\alpha H_{0 0}^s(X)}^2 := \|\chi_\pa u\|_{e^{-\beta/x}x^\alpha H_{0 0}^s}^2 + \|(1-\chi_\pa)u\|_{H^s(\R^n)}^2.
\]
The analogue of Lemma~\ref{LemmaBAnHb} is then for $s,\alpha,\beta\in\R$ the norm equivalence
\begin{equation}
\label{EqBAnH00}
  \|u\|_{e^{-\beta/\rho}\rho^\alpha H_{0 0}^s}^2 \sim \|(1-\chi_\pa)u\|_{H^s(\R^n)}^2 + \sum_{\genfrac{}{}{0pt}{}{p\in\Z^{n-1}}{j\in\Z}} e^{2\beta j}j^{2\alpha}\|\Phi_{p,j}^*(\chi_{p,j}\chi_\pa u)\|_{H^s(\R^n)}^2.
\end{equation}
The proof is analogous to that of Lemma~\ref{LemmaBAnHb}. (Only $\supp(1-\chi_\pa)$ intersects more than $J$ of the $\supp\chi_{p,j}$ nontrivially, which still allows for an application of the Cauchy--Schwarz inequality in the proof of the boundedness of the analogue of the map denoted $\Xi'$ there.)

\begin{cor}[00-Sobolev embedding]
\label{CorBAn00Sob}
  On an $n$-dimensional compact manifold $X$ with boundary, equipped with a positive 00-density, we have $H_{0 0}^s(X)\hra L^\infty(X)$ for $s>\frac{n}{2}$. Moreover, if $\beta>0$, then $e^{-\beta/\rho}H_{0 0}^\infty(X)\hra\CIdot(X)$.
\end{cor}
\begin{proof}
  (Cf.\ the discussion before \cite[Theorem~4]{CorvinoScalar}, and also \cite[Corollary~4.3]{DelayCompact}.) We only prove the last statement. Note that $H_{0 0}^s(X)\hra\cC_{0 0}^k(X)$ for $s>\frac{n}{2}+k$, where $\cC_{0 0}^k(X)$ consists of all elements of $\cC^k(X^\circ)$ which are uniformly bounded upon application of up to $k$ 00-derivatives. Let $u=e^{-\beta/\rho} u_0$ with $u_0\in H_{0 0}^s(X)$. Suppose first that $s>\frac{n}{2}$; then $u\in e^{-\beta/\rho}\cC_{0 0}^0(X)\hra e^{-\beta'/\rho}\cC^0(X)$ for all $\beta'<\beta$. If $s>\frac{n}{2}+1$, and if $V\in\cV(X)$, then
  \[
    V u = \rho^{-2} \rho^2 V e^{-\beta/\rho} u_0 = \rho^{-2}e^{-\beta/\rho}\bigl(\rho^2 V u_0 + e^{\beta/\rho}[\rho^2 V,e^{-\beta/\rho}]u_0\bigr).
  \]
  The terms in parentheses lie in $H_{0 0}^{s-1}(X)$, and thus $V u\in\rho^{-2}e^{-\beta/\rho}\cC_{0 0}^0(X)\hra e^{-\beta'/\rho}\cC^0(X)$. Continuing in this manner, we conclude that $e^{-\beta/\rho}H_{0 0}^s(X)\hra e^{-\beta'/\rho}\cC^k(X)$ for $s>\frac{n}{2}+k$ and $\beta'<\beta$. This implies the Corollary upon letting $s,k\nearrow\infty$, since $e^{-\beta'/\rho}\CI(X)\hra\CIdot(X)$ for $\beta'>0$.
\end{proof}

The characterization~\eqref{EqBAnH00} can be used to prove an analogue of Lemma~\ref{LemmaBAnEll} for 00-differential operators $L\in\Diff_{0 0}^k(X)$ which only need to be assumed to be elliptic in the differential order sense, meaning that the image of ${}^{0 0}\upsigma^m(L)$ in the quotient space $(P^m/P^{m-1})({}^{0 0}T^*X)=P^m_{\rm hom}({}^{0 0}T^*X)$ is elliptic. Thus,
\[
  \|\chi_1 u\|_{e^{-\beta/\rho}\rho^\alpha H_{0 0}^s(X)} \leq C\Bigl(\|\chi_2 L u\|_{e^{-\beta/\rho}\rho^\alpha H_{0 0}^{s-k}(X)} + \|\chi_2 u\|_{e^{-\beta/\rho}\rho^\alpha H_{0 0}^{s_0}(X)}\Bigr).
\]
Improving the error term---concretely, weakening the polynomial weight $\rho^\alpha$ to $\rho^{\alpha'}$ with $\alpha'<\alpha$---requires the ellipticity of the boundary principal symbol of $L$:

\begin{prop}[Elliptic 00-estimate near the boundary]
\label{PropBAnEll00}
  Let $X=[0,\infty)\times\R^{n-1}$ and $L\in\Diff_{0 0}^k(X)$. Let $R>0$, and suppose that $\ell={}^{0 0}\upsigma^k(L)$ is elliptic at $(0,y)\in X$ when $|y|<R$ (i.e.\ $\ell(0,y)\in P^m({}^{0 0}T^*_{(0,y)}X)$ is elliptic and non-vanishing). Let $\chi_1,\chi_2\in\CIc([0,\frac14)_x\times B(0,R))$, with $\chi_2\equiv 1$ near $\supp\chi_1$. Then for all $s,\alpha\in\R$ there exists a constant $C$ so that
  \begin{equation}
  \label{EqBAnEll00}
    \|\chi_1 u\|_{x^\alpha H_{0 0}^s(X)} \leq C\Bigl( \|\chi_2 L u\|_{x^\alpha H_{0 0}^{s-k}(X)} + \|\chi_2 u\|_{x^{\alpha-1}H_{0 0}^s(X)}\Bigr).
  \end{equation}
  This holds in the strong sense that if the right hand side is finite, then so is the left hand side and the estimate holds.
\end{prop}

Note that conjugation of $L$ by an exponential weight may destroy ellipticity of the boundary principal symbol; this is the reason for only using polynomially weighted spaces.

\begin{proof}[Proof of Proposition~\usref{PropBAnEll00}]
  Upon conjugating $L$ by $x^{-\alpha}$, which preserves the ellipticity condition, it suffices to consider the case $\alpha=0$.

  Let $\delta\in(0,\frac18)$ and $R_0<R_1<R_2<R$. It suffices to prove~\eqref{EqBAnEll00} in the case that $\chi_j(x,y)=\psi_j(x)\omega_j(y)$, $j=1,2$, where $\psi_j\in\CIc([0,j\delta))$ is identically $1$ on $[0,(j-\frac12)\delta]$, and $\omega_j\in\CIc(B(0,R_j))$ is identically $1$ on $B(0,R_{j-1})$. Fix moreover $\chi_3$ of the same product form, which is in between $\chi_1$ and $\chi_2$ in the sense that $\chi_3\equiv 1$ near $\supp\chi_1$, and $\chi_2\equiv 1$ near $\supp\chi_3$.

  Writing 00-covectors as $\xi\frac{\dd x}{x^2}+\eta\frac{\dd y}{x^2}$, the ellipticity of $\ell(0,y;\xi,\eta)$ implies an estimate $|\ell(0,y;\xi,\eta)|\geq c(1+|\xi|+|\eta|)^k$ for some $c>0$. Set
  \[
    q(x,y;\xi,\eta):=\chi_2(x,y)\ell(0,y;\xi,\eta)^{-1},
  \]
  which is thus a symbol of order $-k$ in $(\xi,\eta)$. Let $\phi\in\CIc((-\frac12,\frac12))$ be identically $1$ near $0$. We then define the oscillatory integral operator
  \begin{equation}
  \label{EqBAnEll00Quant}
  \begin{split}
    (Q f)(x,y) &:= (2\pi)^{-n} \iiiint e^{i[(x-x')\hat\xi+(y-y')\hat\eta]} \phi\Bigl(\frac{x-x'}{x'}\Bigr)\phi\Bigl(\frac{|y-y'|}{x'}\Bigr) \\
      &\quad\hspace{9em} \times q(x,y;x^2\hat\xi,x^2\hat\eta)f(x',y')\,\dd\hat\xi\,\dd\hat\eta\,\dd x'\,\dd y' \\
        &= (2\pi)^{-n} \iiiint \exp\Bigl(i\Bigl[\frac{x-x'}{x^2}\xi+\frac{y-y'}{x^2}\eta\Bigr]\Bigr) \phi\Bigl(\frac{x-x'}{x'}\Bigr)\phi\Bigl(\frac{|y-y'|}{x'}\Bigr) \\
      &\quad\hspace{9em} \times q(x,y;\xi,\eta)f(x',y')\,\dd\xi\,\dd\eta\,\frac{\dd x'}{x^2}\,\frac{\dd y'}{x^{2(n-1)}} \\
  \end{split}
  \end{equation}
  A few technical comments about this definition are in order. Firstly, the two cutoffs involving $\phi$ ensure that the Schwartz kernel of $Q$, given by the inverse Fourier transform of $q$ (i.e.\ dropping $f$ and the $(x',y')$-integration), is supported in $\frac{x}{x'}\simeq 1$ and $|y-y'|\lesssim x,x'$ (and thus will be shown to preserve whatever weight the input $f$ has on the scale of weighted 00-Sobolev spaces). Carefully note moreover that the cutoffs do \emph{not} localize $(x,y)$ and $(x',y')$ in the same unit cells for 00-geometry, but rather in the same unit cells for b-geometry: they cut the Schwartz kernel off in the region where the Fourier transform of $q$ already vanishes to infinite order in the limit $x+x'\searrow 0$.\footnote{The definition~\eqref{EqBAnEll00Quant}, for a general symbol $q=q(x,y;x^2\hat\xi,x^2\hat\eta)$ (with compact support in $(x,y)$) of order $k$, is a quantization map $\Op_{0 0}(q)$ for 00-pseudodifferential operators of order $k$; adding to the space of all such quantizations the space of residual operators---those having smooth Schwartz kernels on $X\times X$ which vanish to infinite order at $\pa X\times X$ and $X\times\pa X$---gives the full space $\Psi_{0 0}^k(X)$. Starting from the perspective of \cite{MazzeoMelroseHyp}, Schwartz kernels of elements of $\Psi_{0 0}^k(X)$ are conormal distributions on the `00-double space' $X^2_{0 0}:=[X^2_0;\pa\diag_0]$, where $X^2_0=[X^2;\pa\diag_X]$ is the 0-double space, with $\diag_X\subset X^2$ the diagonal, and $\diag_0\subset X^2_0$ is the lift of $\diag_X$. These conormal distributions are required to vanish to infinite order at all boundary hypersurfaces of $X^2_{0 0}$ except the front face of $X^2_{0 0}$. From this perspective, the necessity of localizing in the weak manner provided by the factors of $\phi$ is easily seen: the supports of the differentials of these cutoffs are disjoint from the 00-front face (and indeed instead intersect the 0-front face away from the 00-front face and the left and right boundaries).}

  We shall prove:
  \begin{align}
  \label{EqBAnEll00QBdd}
    \| Q(\chi_2 f)\|_{H_{0 0}^s} &\leq C \|\chi_2 f\|_{H_{0 0}^{s-k}}, \\
  \label{EqBAnEll00QL}
    Q L(\chi_1 u) &= \chi_1 u + R\chi_1 u,\quad \|R\chi_1 u\|_{H_{0 0}^s}\leq C\|x\chi_1 u\|_{H_{0 0}^s}=C\|\chi_1 u\|_{x^{-1}H_{0 0}^s}.
  \end{align}
  Granted these estimates, we have $\chi_1 u=Q L(\chi_1 u)-R\chi_1 u=Q(\chi_2 L\chi_1 u)-R\chi_1 u$; but since $\chi_2 L\chi_1 u=\chi_2 L u-\chi_2[L,\chi_1]u$, with $\|\chi_2[L,\chi_1]u\|_{H_{0 0}^{s-k}}\leq C\|\chi_2 u\|_{x^{-1}H_{0 0}^{s-1}}$, this gives~\eqref{EqBAnEll00}.

  In order to prove~\eqref{EqBAnEll00QBdd}, consider a unit cell $\cD_{p,j}$ (see~\eqref{EqBAn00Cells}) with nonempty intersection with $\supp\chi_2$, and consider points $(x,y)\in\cD_{p,j}$. If $(x,x',y,y')$ lies in the support of the integrand of~\eqref{EqBAnEll00Quant}, then $|\frac{x'}{x}|\in(\frac12,2)$ and $|y'-y|\in B(0,\frac12 x')$, and therefore $(x',y')\in\cD_{p',j'}$ where
  \begin{equation}
  \label{EqBAnEll00Indices}
    \frac{1}{C_1}j\leq j'\leq C_1 j,\qquad
    |p'-p|\leq C_1 j
  \end{equation}
  for some constant $C_1>1$ independent of $p,j$. Choosing the localization constant $\delta>0$ in the definition of $\chi_2$ small enough, we have $j\geq 10 C_1$ and thus $j'\geq 10$. We shall now prove the following estimate for the `matrix elements' of $Q$:
  \begin{equation}
  \label{EqBAnEll00QMxEl}
    \bigl\|\Phi_{p,j}^*\bigl(\chi_{p,j}Q(\tilde\chi_{p',j'} f)\bigr)\bigr\|_{H^s} \leq C_N(1+|j-j'|+|p-p'|)^{-N} \| \Phi_{p',j'}^*(\chi_{p',j'}f)\|_{H^{s-k}}.
  \end{equation}
  Here, $\tilde\chi_{p',j'}=\frac{\chi_{p',j'}}{S}$ where $S=\sum\chi_{p'',j''}$ is positive on $\supp\chi_2$; the pullbacks $\Phi_{p,j}^*S$, $\Phi_{p,j}^*S^{-1}$, and thus also $\Phi_{p,j}^*\tilde\chi_{p',j'}$ are uniformly bounded in $\CI(\cD)$, and $\sum_{p',j'}\tilde\chi_{p',j'}=1$ on $\supp\chi_2$. Once~\eqref{EqBAnEll00QMxEl} is proved, summing over $p',j'$ subject to~\eqref{EqBAnEll00Indices} gives a square-integrable sequence in $(p,j)$ provided $\|\Phi_{p',j'}^*(\chi_{p',j'}f)\|_{H^{s-k}}$ is square-integrable; indeed, note that the sum of~\eqref{EqBAnEll00QMxEl} over $p',j'$ is of convolution type, and $(1+|j|+|p|)^{-N}$ is summable for $N>n$. (That is, we use the fact that $\ell_1*\ell_2\subset\ell_2$.) In view of~\eqref{EqBAnH00}, this establishes~\eqref{EqBAnEll00QBdd} upon using $\chi_2 f$ instead of $f$ in~\eqref{EqBAnEll00QMxEl}.

  For the proof of~\eqref{EqBAnEll00QMxEl} then, consider first the case $(p',j')=(p,j)$. Inserting $(x,y)=\Phi_{p,j}(w,z)$ and $(x',y')=\Phi_{p',j'}(w,z)$ in~\eqref{EqBAnEll00Quant}, the operator $\chi_{p,j}Q\tilde\chi_{p',j'}$ becomes an operator on $\R^n$ with Schwartz kernel
  \begin{align*}
    &(2\pi)^{-n}\iint \exp\Bigl(i\Bigl[\frac{w'-w}{(j+w)/(j-w)}\xi+\frac{z-z'}{(j/(j+w))^2}\eta\Bigr]\Bigr)\phi_2(w,z,w',z') \\
      &\quad \times\chi(w,z)q(\Phi_{p,j}(w,z);\xi,\eta)(\Phi_{p,j}^*\tilde\chi_{p,j})(w',z')\,\dd\xi\,\dd\eta\,\Bigl|\frac{\dd w'}{((j+w)/(j-w))^2}\frac{\dd z'}{(j/(j+w))^{2(n-1)}}\Bigr|,
  \end{align*}
  where $\phi_2$ is the product of the two $\phi$ factors in~\eqref{EqBAnEll00Quant}. Changing variables to $\frac{\xi}{(j+w)/(j-w)}$ and $\frac{\eta}{(j/(j+w))^2}$ gives a pseudodifferential operator on $\R^n$ which obeys uniform (in $p,j$) bounds as a map $H^{s-k}(\R^n_{(w',z')})\to H^s(\R^n_{(w,z)})$.

  Consider next the case that the supports of $\chi_{p,j}$ and $\chi_{p',j'}$ are disjoint; in this case, we can integrate by parts in~\eqref{EqBAnEll00Quant} using that $|(\frac{x-x'}{x^2},\frac{y-y'}{x^2})|^{-2}(\frac{x-x'}{x^2},\frac{y-y'}{x^2})\cdot\nabla_{(\xi,\eta)}$ preserves the exponential, and using that $\nabla_{(\xi,\eta)}$ reduces the symbolic order of $q$ by $1$. The operator $\chi_{p,j}Q\tilde\chi_{p',j'}$ thus has a smooth Schwartz kernel which is rapidly decaying as $|(\frac{x-x'}{x^2},\frac{y-y'}{x^2})|\to\infty$, which upon passing to $(w,z,w',z')$-coordinates as above implies the bound~\eqref{EqBAnEll00QMxEl} (where one can in fact replace the Sobolev orders on both sides by arbitrary but fixed numbers).

  Finally, for those $(p',j')$ for which $\chi_{p,j}$ and $\chi_{p',j'}$ have intersecting supports---recall that the number of such $(p',j')$ is bounded independently of $(p,j)$---we can split $\tilde\chi_{p',j'}$ in~\eqref{EqBAnEll00QMxEl} into $\tilde\chi_{p',j'}\chi^\sharp_{p,j}+\tilde\chi_{p',j'}(1-\chi^\sharp_{p,j})$, where $\chi^\sharp_{p,j}=\Phi_{p,j}^*\chi^\sharp$ with $\chi^\sharp\in\CIc(\cD)$ identically $1$ near $\supp\chi$. In the first summand, we pass to $\Phi_{p,j}$-coordinates for both $x,y$ and $x',y'$, obtaining a ps.d.o.\ on $\R^n$ as before; the second summand gives a smoothing operator using the above integration by parts argument, and~\eqref{EqBAnEll00QMxEl} holds (with $|j-j'|$ and $|p-p'|$ bounded independently of $(p,j)$ for the values of $(p',j')$ we are currently considering). This finishes the proof of~\eqref{EqBAnEll00QMxEl}, and thus of~\eqref{EqBAnEll00QBdd}.

  We next turn to the proof of~\eqref{EqBAnEll00QL}. Plugging $f=L(\chi_1 u)$ into~\eqref{EqBAnEll00Quant}, with $L$ (in primed coordinates) a sum of terms $a(x',y')(x'{}^2 D_{x'})^j(x'{}^2 D_{y'})^\alpha$, we integrate by parts in $(x',y')$. When $-D_{x'}x'{}^2=i x'{}^2\pa_{x'}+2 i x'$ falls on the exponential, we obtain a factor of
  \[
    \Bigl(\frac{x'}{x}\Bigr)^2\xi+2 i x'=\xi-\frac{x-x'}{x^2}(x+x')\xi+2 i x'.
  \]
  We regard the first term as the main term; it has the same symbol as $x^2 D_x$. Note that the third summand has a factor of $x'$; and in the second summand, we can rewrite $\frac{x-x'}{x^2}$ as the $D_\xi$-derivative of the exponential followed by integration by parts in $\xi$, with the $\xi$-derivative of $q$ dropping one symbolic order. Another possibility is that $-x'{}^2 D_{x'}$ falls on one of the cutoffs $\phi$ in~\eqref{EqBAnEll00Quant}, resulting in a term with an extra factor of $x'$ and a differentiated cutoff which is thus equal to $0$ near $(\frac{x-x'}{x'},\frac{|y-y'|}{x'})=0$. Derivatives along $-D_{y'}x'{}^2$ are rewritten in a similar manner.

  Altogether then, $Q L(\chi_1 u)$ is the sum of three types of terms: one term which is a quantization as in~\eqref{EqBAnEll00Quant} with symbol $q(x,y;\xi,\eta)\ell(x,y;\xi,\eta)=\chi_2(x,y)+x r(x,y;\xi,\eta)$ where $r$ is a symbol of order $0$; secondly, terms with undifferentiated cutoffs $\phi$, with symbolic order $0$, but with extra factors of $x$ or $x'$; and thirdly, terms in which at least one of the cutoffs $\phi$ is differentiated at least once. It then suffices to note that the quantization of $x r$ as well as terms of the second type are bounded from $x^{-1}H_{0 0}^s\to H_{0 0}^s$ (as follows from the mapping properties we showed above for $Q$---which did not rely on the specific form of $Q$). Moreover, terms of the third type, due to the localization away from the diagonal, are bounded between any two polynomially weighted 00-Sobolev spaces; for present purposes, it suffices to note the weaker statement that upon performing one integration by parts in the momentum variables, one obtains an operator which is a 00-quantization (with slightly enlarged cutoffs $\phi$) whose symbol gains a power of $|(\frac{x-x'}{x^2},\frac{y-y'}{x^2})|^{-1}=x|(\frac{x-x'}{x},\frac{y-y'}{x})|^{-1}$, and thus a power of $x$ since $|(\frac{x-x'}{x},\frac{y-y'}{x})|$ is bounded from below by a positive constant on $\supp\dd(\phi(\frac{x-x'}{x})\phi(\frac{|y-y'|}{x}))$.
\end{proof}

\begin{example}[Exterior derivative]
\label{ExBAnExt}
  Slightly modifying Example~\ref{ExB00Ext}, let $\Omega\subset\R^n$ be a bounded domain with smooth boundary, and let $\rho\in\CI(\ol\Omega)$ be a boundary defining function. For $\beta\in\R$, we consider
  \[
    L:=e^{-\beta/\rho}\rho^2\,\dd\,e^{\beta/\rho}\in\Diff_{0 0}^1(\ol\Omega;\ul\R,T^*\R^n),
  \]
  with principal symbol ${}^{0 0}\upsigma^1(L)(z,\zeta)=i\rho^2\zeta-\beta\,\dd\rho\colon\R\to T^*_z\R^n$ where $\zeta\in{}^{0 0}T^*\ol\Omega$; recall also that $\rho^2\colon{}^{0 0}T^*\ol\Omega\to T^*_{\ol\Omega}\R^n$ is an isomorphism). For $\beta\neq 0$, this is injective. Taking adjoints with respect to the standard volume density and fiber inner products on $\R^n$, we have $L^*=e^{\beta/\rho}\delta_\fe\,\rho^2 e^{-\beta/\rho}$ where $\delta_\fe$ is the (negative) divergence on 1-forms. The operator $L^*L\in\Diff_{0 0}^2(\ol\Omega)$ has an elliptic principal symbol (both at fiber infinity and at $\pa\Omega$). The elliptic estimates proved above therefore give
  \begin{equation}
  \label{EqBAnExtEst}
    \|u\|_{H_{0 0}^s(\ol\Omega)} \leq C\Bigl(\|L^*L u\|_{H_{0 0}^{s-2}(\ol\Omega)} + \|u\|_{x^{-1}H_{0 0}^{s-1}(\ol\Omega)}\Bigr).
  \end{equation}
  Note then that the inclusion $H_{0 0}^s(\ol\Omega)\to x^{-1}H_{0 0}^{s-1}(\ol\Omega)$ is compact; and any $u\in H_{0 0}^s$ with $L^*L u=0$ automatically lies in $\rho^\infty H_{0 0}^\infty(\ol\Omega)$, and therefore (via pairing with $u$ and integrating by parts) satisfies $L u=0$, so $u=c e^{-\beta/\rho}$ for some constant $c\in\C$. For $\beta>0$, we have $e^{-\beta/\rho}\in H_{0 0}^s$ indeed, and we conclude that $L^*L\colon H_{0 0}^s(\ol\Omega)\to H_{0 0}^{s-2}(\ol\Omega)$ has closed range with 1-dimensional kernel and cokernel spanned by $e^{-\beta/\rho}$. Let now $f\in\CIc(\Omega)$ with $\int f(x)\,\dd x=0$; then $e^{\beta/\rho}f\in\ran L^*L$, so there exists $u\in H_{0 0}^\infty(\ol\Omega)$ with $L^*(L u)=e^{\beta/\rho}f$, or equivalently
  \[
    \delta_\fe\omega = f,\qquad \omega=\rho^2 e^{-2\beta/\rho}\rho^2\,\dd(e^{\beta/\rho}u) \in e^{-\beta/\rho}\rho^2 H_{0 0}^\infty(\ol\Omega)\subset\CIdot(\ol\Omega),
  \]
  where we used Corollary~\ref{CorBAn00Sob}. Thus, we have solved the divergence equation with control on the support $\omega$.
\end{example}

\section{Geometry and analysis on the total gluing space}
\label{SG}

We now turn to the gluing problem in earnest. Fix an $n$-dimensional manifold $X$ (without boundary) and a point $\fp\in X$; here $n\geq 3$. We then consider a fibration
\[
  X - \wt X' \to [0,1)
\]
together with a choice of identification of the fiber over $0$ with $X$. For the purpose of doing analysis, it is convenient to immediately fix an identification of the other fibers with $X$ as well; that is, we fix a trivialization
\[
  \wt X' = [0,1) \times X.
\]
We denote the coordinate in the first factor by $\eps$.

\begin{definition}[Total gluing space; tangent bundle]
\label{DefGTot}
  In the above notation, we define
  \[
    \wt X := [\wt X'; \{0\}\times\{\fp\}]
  \]
  as the real blow-up of $\wt X'$ at the point $\fp\in X$ in the fiber over $\eps=0$; we write $\wt\upbeta\colon\wt X\to\wt X'$ for the blow-down map. We write
  \[
    \hat X=\wt\upbeta^*\{(0,\fp)\},\qquad
    X_\circ=\wt\upbeta^*\eps^{-1}(0)
  \]
  for the front face and the side face (the lift of the original boundary), respectively; the restriction of $\wt\upbeta$ to $X_\circ$ is denoted $\upbeta_\circ\colon X_\circ\to X$. For $\eps\in(0,1)$ we define the level set
  \[
    \wt X_\eps := \{\eps\}\times X \subset \wt X,
  \]
  which we also regard as a submanifold of $\wt X'$. We moreover denote by $\wt T\wt X'\to\wt X'$ the vertical tangent bundle, i.e.\ the bundle of tangent vectors which are tangent to the fibers, and by $\wt T\wt X\to\wt X$ the pullback of $\wt T\wt X'$ along $\wt\upbeta$. We write $\wt\cV(\wt X):=\CI(\wt X;\wt T\wt X)$.
\end{definition}

We denote by $\hat\rho$ and $\rho_\circ\in\CI(\wt X)$ defining functions of $\hat X$ and $X_\circ$, respectively; when working in subsets of $\wt X$, we use the same notation for local defining functions (i.e.\ defining functions of $\hat X$ and $X_\circ$ inside the submanifold). 

Concretely, in local coordinates $x=(x^1,\ldots,x^n)\in\R^n$ on $X$, valid for $|x|<r_0$ and with $\fp$ given by $x=0$, a neighborhood of $\hat X\subset\wt X$ is covered by the two coordinate charts
\begin{alignat*}{3}
  &(\eps,\hat x),&\qquad \hat x&:=\frac{x}{\eps}&\qquad&\text{for $\eps\geq 0$, $|\hat x|\lesssim 1$}, \\
  &(\hat\rho,\rho_\circ,\omega),&\qquad \hat\rho&:=|x|,\ \rho_\circ:=\frac{\eps}{|x|},\ \omega:=\frac{x}{|x|}\in\Sph^{n-1} &\qquad&\text{for $\hat\rho\geq 0$, $0\leq\rho_\circ\lesssim 1$}.
\end{alignat*}
We shall see in Lemma~\ref{LemmaGTot}\eqref{ItGTothatX} below that $\hat x$ is a linear coordinate on the interior $\hat X^\circ\cong T_\fp X$ of $\hat X$. See Figure~\ref{FigGTot}.

For any $\eps\in(0,1)$, we have $\wt X_\eps\cong X$ via $(\eps,p)\mapsto p$, and $\wt T_{\wt X_\eps}\wt X=\wt T_{\wt X_\eps}\wt X'=T X$. In $\eps>0$ then, elements $V\in\wt\cV(\wt X)$ are smooth families (in $\eps$) of smooth vector fields $V|_{\wt X_\eps}$ on $\wt X_\eps=X$ which become singular in a specific fashion as $\eps\searrow 0$: in local coordinates $x$ on $X$, they are linear combinations of $\pa_{x^i}$ ($i=1,\ldots,n$) with coefficients which are smooth functions on $\wt X$, i.e.\ smooth functions of $(\eps,\hat x)$, resp.\ $(\hat\rho,\rho_\circ,\omega)$. (Away from $\hat X$, the notions of smoothness on $\wt X$ and $\wt X'$ agree since $\wt X\setminus\wh X=\wt X'\setminus\{(0,\fp)\}$.)

\begin{figure}[!ht]
\centering
\includegraphics{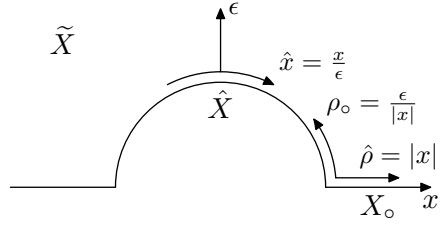}
\caption{Illustration of $\wt X$ and local coordinates $(\eps,\hat x)$ near $\hat X^\circ$ and $(\rho_\circ,\hat\rho,\omega)$ near $\hat X\cap X_\circ$ (with $\omega$ not shown for artistic reasons).}
\label{FigGTot}
\end{figure}

\begin{lemma}[Basic properties of $\wt X$]
\fakephantomsection
\label{LemmaGTot}
  \begin{enumerate}
  \item\label{ItGTotXcirc} We have $\wt T_{X_\circ}\wt X=\upbeta_\circ^*(T X)$. For all $q\in\hat X$, we have $\wt T_q\wt X=T_\fp X$.
  \item\label{ItGTothatX} We have $\hat X\cong\ol{T_\fp}X$, with the isomorphism given by continuous extension of $T_\fp X\ni V\mapsto\lim_{s\searrow 0}(s,\gamma(s))$ where $\gamma$ is a curve in $X$ with $\gamma(0)=\fp$ and $\gamma'(0)=V$. In particular, $\hat X^\circ\cong T_\fp X$ is a vector space.
  \item\label{ItGTotScale} Let $q\in\hat X^\circ$. Then the map assigning to $V\in\wt T_q\wt X$ the restriction $(\eps\wt V)|_q$, where $\wt V\in\wt\cV(\wt X)$ satisfies $\wt V(q)=V$, defines an isomorphism $\wt T_q\wt X\xra{\cong} T_q\hat X$ which extends to an isomorphism $\wt T_{\hat X}\wt X\xra{\cong}\Tsc\hat X$.
  \item\label{ItGTotb} We have $\wt\cV(\wt X)\hra\hat\rho^{-1}\Vb(\wt X)$.
  \end{enumerate}
\end{lemma}
\begin{proof}
  The first part follows directly from the definitions. For the second part, write $V=V^i\pa_{x^i}$ in local coordinates centered at $\fp$; then $\gamma(s)=(s V^1,\ldots,s V^n)+\cO(s^2)$, and therefore $s\mapsto(s,\gamma(s))$ in $(\eps,\hat x)$-coordinates is given by $s\mapsto(\eps,\hat x)=(s,(V^1,\ldots,V^n)+\cO(s))$. This establishes an isomorphism $T_\fp X\cong\hat X^\circ$. Passing to inverse polar coordinates $(\rho_T,\omega_T):=|(V^1,\ldots,V^n)|^{-1}(1,(V^1,\ldots,V^n))$ and $(\rho_\circ,\omega)$, $\rho_\circ=|\hat x|^{-1}$, near $\pa\ol{T_\fp}X$ and $\pa\hat X$, respectively, this map is given by $(\rho_T,\omega_T)\mapsto(\rho_\circ,\omega)=(\rho_T,\omega_T)$ and thus extends to a diffeomorphism as claimed.

  In part~\eqref{ItGTotScale}, one assigns to $V^i\pa_{x^i}=\eps V^i\pa_{\hat x^i}$ the tangent vector $V^i\pa_{\hat x^i}$. It then remains to note that $\pa_{\hat x^i}$ is a frame of $\Tsc\hat X$. Part~\eqref{ItGTotb} finally follows from the above away from the corner $X_\circ\cap\hat X$ of $\wt X$; near the corner on the other hand, note that $V=|x|\pa_{x^i}=\hat\rho\pa_{x^i}$ is a smooth b-vector field on $X_\circ$, thus its lift to $\wt X$ is a linear combination of $\hat\rho\pa_{\hat\rho}$, $\rho_\circ\pa_{\rho_\circ}$, $\pa_\omega$ with $\CI([0,1)_{\hat\rho}\times\Sph^{n-1}_\omega)$-coefficients. Therefore, $\pa_{x^i}=\hat\rho^{-1}V\in\hat\rho^{-1}\Vb(\wt X)$, as claimed.
\end{proof}

\begin{definition}[Scaling]
\label{DefGTotScale}
  For $p,q\in\N_0$, we write $\sfs\colon\wt T_{\wt X^\circ}^{p,q}\wt X\to\wt T_{\wt X^\circ}^{p,q}\wt X$ for multiplication by $\eps^{q-p}$. By an abuse of notation, we denote the direct sum of several such maps (for various values of $(p,q)$) by $\sfs$ as well. We write $\hat\sfs\colon\Tsc^{p,q}\hat X\to\wt T_{\hat X}\wt X$ for the restriction of $\sfs$ to $\hat X$.
\end{definition}

The map $\hat\sfs$ is well-defined and indeed a bundle isomorphism by Lemma~\ref{LemmaGTot}\eqref{ItGTotScale}. For example, $\sfs$ is division by $\eps$ on tangent vectors (mapping $\pa_{\hat x^i}\mapsto\pa_{x^i}$ in the above local coordinates), and multiplication by $\eps^2$ on symmetric 2-tensors (mapping $\dd\hat x^i\,\dd\hat x^j\mapsto\dd x^i\,\dd x^j$).

When relating objects (such as functions or differential operators) on $\wt X$ to their restrictions to $\hat X$ or $X_\circ$, the following geometric result will be useful; we work with local coordinates $x\in\R^n$ on $X$ near $\fp$, and with $\hat x=\frac{x}{\eps}$.

\begin{lemma}[Relationships between parameterized spaces]
\label{LemmaGTotRel}
  The identity map on $\wt X'$ lifts to a diffeomorphism
  \[
    [\wt X; [0,1)_\eps\times\{\fp\} ] \xra{\cong} \bigl[ [0,1)_\eps\times X_\circ; \{0\}\times\pa X_\circ \bigr],
  \]
  and also, near $\{0\}\times\{\fp\}$, to a diffeomorphism
  \begin{equation}
  \label{EqGTotRelHat}
    \wt X \cap \{ |x|<r_0 \} \xra{\cong} \bigl[ [0,1)_\eps\times\hat X; \{0\}\times\pa\hat X \bigr] \cap \{ \eps|\hat x|<r_0 \}.
  \end{equation}
\end{lemma}
\begin{proof}
  This can be checked directly by covering the spaces on both sides with coordinate charts and writing down the identity map on $\wt X'$ with respect to those; see Figure~\ref{FigGTotRelHat} for the diffeomorphism~\eqref{EqGTotRelHat}. Alternatively, the first diffeomorphism arises from the commutation result \cite[Proposition~5.8]{MelroseDiffOnMwc} for iterated blow-ups, which gives
  \begin{align*}
    [\wt X;[0,1)\times\{\fp\}] &= [ [0,1)\times X; \{0\}\times\{\fp\}; [0,1)\times\{\fp\} ] \\
      &\cong [ [0,1)\times X; [0,1)\times\{\fp\}; \{0\}\times\{\fp\}] \\
      & = [[0,1)\times[X;\{\fp\}]; \{0\}\times\pa X_\circ ],
  \end{align*}
  where we used in the last line that the lift of $\{0\}\times\{\fp\}$ to $[0,1)\times[X;\{\fp\}]=[0,1)\times X_\circ$ is $\{0\}\times\pa X_\circ$.
\end{proof}

\begin{figure}[!ht]
\centering
\includegraphics{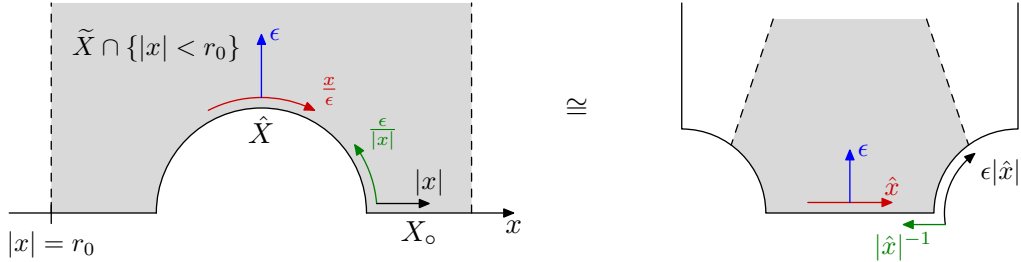}
\caption{The diffeomorphism~\eqref{EqGTotRelHat} in local coordinates. Matching coordinates are indicated with matching colors.}
\label{FigGTotRelHat}
\end{figure}

Let now
\begin{equation}
\label{EqGTotCutoffs}
  \hat\chi,\ \chi_\circ\in\CI(\wt X)
\end{equation}
be identically $1$ near, and supported in, a small collar neighborhood of $\hat X$ and $X_\circ$, respectively. Then via Lemma~\ref{LemmaGTotRel}, we can extend any given function $f\in\CI(X_\circ)$ to an $\eps$-independent function on $[0,1)\times X_\circ$, lift it to the resolution, and upon localizing with $\chi_\circ$ finally define $\chi_\circ f$ as a smooth function on $\wt X$ (given that $\supp\chi_\circ$ is disjoint from the lift of $x=0$), with support near $X_\circ$; similarly for the extension of functions from $\hat X$. In the same manner, a differential operator such as $a(x)D_x^\alpha$ on $X_\circ$, resp.\ $a(\hat x)D_{\hat x}^\alpha$ on $\hat X$, lifts upon localization to a differential operator $\chi_\circ a(x)D_x^\alpha$, resp.\ $\hat\chi a(x/\eps)(\eps D_x)^\alpha$ on $\wt X$ (using coordinates $\eps>0$ and $x\in\R^n$).

\subsection{Analysis on the total gluing space}
\label{SsGq}

As we shall see in~\S\ref{SsCETot}, the degeneration of the linearized constraints maps associated with a metric and second fundamental form which are (weighted) sections of $S^2\wt T^*\wt X$ can be captured precisely using the terminology of \emph{q-analysis}, which was introduced in \cite[\S2.1]{HintzKdSMS} as a close relative of the analytic surgery calculus of McDonald \cite{McDonaldThesis} and Mazzeo--Melrose \cite{MazzeoMelroseSurgery}. Using the notation of Definition~\ref{DefGTot}, we recall the Lie algebra of \emph{q-vector fields}
\[
  \Vq(\wt X) = \{ V\in\Vb(\wt X) \colon V\eps=0 \},
\]
and the associated space $\Diffq^m(\wt X)$ of $m$-th order q-differential operators; thus, elements of $\Vq(\wt X)$ are families of vector fields on $X$ which depend smoothly on $\eps\in(0,1)$ and degenerate in a specific manner as $\eps\searrow 0$. The restriction of q-vector fields to the boundary hypersurfaces of $\wt X$ induces surjective normal operator maps
\begin{equation}
\label{EqGqNorm}
  N_{\hat X} \colon \Vq(\wt X) \to \Vb(\hat X), \qquad
  N_{X_\circ} \colon \Vq(\wt X) \to \Vb(X_\circ),
\end{equation}
with kernels $\hat\rho\Vq(\wt X)$ and $\rho_\circ\Vq(\wt X)$, respectively. These extend multiplicatively to normal operator homomorphisms
\[
  N_{\hat X} \colon \Diffq(\wt X;E) \to \Diffb(\hat X;E|_{\hat X}),\qquad
  N_{X_\circ} \colon \Diffq(\wt X;E) \to \Diffb(X_\circ;E|_{X_\circ}),
\]
when $E\to\wt X$ is a vector bundle. We remark that if $E=\wt T^{p,q}\wt X$ and $A\in\Diffq(\wt X;E)$, then $\hat\sfs^{-1}\circ N_{\hat X}(A)\circ\hat\sfs$, resp.\ $N_{X_\circ}(A)$ is an element of $\Diffb(\hat X;\Tsc^{p,q}\hat X)$, resp.\ $\Diffb(X_\circ;\upbeta_\circ^* T^{p,q}X)$; cf.\ Lemma~\ref{LemmaGTot}\eqref{ItGTotXcirc}, \eqref{ItGTotScale}. We furthermore note that Lemma~\ref{LemmaGTot}\eqref{ItGTotb} gives $\wt\upbeta^*\cV(X)\hra\hat\rho^{-1}\Vq(\wt X)$, where on the left we regard vector fields on $X$ as $\eps$-independent vector fields on $\wt X'$.

In local coordinates $x\in\R^n$ on $X$, with $x=0$ at $\fp$, we can take $\hat\rho=(\eps^2+|x|^2)^{1/2}=\eps\la\hat x\ra$ where $\hat x=\frac{x}{\eps}$, and the space $\Vq(\wt X)$ is then spanned over $\CI(\wt X)$ by the vector fields
\begin{equation}
\label{EqGqVFs}
  \hat\rho\pa_{x^j}=\la\hat x\ra\pa_{\hat x^j}\qquad (j=1,\ldots,n),
\end{equation}
whose $\hat X$-, resp.\ $X_\circ$-normal operators are given by $\la\hat x\ra\pa_{\hat x^j}$ (cf.\ Lemma~\ref{LemmaGTot}\eqref{ItGTotScale}), resp.\ $|x|\pa_{x^j}$. (This can be used to prove the stated properties of the maps~\eqref{EqGqNorm}.) The vector fields~\eqref{EqGqVFs} form a frame of the \emph{q-tangent bundle} $\Tq\wt X\to\wt X$, the smooth sections of which are exactly the q-vector fields. The principal symbol of $A\in\Diffq^k(\wt X)$ is then an element
\[
  \sigmaq^k(A) \in P^k_{\rm hom}(\Tq^*\wt X),
\]
defined analogously to the b-principal symbol in~\S\ref{SsBb}.

Turning to function spaces, suppose that $X$ is compact without boundary, and $\wt X=[[0,1)\times X;\{0\}\times\{\fp\}]$ as usual. Fix a smooth positive q-density $\mu$ on $\wt X$, which in local coordinates as above is thus a smooth positive multiple of $(\eps^2+|x|^2)^{-n/2}|\dd x|=\la\hat x\ra^{-n}|\dd\hat x|$. For $\eps\in(0,1)$, we then define $H_{\qop,\eps}^0(X)$ as the $L^2$-space on $X$ with respect to the smooth density $\mu|_{\wt X_\eps}$ on $\wt X_\eps\cong X$. Fixing moreover a collection $\{V_1,\ldots,V_N\}\subset\Vq(\wt X)$ of q-vector fields which spans $\Vq(\wt X)$ over $\CI(\wt X)$, we define for $s\in\N_0$ and $\eps\in(0,1)$
\begin{equation}
\label{EqGqHqNorm}
  \|u\|_{H_{\qop,\eps}^s(X)}^2 := \sum_{\genfrac{}{}{0pt}{}{\alpha\in\N_0^N}{|\alpha|\leq s}} \|(V_1,\ldots,V_N)^\alpha u\|_{H_{\qop,\eps}^0(X)}^2.
\end{equation}
This is equivalent to the norm on $H^s(X)$ for any fixed $\eps>0$, but not uniformly equivalent when $\eps\searrow 0$. Thus, we define $H_{\qop,\eps}^s(X)=H^s(X)$ with the norm~\eqref{EqGqHqNorm}. Via interpolation and duality, we can define $H_{\qop,\eps}^s(X)$ as a Hilbert space for $s\in\R$; we shall also consider weighted versions $\hat\rho^\alpha\rho_\circ^\beta H_{\qop,\eps}^s(X)$, where $\alpha,\beta\in\R$. Any fixed element $A\in\Diffq^m(\wt X)$ defines a uniformly (in $\eps$) bounded linear map $H_{\qop,\eps}^s(X)\to H_{\qop,\eps}^{s-m}(X)$, similarly for weighted spaces.

In a collar neighborhood of $\hat X$, resp.\ $X_\circ$, we can relate q-Sobolev spaces to b-Sobolev spaces (defined with respect to positive smooth b-densities on $\hat X$ and $X_\circ$); to wit, in terms of the maps
\begin{equation}
\label{EqGqPhis}
\begin{alignedat}{2}
  \hat\phi(\eps,\hat x)&:=(\eps,\eps\hat x),&\qquad \hat x&\in\hat X^\circ=\R^n, \\
  \phi_\circ(\eps,x)&:=(\eps,x),& \qquad x&\in(X_\circ)^\circ=X\setminus\{\fp\},
\end{alignedat}
\end{equation}
and recalling the cutoffs $\hat\chi$, $\chi_\circ$ from~\eqref{EqGTotCutoffs}, we have uniform (in $\eps\in(0,1)$) norm equivalences
\begin{equation}
\label{EqGqHqNormEquivPre}
\begin{split}
  \|\hat\chi u\|_{\hat\rho^\alpha\rho_\circ^\beta H_{\qop,\eps}^s(X)} &\sim \eps^{-\alpha}\|\hat\phi^*(\hat\chi u)\|_{\rho_\circ^{\beta-\alpha}\Hb^s(\hat X)}, \\
  \|\chi_\circ u\|_{\hat\rho^\alpha\rho_\circ^\beta H_{\qop,\eps}^s(X)} &\sim \eps^{-\beta}\|\phi_\circ^*(\chi_\circ u)\|_{\hat\rho^{\alpha-\beta}\Hb^s(X_\circ)},
\end{split}
\end{equation}
i.e.\ the left hand side is bounded by a constant (independent of $\eps$ and $u$) times the right hand side, and vice versa. See \cite[Proposition~2.13]{HintzKdSMS} for a pseudodifferential proof (though using differently weighted densities); alternatively, for $s\in\N_0$, the equivalences~\eqref{EqGqHqNormEquivPre} follow from the fact that lifts of b-vector fields on $\hat X$, resp.\ $X_\circ$, span the space of q-vector fields near $\supp\hat\chi$, resp.\ $\supp\chi_\circ$, cf.\ the local coordinate descriptions above. For general $s$ one then uses interpolation and duality. We can combine the two equivalences in~\eqref{EqGqHqNormEquivPre} to the equivalence
\begin{equation}
\label{EqGqHqNormEquiv}
  \|u\|_{\hat\rho^\alpha\rho_\circ^\beta H_{\qop,\eps}^s(X)} \sim \eps^{-\alpha}\|\hat\phi^*(\hat\chi u)\|_{\rho_\circ^{\beta-\alpha}\Hb^s(\hat X)} + \eps^{-\beta}\|\phi_\circ^*((1-\hat\chi)u)\|_{\hat\rho^{\alpha-\beta}\Hb^s(X_\circ)}
\end{equation}
upon (as we may) choosing $\hat\chi,\chi_\circ$ such that $\chi_\circ=1$ on $\supp(1-\hat\chi)$. (In the second term, the power of $\hat\rho$ is arbitrary since $1-\hat\chi=0$ near $\hat\rho=0$. One has an analogous norm equivalence upon replacing $\hat\chi$, $1-\hat\chi$ by $1-\chi_\circ$, $\chi_\circ$, respectively; in this case the power of $\rho_\circ$ in the first term is arbitrary.) To each of the two summands, the analysis of~\S\ref{SsBAn} applies; thus we obtain a bounded geometry perspective on q-analysis. This immediately gives (elliptic) estimates for q-differential operators using the arguments in~\S\ref{SsBAn}.

\section{The constraints map and its linearization}
\label{SCE}

Let $\Lambda\in\R$. We fix an $n$-dimensional manifold $X$, $n\geq 3$, and define, for a $\cC^2$ Riemannian metric $\gamma$ and a $\cC^1$ symmetric 2-tensor $k$ on $X$, the tensor
\begin{equation}
\label{EqCEOp}
\begin{alignedat}{2}
  P(\gamma,k;\Lambda) &:= \bigl(P_1(\gamma,k;\Lambda),&\ &P_2(\gamma,k)\bigr) \\
    &:= \bigl(R_\gamma-|k|_\gamma^2+(\tr_\gamma k)^2-2\Lambda,&\ &\delta_\gamma k+\dd(\tr_\gamma k) \bigr) \in \cC^0(X;\ul\R\oplus T^*X),
\end{alignedat}
\end{equation}
where $\ul\R=X\times\R\to X$ is the trivial bundle. The operator $P$ is called the \emph{constraints map}, and the constraint equations~\eqref{EqICE} read $P(\gamma,k;\Lambda)=0$.

\begin{lemma}[Linearized constraints map]
\label{LemmaCELin}
  The linearization of $P$ around $(\gamma,k)$, defined by
  \[
    L_{\gamma,k}(h,q) := D_{(\gamma,k)}P(h,q;\Lambda) = \frac{\dd}{\dd s}P(\gamma+s h,k+s q;\Lambda)|_{s=0},
  \]
  takes the block form
  \begin{equation}
  \label{EqCELin}
     L_{\gamma,k}=\openbigpmatrix{6pt}
          \Delta_\gamma\tr_\gamma{+}\delta_\gamma\delta_\gamma{+}\la {-}\Ric(\gamma){+}2 k\circ k{-}2(\tr_\gamma k)k,-\ra & -2\la k,-\ra {+} 2(\tr_\gamma k)\tr_\gamma \\
          \la\nabla k,-\ra_{2 3}{-}\la\nabla k,-\ra_{1 2}{-}\frac12\la k,\nabla(-)\ra_{1 2}{-}\la k,(\delta_\gamma{+}\frac12\dd\tr_\gamma)(-)\ra_1 & \delta_\gamma{+}\dd\tr_\gamma
        \closebigpmatrix
  \end{equation}
  where we write $A\circ B$ (for $A,B$ sections of $T^*X\otimes T^*X$) for the section of $T^*X\otimes T^*X$ given by $(A\circ B)_{i j}=A_{i\ell}B^\ell{}_j$; moreover, for tensors $T\in(T^*X)^p$ and $S\in(T^*X)^q$ with $p<q$, and for $1\leq i_1<\ldots<i_p\leq q$, we write $\la T,S\ra_{i_1\ldots i_p}=\la S,T\ra_{i_1\ldots i_p}\in(T^*X)^{q-p}$ for the contraction of the $(i_1,\ldots,i_p)$-components of $S$ with $T$. All inner products and covariant derivatives are with respect to $\gamma$, and we write $(\nabla k)_{i j \ell}=k_{i j;\ell}$. Finally, $\Delta_\gamma=-\tr_\gamma\nabla^2$. The adjoint is
  \begin{align*}
    L_{\gamma,k}^*&=
      \left(
      \begin{matrix}
        \gamma\Delta_\gamma{+}\delta_\gamma^*\dd{+}({-}\Ric(\gamma){+}2 k\circ k{-}2(\tr_\gamma k)k) \\
        {-}2 k{+}2\gamma(\tr_\gamma k)
      \end{matrix}
      \right. \quad \cdots \\
      &\quad\hspace{6em} \cdots\quad \left.
        \begin{matrix}
          {-}\frac12\la\nabla k,-\ra_3 {-} \frac12 k\delta_\gamma {-} k\circ\nabla(-){-}\frac12\gamma\la\delta_\gamma k,-\ra{+}\frac12\gamma\la k,\nabla(-)\ra \\
          \delta_\gamma^*{+}\gamma\delta_\gamma
        \end{matrix}
        \right).
  \end{align*}
\end{lemma}
\begin{proof}
  For the formula for $L_{\gamma,k}$, see \cite{FischerMarsdenLinStabEinstein} (which uses the momentum tensor $\pi=(\tr_\gamma k)\gamma-k$ instead of $k$), or \cite[Equation~(2.2)]{ChruscielDelayMapping} (where the opposite sign convention for the Laplacian is used). The expression for the adjoint follows from a short calculation.
\end{proof}

As observed in~\cite{ChruscielDelayMapping}, $L_{\gamma,k}$ is a Douglis--Nirenberg system of operators \cite{DouglisNirenbergElliptic}, with
\begin{equation}
\label{EqCEDNOrders}
  (L_{\gamma,k})_{i j}\in\Diff^{t_j+s_i},\qquad t_1=0,\ t_2=0,\ \ s_1=2,\ s_2=1,
\end{equation}
and thus $(L_{\gamma,k}^*)_{i j}\in\Diff^{s_j+t_i}$; and indeed
\begin{equation}
\label{EqCEDNOrders2}
  L_{\gamma,k} \in \begin{pmatrix} \Diff^2 & \Diff^0 \\ \Diff^1 & \Diff^1 \end{pmatrix} \subset \begin{pmatrix} \Diff^2 & \Diff^2 \\ \Diff^1 & \Diff^1 \end{pmatrix},\qquad
  L_{\gamma,k}^* \in \begin{pmatrix} \Diff^2 & \Diff^1 \\ \Diff^0 & \Diff^1 \end{pmatrix}.
\end{equation}
At $\xi\in T_x^*X$, the principal symbol of $L_{\gamma,k}$, which is a map $S^2 T_x^*X\oplus T_x^*X\to\R\oplus T_x^*X$ in the Douglis--Nirenberg sense, is
\[
  \upsigma(L_{\gamma,k})(x,\xi) = \begin{pmatrix} |\xi|_{\gamma^{-1}}^2\tr_\gamma-\,\iota_{\xi^\sharp}\iota_{\xi^\sharp} & 0 \\ [\cdots] & -i\iota_{\xi^\sharp}+i\xi\tr_\gamma \end{pmatrix},
\]
where $\iota_{\xi^\sharp}$ is the contraction with $\xi^\sharp=\gamma^{-1}(\xi,-)$, and we do not write out the off-diagonal term explicitly. The principal symbol of the adjoint is
\begin{equation}
\label{EqCELstarSymb}
  \upsigma(L_{\gamma,k}^*)(x,\xi) = \begin{pmatrix} \gamma|\xi|_{\gamma^{-1}}^2-\xi\otimes\xi & [\cdots] \\ 0 & i\xi\otimes_s(-)-i\gamma\iota_{\xi^\sharp} \end{pmatrix}.
\end{equation}
An important feature of the order convention~\eqref{EqCEDNOrders} (specifically, the fact that all $t_j$ are equal) is that the Douglis--Nirenberg principal symbol of
\begin{equation}
\label{EqCELLstar}
  L_{\gamma,k}L_{\gamma,k}^* \in \begin{pmatrix} \Diff^4 & \Diff^3 \\ \Diff^3 & \Diff^2 \end{pmatrix} = (\Diff^{s_j+s_i})_{i,j=1,2},
\end{equation}
i.e.\ the matrix of principal symbols, is equal to the product $\upsigma(L_{\gamma,k})\upsigma(L_{\gamma,k}^*)$. The following result thus shows that~\eqref{EqCELLstar} is elliptic:

\begin{lemma}[Left ellipticity of the adjoint]
\label{LemmaCELstarInj}
  For $\xi\in T^*_x X\setminus o$, the linear map
  \[
    \upsigma(L_{\gamma,k}^*)(x,\xi)\colon\R\oplus T_x^*X\to S^2 T^*_x X\oplus T^*_x X
  \]
  is injective.
\end{lemma}
\begin{proof}
  This is proved in \cite[Lemma~2.3]{ChruscielDelayMapping}: if $(f^*,j^*)\in\R\oplus T^*_x X$ lies in the kernel of $\upsigma(L_{\gamma,k}^*)(x,\xi)$, then $\xi\otimes_s j^*-\gamma\la\xi,j^*\ra=0$, so upon taking traces $(1-n)\la\xi,j^*\ra=0$; this then gives $\xi\otimes_s j^*=0$, and therefore $j^*=0$. Moreover, the fact that $\gamma|\xi|_{\gamma^{-1}}^2-\xi\otimes\xi\neq 0$ for $\xi\neq 0$ (this being the difference of a rank $n$ and a rank $1$ operator, which thus has rank $\geq n-1>0$) forces $f^*=0$.
\end{proof}

\begin{lemma}[Prolongation of $L_{\gamma,k}^*$]
\label{LemmaCEProlong}
  Fix a Riemannian metric $\gamma\in\cC^2(X;S^2 T^*X)$ and a symmetric 2-tensor $k\in\cC^1(X;T^*X)$. Let $\alpha\colon[0,1]\to X$ be a $\cC^2$ path. Then there exists a $\cC^0$ bundle endomorphism $F$ on $(\,(\ul\R\oplus T^*X)\oplus(T^*X\oplus(T^*X\otimes T^*X))\oplus(S^2 T^*X\oplus (S^2 T^*X\oplus(T^*X\otimes S^2 T^*X)))\,)|_{\alpha([0,1])}$ with the following property: for all $f^*\in\cC^2(X)$ and $j^*\in\cC^2(X;T^*X)$, and writing $(h^*,q^*)=L_{\gamma,k}^*(f^*,j^*)\in\cC^0(X;S^2 T^*X)\oplus\cC^0(X;S^2 T^*X)$, the tensor $Z(t):=(f^*,\nabla f^*,j^*,\nabla j^*)|_{\alpha(t)}$ satisfies the ODE
  \[
    \frac{D Z}{\dd t}=F\bigl(Z,(h^*,(q^*,\nabla q^*))\bigr).
  \]
  In particular, if $Z(0)=0$ and if $h^*,q^*,\nabla q^*$ vanish on $\alpha([0,1])$, then $Z(t)=0$ for $t\in[0,1]$.

  Similarly, increasing the regularity requirements on $\gamma,k,\alpha$ by one order, the tensor $\tilde Z(t):=(f^*,\nabla f^*,\nabla^2 f^*,j^*,\nabla j^*)|_{\alpha(t)}$ satisfies an ODE of the form
  \[
    \frac{D\tilde Z}{\dd t}=\tilde F\bigl(\tilde Z,(h^*,\nabla h^*,q^*,\nabla q^*)\bigr).
  \]
\end{lemma}
\begin{proof}
  This is an extension of analogous results for the adjoint of the linearization of the scalar curvature operator (see \cite[Proposition~2.3 and Corollary~2.4]{CorvinoScalar}, following \cite[Proof of Theorem~1]{FischerMarsdenDeformation}, i.e.\ the special case $k=0$, $j^*=0$, to the adjoint of the full linearized constraints map. We give an ad hoc argument, inspired by the setup of \cite{DelayCompact}; see \cite{BransonCapEastwoodGoverProlongations} for a systematic approach for general systems of overdetermined (semi)linear equations.

  Let us write the $(i,j)$ entry of $L_{\gamma,k}^*$ as $L_{i j}^*$. First, note that $\gamma\Delta_\gamma+\delta_\gamma^*\dd=(1-\gamma\tr_\gamma)\nabla^2=(1-\frac{\gamma}{n-1}\tr_\gamma)^{-1}\nabla^2$. The equation $L_{1 1}^*f^*=h^*-L_{1 2}^*j^*$ thus implies, upon applying $(1-\frac{\gamma}{n-1}\tr_\gamma)$ to both sides and putting the terms arising from $(-\Ric(\gamma)+2 k\circ k-2(\tr_\gamma k)k)f^*$ (the coefficient of $f^*$ lying in $\cC^0$) on the right hand side, an equation
  \[
    \nabla^2 f^*=F_1(f^*,j^*,\nabla j^*,h^*),
  \]
  where $F_1$ is linear and only depends on $\gamma,k$; note here that $L_{1 2}^*$ is a first order operator (with $\cC^0$ coefficients) acting on $j^*$.

  Consider next the equation $L_{2 2}^*j^*=q^*-L_{2 1}^*f^*$, and note that $L_{2 1}^*$ is an operator of order zero (i.e.\ a bundle map) and of class $\cC^1$. Since $L_{2 2}^*=\delta_\gamma^*+\gamma\delta_\gamma=(1-\gamma\tr_\gamma)\delta_\gamma^*$, this can be written in the form
  \[
    2\delta_\gamma^*j^* = 2\Bigl(1-\frac{\gamma}{n-1}\tr_\gamma\Bigr)(q^*-L_{2 1}^*f^*) =: \tilde F_2(f^*,q^*),
  \]
  with $\tilde F_2$ of class $\cC^1$. We now claim that we can express $\nabla^2 j^*$ in terms of $j^*$, $f^*$, $q^*$ (through $\tilde F_2(f^*,q^*)$), and $\nabla f^*$, $\nabla q^*$. To this end, let $t\in[0,1]$ and take $V,W,Y$ to be coordinate vector fields in normal coordinates on $X$ centered at $\alpha(t)$; thus, $V,W,Y$ are $\cC^1$. We then compute at $\alpha(t)$:
  \begin{align*}
    \la\nabla_V\nabla_W j^*,Y\ra &= V\la\nabla_W j^*,Y\ra \\
      &= V\bigl( 2 (\delta_\gamma^*j^*)(W,Y) - \la W,\nabla_Y j^*\ra \bigr) \\
      &= V\bigl(\tilde F_2(f^*,q^*)(W,Y)\bigr) - \la\nabla_V\nabla_Y j^*,W\ra \\
      &= -\la \nabla_Y\nabla_V j^*,W\ra + \la[\nabla_Y,\nabla_V]j^*,W\ra + V\bigl(\tilde F_2(f^*,q^*)(W,Y)\bigr).
  \end{align*}
  The second and third terms are lower order terms in the following sense: the second summand on the right can be written in terms of the Riemann curvature tensor of $\gamma$ and is thus tensorial in $j^*$; the third term only involves up to first derivatives of $f^*,q^*$. We can then repeat the same procedure for the first term, which is thus equal to $\la\nabla_W\nabla_Y j^*,V\ra$ modulo lower order terms, and this in turn is equal to $-\la\nabla_V\nabla_W j^*,Y\ra$ modulo lower order terms. Rearranging the resulting equation $\la\nabla_V\nabla_W j^*,Y\ra=-\la\nabla_V\nabla_W j^*,Y\ra+\text{(l.o.t)}$ expresses $\la(\nabla^2 j^*)(V,W),Y\ra$ in terms of $j^*$, $f^*$, $q^*$, $\nabla f^*$, $\nabla q^*$, as claimed.

  In summary, we now have a closed system
  \begin{equation}
  \label{EqCEProlongPf}
    \nabla^2 f^*=F_1(f^*,j^*,\nabla j^*,h^*),\qquad
    \nabla^2 j^*=F_2(f^*,\nabla f^*,j^*,q^*,\nabla q^*),
  \end{equation}
  so a fortiori $(\nabla^2 f^*,\nabla^2 j^*)=\tilde F(f^*,\nabla f^*,j^*,\nabla j^*,h^*,q^*,\nabla q^*)$, with $\tilde F$ only depending on $\gamma,\pa\gamma,\pa^2\gamma,k,\pa k$ (and $\alpha$). Since
  \[
    \frac{D}{\dd t}(f^*,j^*)=(\nabla f^*,\nabla j^*)(\gamma'(t)),\quad
    \frac{D}{\dd t}(\nabla f^*,\nabla j^*)=\tilde F(f^*,\nabla f^*,j^*,\nabla j^*,h^*,q^*,\nabla q^*)(\gamma'(t)),
  \]
  this implies the first part of the Lemma.

  The final statement follows by taking the covariant derivative of the first equation in~\eqref{EqCEProlongPf}; this expresses $\nabla^3 f^*$ in terms of $f^*,\nabla f^*,j^*,\nabla j^*,\nabla^2 j^*,\nabla h^*$, with $\nabla^2 j^*$ itself expressable via the second equation in~\eqref{EqCEProlongPf}. This implies the desired ODE for $\tilde Z(t)$.
\end{proof}

In preparation for the analysis of $P$ on the total gluing space $\wt X$, we develop some preliminary material on the (linearized) constraint equations on the two model spaces $\hat X$ and $X_\circ$ in~\S\ref{SsCEAf} and \S\ref{SsCEPct}, respectively. In~\S\ref{SsCETot}, we study $P$, its linearization, and its model operators on $\wt X$.

\subsection{Asymptotically flat initial data}
\label{SsCEAf}

Let\footnote{The somewhat cumbersome notation involving hats and circles is used here for the sake of consistency with later sections.} $\hat K\subset\R^n$ be compact (possibly empty), smoothly bounded, and having connected complement; set
\[
  \hat X := \ol{\R^n}\setminus\hat K^\circ.
\]
All tensors and coefficients below will be required to be smooth down to $\pa\hat K\subset\hat X$; polyhomogeneity, or b- or scattering behavior, always refer to behavior at $\pa\ol{\R^n}\subset\hat X$. For example, we write $\Tsc\hat X=\Tsc_{\hat X}\ol{\R^n}$. We write $\hat x=(\hat x^1,\ldots,\hat x^n)$ for the standard coordinates on $\R^n$, and let $\hat r=|\hat x|$. Let $\rho_\circ\in\CI(\hat X)$ denote a function which equals $\ubar\rho_\circ:=\hat r^{-1}$ for $\hat r>1$.

\begin{definition}[Asymptotic flatness]
\label{DefCEAf}
  Let $\cE\subset\C\times\N_0$ be a nonlinearly closed index set with $\Re\cE>0$. We then call a pair $(\hat\gamma,\hat k)$ of symmetric 2-tensors on $\R^n\setminus\hat K^\circ$ \emph{$\cE$-asymptotically flat} if
  \[
    \hat\gamma-\hat\fe\in\cA_\phg^\cE(\hat X\setminus\hat K^\circ;S^2\,\Tsc^*\hat X),\qquad
    \hat k\in\cA_\phg^{\cE+1}(\hat X\setminus\hat K^\circ;S^2\,\Tsc^*\hat X),
  \]
  where $\hat\fe=\sum_{j=1}^n(\dd\hat x^j)^2$ is the Euclidean metric.
\end{definition}

In particular, $\hat\gamma=\hat\fe+\cO(\hat r^{-\eps})$ and $\hat k=\cO(\hat r^{-1-\eps})$ for any $\eps<\Re\cE$ as $\hat r\to\infty$. Since nothing is required of the index set $\cE$ except for $\Re\cE>0$, the data $(\hat\gamma,\hat k)$, apart from having strong regularity near infinity, may have rather weak decay. A more common notion of asymptotic flatness requires $\cO(\hat r^{-q})$ decay for $\hat\gamma-\hat\fe$ (and up to 2 b-derivatives thereof), $\cO(\hat r^{-q-1})$ decay for $\hat k$ (and its first b-derivative), and $\cO(\hat r^{-q_0})$ decay for the mass and current densities $P_1(\hat\gamma,\hat k;0)$ and $P_2(\hat\gamma,\hat k)$, where $q>\frac{n-2}{2}$ and $q>n$, cf.\ \cite[Definition~B.7]{CarlottoConstraints}; this ensures the finiteness of the ADM mass and momentum. See also Remark~\ref{RmkGlTPMT}.

We furthermore remark that whether or not $(\hat\gamma,\hat k)$ satisfies the constraint equations is immaterial for the results in this section. (By Lemma~\ref{LemmaCEAfMap}\eqref{ItCEAfMapNonlin} below, $(\hat\gamma,\hat k)$ can satisfy the constraint equations only for $\Lambda=0$.)

\begin{definition}[Weights]
\label{DefCEAfWeight}
  We define
  \[
    \hat w := \begin{pmatrix} \rho_\circ^{-2} & 0 \\ 0 & \rho_\circ^{-1} \end{pmatrix} \in \CI\bigl(\hat X^\circ;\End(S^2\,\Tsc^*\hat X\oplus S^2\,\Tsc^*\hat X)\bigr),
  \]
  and similarly $\hat{\ubar w}=\diag(\ubar\rho_\circ^{-2},\ubar\rho_\circ^{-1})$ on $[0,\infty)_{\ubar\rho_\circ}\times\Sph^{n-1}=\ol{\R^n}\setminus\{0\}$.
\end{definition}

\begin{lemma}[Constraints map in the asymptotically flat case]
\label{LemmaCEAfMap}
  Let $(\hat\gamma,\hat k)$ be $\cE$-as\-ymp\-tot\-i\-cal\-ly flat. Let $(\hat f,\hat j)=P(\hat\gamma,\hat k;0)$. Then:
  \begin{enumerate}
  \item\label{ItCEAfMapNonlin} We have $(\hat f,\hat j)\in\cA_\phg^{\cE+2}(\hat X\setminus\hat K^\circ;\ul\R\oplus\Tsc^*\hat X)$.
  \item\label{ItCEAfMapLin} The operator $L_{\hat\gamma,\hat k}\hat w$ is an element of
    \[
      \begin{pmatrix} (\CI{+}\cA_\phg^\cE)\Diffb^2(\hat X\setminus\hat K^\circ;S^2\,\Tsc^*\hat X,\ul\R) & \cA_\phg^\cE(\hat X\setminus\hat K^\circ;\Hom(S^2\,\Tsc^*\hat X,\ul\R)) \\ \cA_\phg^\cE\Diffb^1(\hat X\setminus\hat K^\circ;S^2\,\Tsc^*\hat X,\Tsc^*\hat X) & (\CI{+}\cA_\phg^\cE)\Diffb^1(\hat X\setminus\hat K^\circ;S^2\,\Tsc^*\hat X,\Tsc^*\hat X) \end{pmatrix}.
    \]
    The normal operator of $L_{\hat\gamma,\hat k}\hat w$ at $\pa\hat X$ is $\hat{\ubar L}\hat{\ubar w}$ where $\hat{\ubar L}:=L_{\hat\fe,0}$, i.e.
    \begin{equation}
    \label{EqCEAfMapLot}
      \hat{\ubar L} = \begin{pmatrix} \hat{\ubar L}_1 & 0 \\ 0 & \hat{\ubar L}_2 \end{pmatrix},\qquad \hat{\ubar L}_1:=\Delta_{\hat\fe}\tr_{\hat\fe}+\,\delta_{\hat\fe}\delta_{\hat\fe},\quad \hat{\ubar L}_2 := \delta_{\hat\fe}+\dd\tr_{\hat\fe},
    \end{equation}
    in the sense that $L_{\hat\gamma,\hat k}\hat w-\chi_\circ\hat{\ubar L}\hat w\in\cA_\phg^\cE\Diffb$ when $\chi_\circ\in\CI(\hat X)$ is equal to $1$ in a collar neighborhood of $\pa\hat X$ and has support in any larger neighborhood.
  \end{enumerate}
\end{lemma}

Since $L_{\hat\gamma,\hat k}\hat w$, being an unweighted b-operator, naturally acts on pairs of conormal sections of $S^2\,\Tsc^*\hat X$ which have the same weight at $\hat X$, the weight $\hat w$ encodes the relative weighting of linearized metrics $\hat h$ and linearized second fundamental forms $\hat q$, namely $\hat q$ decays faster than $\hat h$ by one order.\footnote{The fact that $L_{\hat\gamma,\hat k}\hat w$ is unweighted can be re-interpreted as the statement that $L_{\hat\gamma,\hat k}$ is a Douglis--Nirenberg system of weighted b-differential operators not only in the differential order sense, but also in the sense of weights, in that the $(i,j)$-block has weight $\rho_\circ^{a_j+b_i}$ where $a_1=2$, $a_2=1$, $b_1=0$, $b_2=0$. We opt for the explicit weight $\hat w$ for clarity here, however.}

\begin{proof}[Proof of Lemma~\usref{LemmaCEAfMap}]
   We have $\hat\gamma^{-1}-\hat\fe^{-1}=\sum_{j=1}^\infty(-1)^j(\hat\gamma-\hat\fe)^j\in\cA_\phg^\cE(\hat X;S^2\,\Tsc\hat X)$ since $\cE\supset j\cE$ for all $j\in\N$, and therefore $\hat\gamma^{-1}\in\cA_\phg^{\N_0\cup\cE}$. The Christoffel symbols of $\hat\gamma$, in the coordinate system $\hat x=(\hat x^1,\ldots,\hat x^n)$, therefore satisfy
   \[
     \Gamma(\hat\gamma)_{l i j} = \frac12(\pa_i\hat\gamma_{j l}+\pa_j\hat\gamma_{i l}-\pa_l\hat\gamma_{i j}) \in \cA_\phg^{\cE+1}(\hat X),\qquad
     \Gamma(\hat\gamma)^k_{i j}=\hat\gamma^{k l}\Gamma(\hat\gamma)_{l i j} \in \cA_\phg^{\cE+1}(\hat X),
   \]
   where we use that $\pa_i\in\rho\Vb(\hat X)$. Therefore,
   \[
     \Ric(\hat\gamma)_{j l} = \pa_i\Gamma(\hat\gamma)^i_{j l} - \pa_j\Gamma(\hat\gamma)^i_{i l} + \Gamma(\hat\gamma)^i_{i k}\Gamma(\hat\gamma)^k_{j l} - \Gamma(\hat\gamma)^i_{j k}\Gamma(\hat\gamma)_{i l}^k \in \cA_\phg^{\cE+2}(\hat X),
   \]
   and thus also $R_{\hat\gamma}=\hat\gamma^{j l}\Ric(\hat\gamma)_{j l}\in\cA_\phg^{\cE+2}(\hat X)$.

   We furthermore have $|\hat k|_{\hat\gamma}^2=\hat\gamma^{i i'}\hat\gamma^{j j'}\hat k_{i j}\hat k_{i'j'}$, $(\tr_{\hat\gamma}\hat k)^2=\hat\gamma^{i i'}\hat\gamma^{j j'}\hat k_{i i'}\hat k_{j j'}\in\cA_\phg^{\cE+2}$. Thus, $\hat f\in\cA_\phg^{\cE+2}(\hat X)$. Similarly, $\delta_{\hat\gamma}\hat k$ is of the schematic form $\hat\gamma^{-1}(\pa\hat k+\Gamma(\hat\gamma)\hat k)$ and thus lies in $\cA_\phg^{\cE+2}(\hat X;\Tsc^*\hat X)$; and also $\dd(\tr_{\hat\gamma}\hat k)$ (schematically: $\pa(\hat\gamma^{-1}\hat k)$) lies in $\cA_\phg^{\cE+2}$. This proves part~\eqref{ItCEAfMapNonlin}.

   For part~\eqref{ItCEAfMapLin}, note that the only terms of $L_{\hat\gamma,\hat k}$ which are of leading order at $\pa\hat X$ (in the sense that they contribute to the normal operator) are those in which $\hat k$ does not appear; for example, the $\la\nabla\hat k,-\ra_{2 3}$ term of $L_{\hat\gamma,\hat k}$ has the schematic form $\hat\gamma^{-1}\hat\gamma^{-1}(\pa\hat k+\Gamma(\hat\gamma)\hat k)$ and is thus polyhomogeneous with index set $\cE+2$. In the terms of $L_{\hat\gamma,\hat k}$ then in which only $\hat\gamma$ appears, only the leading part $\hat\fe$ of $\hat\gamma$ contributes to the normal operator. This gives~\eqref{EqCEAfMapLot}.
\end{proof}

We proceed to study the solvability properties of $L_{\hat\gamma,\hat k}$. Using a geometric singular analysis point of view, rather than more frequently used direct coercivity estimates (as for example in \cite[\S3]{CorvinoScalar} or \cite[\S4]{CarlottoSchoenData} in related contexts), we provide a technically novel perspective on this. We begin by studying the symbolic properties of the adjoint of $L_{\hat\gamma,\hat k}\hat w$.

\begin{lemma}[Injective b-principal symbol]
\label{LemmaCEAfbEll}
  In the notation of Lemma~\usref{LemmaCEAfMap}\eqref{ItCEAfMapLin}, the operator
  \[
    \hat w L_{\hat\gamma,\hat k}^* \in \begin{pmatrix} (\CI+\cA_\phg^\cE)\Diffb^2 & \cA_\phg^\cE\Diffb^1 \\ \cA_\phg^\cE & (\CI+\cA_\phg^\cE)\Diffb^1 \end{pmatrix}
  \]
  has an injective principal symbol as a b-differential operator (in the sense of Douglis--Nirenberg).
\end{lemma}
\begin{proof}
  This was verified in Lemma~\ref{LemmaCELstarInj} in $\hat X\setminus\pa\ol{\R^n}$; near $\pa\ol{\R^n}$, the same argument applies upon noting that the principal symbol of $\ubar\rho_\circ^{-1}\pa_{x^j}$ is a nonvanishing linear form on $\Tb^*\ol{\R^n}\setminus o$.
\end{proof}

We next turn to the normal operator of the adjoint $L_{\hat\gamma,\hat k}^*$ (with respect to the volume density and bundle inner products induced by $\hat\gamma$), which is the diagonal operator
\[
  \hat{\ubar L}^* = \begin{pmatrix} \hat{\ubar L}_1^* & 0 \\ 0 & \hat{\ubar L}_2^* \end{pmatrix},\qquad
  \hat{\ubar L}_1^* = \hat\fe\Delta_{\hat\fe} + \delta_{\hat\fe}^*\dd, \quad
  \hat{\ubar L}_2^* = \delta_{\hat\fe}^* + \hat\fe\delta_{\hat\fe},
\]
i.e.\ the formal adjoint of $\hat{\ubar L}$ with respect to $\hat\fe$; the operator $\hat{\ubar L}_1^*$, resp.\ $\hat{\ubar L}_2^*$ maps real-valued functions, resp.\ sections of $\Tsc^*\ol{\R^n}$ to sections of $S^2\,\Tsc^*\ol{\R^n}$. We work in polar coordinates $\hat r\geq 0$, $\omega\in\Sph^{n-1}$ on $\R^n\setminus\{0\}$. By trivializing $\Tsc^*\ol{\R^n}$ by means of the standard coordinate differentials, we can extend sections of $\Tsc^*\ol{\R^n}$ (and its tensor powers) over $\Sph^{n-1}=\pa\ol{\R^n}$ to sections over $\ol{\R^n}\setminus\{0\}$ by degree $0$ homogeneity with respect to dilations $(\ubar\rho_\circ,\omega)\mapsto(s\ubar\rho_\circ,\omega)$, $s>0$. Consider then the Mellin-transformed normal operator
\[
  N(\hat{\ubar w}\hat{\ubar L}^*,\lambda)=\ubar\rho_\circ^{-i\lambda}\hat{\ubar w}\hat{\ubar L}^*\ubar\rho_\circ^{i\lambda} \in \begin{pmatrix} \Diff^2(\pa\ol{\R^n};\ubar\R,S^2\,\Tsc^*\ol{\R^n}) & 0 \\ 0 & \Diff^1(\pa\ol{\R^n};\Tsc^*\ol{\R^n},S^2\,\Tsc^*\ol{\R^n}) \end{pmatrix}.
\]

\begin{lemma}[Kernel and boundary spectrum of $\ubar L^*$]
\label{LemmaCEAfSpecb}
  (See also \cite[Lemma~2.5]{CorvinoSchoenAsymptotics}.) For any connected open subset $\Omega\subset\R^n$, the kernel of $\ubar L^*$ on $\sD'(\Omega;\ul\R\oplus T^*\R^N)$ is
  \begin{align*}
    \ker\hat{\ubar L}^* &= \mathspan\{1, \hat x^i\ (i=1,\ldots,n)\} \\
      &\qquad \oplus \mathspan\{\dd\hat x^i\ (i=1,\ldots,n),\ \hat x^i\,\dd\hat x^j-\hat x^j\,\dd\hat x^i\ (1\leq i<j\leq n) \}.
  \end{align*}
  In particular, the operator $N(\hat w\hat{\ubar L}^*,\lambda)$ is injective for $\lambda\in\C\setminus\{0,i\}$. For any compact set $\Lambda\subset\C\setminus\{0,i\}$, and for any $s,s'\in\R$, there exists a constant $C$ so that, for $\lambda\in\Lambda$,
  \begin{equation}
  \label{EqCEAfSpecbEst}
  \begin{split}
    &\|(\hat f^*,\hat j^*)\|_{H^s(\pa\ol{\R^n})\oplus H^{s'}(\pa\ol{\R^n};\Tsc^*\ol{\R^n})} \\
    &\qquad\leq C\|N(\hat{\ubar w}\hat{\ubar L}^*,\lambda)(\hat f^*,\hat j^*)\|_{H^{s-2}(\pa\ol{\R^n};S^2\,\Tsc^*\ol{\R^n})\oplus H^{s'-1}(\pa\ol{\R^n};S^2\,\Tsc^*\ol{\R^n})}.
  \end{split}
  \end{equation}
\end{lemma}
\begin{proof}
  By taking the trace of $\hat{\ubar L}_1^*\hat f^*=0$, we obtain $(n-1)\Delta_{\hat\fe}\hat f^*=0$; this then implies $0=\hat{\ubar L}_1^*\hat f^*=\delta_{\hat\fe}^*\dd\hat f^*=(\pa_{\hat x^i}\pa_{\hat x^j}\hat f^*)_{i,j=1,\ldots,n}$, so $\hat f^*$ is linear. Conversely, for linear $\hat f^*$ we do have $\hat{\ubar L}_1^*\hat f^*=0$.

  We similarly note that taking the trace of $\hat{\ubar L}_2^*\hat j^*=0$ gives $(n-1)\delta_{\hat\fe}\hat j^*=0$ and thus $\delta_{\hat\fe}^*\hat j^*=0$. Conversely, if $\delta_{\hat\fe}^*\hat j^*=0$, then by taking the trace we have $-\delta_{\hat\fe}\hat j^*=0$, and therefore $\hat{\ubar L}_2^*\hat j^*=0$. Therefore, $\ker\hat{\ubar L}_2^*=\ker\delta_{\hat\fe}^*$ consists of all Killing 1-forms on Euclidean space; these are the 1-forms dual to translations and rotations.

  The estimate in~\eqref{EqCEAfSpecbEst} is the sum of two estimates, one for $\hat f^*$ (with $\hat j^*=0$) and one for $\hat j^*$ (with $\hat f^*=0$). The estimate for $\hat f^*$ follows from the ellipticity of $N(\ubar\rho_\circ^{-2}\hat{\ubar L}_1^*,\lambda)^*N(\ubar\rho_\circ^{-2}\hat{\ubar L}_1^*,\lambda)\in\Diff^4(\pa\ol{\R^n})$ (which gives the desired estimate except for an additional, relatively compact, error term $C\|\hat f^*\|_{H^{-N}(\pa\ol{\R^n})}$ on the right hand side) and the triviality of $\ker N(\ubar\rho_\circ^{-2}\hat{\ubar L}_1^*,\lambda)$ (which allows one to drop this relatively compact term, upon increasing the constant $C$, by a standard functional analytic argument). The proof of the estimate for $\hat j^*$ is completely analogous.
\end{proof}

\begin{lemma}[Cokernel of $L_{\hat\gamma,\hat k}$]
\label{LemmaCEAfCoker}
  Let $(\hat\gamma,\hat k)$ be $\cE$-asymptotically flat. Let $\alpha^*>0$, and suppose\footnote{Recall that we are using b-densities throughout. With respect to the Euclidean volume density, the assumption would read $(\hat f^*,\hat j^*)\in\Hb^{-\infty,\frac{n}{2}+\alpha^*}$, with $\frac{n}{2}+\alpha^*>\frac{n}{2}$.} $(\hat f^*,\hat j^*)\in\Hb^{-\infty,\alpha^*}(\ol{\R^n};\ul\R\oplus S^2\,\Tsc^*\ol{\R^n})$ satisfies $L_{\hat\gamma,\hat k}^*(\hat f^*,\hat j^*)=0$ in a connected open subset $\Omega\subset\ol{\R^n}\setminus K$ which contains a neighborhood of $\pa\ol{\R^n}$. Then $(\hat f^*,\hat j^*)=0$ on $\Omega$.
\end{lemma}
\begin{proof}
  By elliptic regularity for $L_{\hat\gamma,\hat k}^*$ (cf.\ Lemma~\ref{LemmaCEAfbEll}), we automatically have $(\hat f^*,\hat j^*)\in\Hb^{\infty,\alpha^*}$ on $\Omega$. We first show that $(\hat f^*,\hat j^*)=0$ for large $\hat r$. To this end,\footnote{Using the information about the boundary spectrum of $\hat{\ubar L}^*$, one could improve the conormality of $(\hat f^*,\hat j^*)$ at infinity to rapid vanishing of $(\hat f^*,\hat j^*)$, whence the vanishing of $\hat f^*,\hat j^*$ near infinity would follow from a unique continuation result; this is the path taken in \cite[Proposition~3.1]{CorvinoSchoenAsymptotics}. We present a different argument here which directly exploits elliptic estimates and the overdetermined nature of $L_{\hat\gamma,\hat k}^*$.} we work with inverse polar coordinates $\rho_\circ,\omega$ near $\pa\ol{\R^n}$ and consider the scaling map $s_\eps\colon(\rho_\circ,\omega)\mapsto(\eps\rho_\circ,\omega)$. Using the weight $\hat w$ of Definition~\ref{DefCEAfWeight}, we consider
  \[
    \cL_\eps^* = s_\eps^*(\hat w L_{\hat\gamma,\hat k}^*)
  \]
  for $\eps\in(0,\rho_0)$ as a b-differential operator on $[0,1]\times\Sph^{n-1}$, where $\rho_0\in(0,1]$ is fixed so that $\rho_\circ^{-1}([0,\rho_0))\subset\Omega$ and $\hat r=\rho_\circ^{-1}<\rho_0^{-1}$ on $K$. As such, we have
  \begin{equation}
  \label{EqCEAfCokerLim}
    \cL_\eps^* \xra{\eps\searrow 0} \cL_0^* := \hat w\hat{\ubar L}^*
  \end{equation}
  in the space
  \[
    \begin{pmatrix} \Diffb^2 + \cA_\phg^\cE\Diffb^2 & \cA_\phg^\cE\Diffb^1 \\ \cA_\phg^\cE & \Diffb^1 + \cA_\phg^\cE\Diffb^1 \end{pmatrix};
  \]
  note indeed that the pullback of the polyhomogeneous terms with index set $\cE$ vanish in the limit $\eps\searrow 0$, since, recalling that $\Re\cE>0$, we have $s_\eps^*(\rho_\circ^z(\log\rho_\circ^{-1})^k)=\eps^z\cdot\rho_\circ^z(\log\rho_\circ^{-1}+\log\eps^{-1})^k\to 0$ when $\Re z>0$.

  We now use the uniform b-ellipticity of $(\cL_\eps^*)^*\cL_\eps^*\in(\Diffb^{s_j+s_i})_{i,j=1,2}$ in $\eps\in(0,\rho_0)$, with $s_1=2$, $s_2=1$ (see~\eqref{EqCELLstar}). This gives the uniform elliptic estimate\footnote{Here, the b-Sobolev norm on $\rho_\circ^{-1}([0,\rho_1])$, for $\rho_1>0$, is defined via differentiation along $\rho_\circ\pa_{\rho_\circ}$ and spherical vector fields; that is, we test with \emph{all} smooth vector fields near the `artificial boundary' at $\rho_\circ=\rho_1$. The usage of the notation $\Hbext$ instead of $\Hb$ stresses this \emph{extendible} character of the function space in the terminology of H\"ormander \cite[Appendix~B]{HormanderAnalysisPDE3}.}
  \begin{align*}
    &\|(\hat f^*,\hat j^*)\|_{(\Hbext^{s+2,\alpha^*}\oplus\Hbext^{s+1,\alpha^*})(\rho_\circ^{-1}([0,\frac12]))} \\
    &\qquad \leq C\Bigl( \|(\cL_\eps^*)^*\cL_\eps^*(\hat f^*,\hat j^*)\|_{(\Hbext^{s-2,\alpha^*}\oplus\Hbext^{s-1,\alpha^*})(\rho_\circ^{-1}([0,1]))} + \|(\chi \hat f^*,\chi \hat j^*)\|_{\Hb^{s_0+2,\alpha^*}\oplus\Hb^{s_0+1,\alpha^*}}\Bigr) \\
    &\qquad \leq C'\Bigl( \|\cL_\eps^*(\hat f^*,\hat j^*)\|_{\Hbext^{s,\alpha^*}(\rho_\circ^{-1}([0,1]))} + \|(\chi \hat f^*,\chi \hat j^*)\|_{\Hb^{s_0,\alpha^*}\oplus\Hb^{s_0+1,\alpha^*}}\Bigr)
  \end{align*}
  for any $s_0<s$, where we fix $\chi\in\CIc([0,1)_{\rho_\circ}\times\Sph^{n-1})$ to be $1$ near $\rho_\circ^{-1}([0,\frac12])$. We claim that we can replace the norm on $(\chi \hat f^*,\chi \hat j^*)$ by the $\Hbext^{s_0,\alpha^*-\delta}(\rho_\circ^{-1}([0,1]))$-norm on $(\hat f^*,\hat j^*)$, where $\delta>0$ is such that $\Re\cE>\delta$; we accomplish this via a standard b-normal operator argument \cite[\S6]{VasyMinicourse}. To wit, passing to the Mellin-transform in $\rho_\circ$ and using the material of~\S\ref{SsBFn}, we have an equivalence of norms
  \[
    \|(\chi \hat f^*,\chi \hat j^*)\|_{\Hb^{s_0+2,\alpha^*}\oplus\Hb^{s_0+1,\alpha^*}}^2 \sim \int_{\Im\lambda=-\alpha^*} \|\cM((\chi \hat f^*,\chi \hat j^*))(\lambda)\|_{(H^{s_0+2}_{\la\lambda\ra}\oplus H^{s_0+1}_{\la\lambda\ra})(\pa\ol{\R^n})}^2\,\dd\lambda.
  \]
  Using a parametrix for the Mellin-transformed normal operator of $(\cL_\eps^*)^*\cL_\eps^*$, which thus has an error term which tends to zero in operator norm on $H_{\la\lambda\ra}^{s_0+2}\oplus H_{\la\lambda\ra}^{s_0+1}$ as $|\Re\lambda|\to\infty$ along the contour $\Im\lambda=-\alpha^*$, we can estimate
  \begin{equation}
  \label{EqCEAfCokerEst}
    \|\cM((\chi \hat f^*,\chi \hat j^*))(\lambda)\|_{H_{\la\lambda\ra}^{s_0+2}\oplus H_{\la\lambda\ra}^{s_0+1}} \leq C\|\cM(w\ubar L^*(\chi \hat f^*,\chi \hat j^*))(\lambda)\|_{H_{\la\lambda\ra}^{s_0}\oplus H_{\la\lambda\ra}^{s_0}}
  \end{equation}
  for $\Im\lambda=-\alpha^*$ when $|\Re\lambda|\geq C_0$ (independently of $\eps$). For $|\Re\lambda|\leq C_0$ on the other hand, the large parameter norms are equivalent to standard Sobolev norms, and we can then appeal to~\eqref{EqCEAfSpecbEst} to conclude the estimate~\eqref{EqCEAfCokerEst} for all $\lambda$ with $\Im\lambda=-\alpha^*$. We have thus proved
  \[
    \|(\chi \hat f^*,\chi \hat j^*)\|_{\Hb^{s_0+2,\alpha^*}\oplus\Hb^{s_0+1}} \leq C\|\cL_0^*(\chi \hat f^*,\chi \hat j^*))\|_{\Hb^{s_0,\alpha^*}}.
  \]
  We then commute $\cL_0^*$ through $\chi$, producing an error term $\|(\chi^\flat \hat f^*,\chi^\flat \hat j^*)\|_{\Hb^{s_0+1}\oplus\Hb^{s_0}}$ where $\chi^\flat\in\CIc((0,1)_{\rho_\circ}\times\Sph^{n-1})$, with $\supp\chi^\flat\supset\supp\dd\chi$, is disjoint from $\pa\ol{\R^n}$. We can further write $\chi\cL_0^*(\hat f^*,\hat j^*)=\chi\cL_\eps^*(\hat f^*,\hat j^*)-\chi(\cL_\eps^*-\cL_0^*)(\hat f^*,\hat j^*)$ and estimate the error using the fact that
  \[
    \cL_\eps^*-\cL_0^* \in \cA^\delta \begin{pmatrix} \Diffb^2 & \Diffb^1 \\ \Diffb^0 & \Diffb^1 \end{pmatrix}
  \]
  is uniformly bounded in $\eps$. Altogether, we obtain the uniform estimate
  \begin{equation}
  \label{EqCEAfCokerEst2}
  \begin{split}
    &\|(\hat f^*,\hat j^*)\|_{(\Hbext^{s+2,\alpha^*}\oplus\Hbext^{s+1,\alpha^*})(\rho_\circ^{-1}([0,\frac12]))} \\
    &\qquad \leq C\Bigl( \|\cL_\eps^*(\hat f^*,\hat j^*)\|_{\Hbext^{s,\alpha^*}(\rho_\circ^{-1}([0,1]))} + \|(\hat f^*,\hat j^*)\|_{(\Hbext^{s_0+2,\alpha^*-\delta}\oplus\Hbext^{s_0+1,\alpha^*-\delta})(\rho_\circ^{-1}([0,1]))}\Bigr).
  \end{split}
  \end{equation}
  
  By the final part of Lemma~\ref{LemmaCEProlong}, we can express $(\hat f^*,\hat j^*,\nabla \hat f^*,\nabla \hat j^*,\nabla^2 \hat f^*)(\rho_\circ,\omega)$ for $\rho_\circ\in[\frac12,1]$ in terms of the values of the tensors $\hat f^*,\hat j^*,\nabla \hat f^*,\nabla \hat j^*$ at the point $(\rho_0,\omega)$ and of  $\cL_\eps^*(\hat f^*,\hat j^*),\nabla(\cL_\eps^*(\hat f^*,\hat j^*))$ at all points $(\rho_1,\omega)$, $\rho_1\in[\rho_0,\rho_\circ]$, where $\rho_0\in[\frac14,\frac12]$ is arbitrary, as the solution of an ODE whose coefficients depend only on $\rho_0$, $\omega$, $\eps$, $\gamma$, and $k$. Therefore, upon integrating over the $(n-1)$-sphere and averaging over $\rho_0$,
  \begin{align*}
    &\|(\hat f^*,\hat j^*,\nabla \hat f^*,\nabla \hat j^*,\nabla^2 \hat f^*)(\rho_\circ,-)\|_{L^2(\Sph^{n-1})}^2 \\
    &\qquad \leq C\Biggl(\int_{\frac14}^{\frac12} \|(\hat f^*,\hat j^*,\nabla \hat f^*,\nabla \hat j^*,\nabla^2 \hat f^*)(\rho_0,-)\|_{L^2(\Sph^{n-1})}^2\,\dd\rho_0 \\
    &\qquad\qquad\qquad + \int_{\frac14}^{\frac12}\int_{\Sph^{n-1}} \left\|\int_{\rho_0}^{\rho_\circ} \bigl|(\cL_\eps^*(\hat f^*,\hat j^*),\nabla(\cL_\eps^*(\hat f^*,\hat j^*)))|_{(\rho_1,\omega)}\bigr|\,\dd\rho_1\right\|^2\,\dd\omega\,\dd\rho_0\Biggr),
  \end{align*}
  Applying the Cauchy--Schwarz inequality in the last integral, performing the integral over $\rho_0$, and subsequently integrating over $\rho_\circ\in[\frac12,1]$ gives
  \[
    \|(\hat f^*,\hat j^*)\|_{(H^2\oplus H^1)(\rho_\circ^{-1}([\frac12,1]))} \leq C\Bigl( \|(\hat f^*,\hat j^*)\|_{(H^2\oplus H^1)(\rho_\circ^{-1}([\frac14,\frac12]))} + \|\cL_\eps^*(\hat f^*,\hat j^*)\|_{H^1(\rho_\circ^{-1}([\frac14,1]))} \Bigr).
  \]
  Estimating the first term on the right here by means of~\eqref{EqCEAfCokerEst2} with $s=0$ and $s_0=-1$, and adding the resulting estimate to~\eqref{EqCEAfCokerEst2}, we finally deduce the uniform semi-Fredholm estimate
  \begin{equation}
  \label{EqCEAfCokerEstFinal}
  \begin{split}
    &\|(\hat f^*,\hat j^*)\|_{(\Hbext^{2,\alpha^*}\oplus\Hbext^{1,\alpha^*})(\rho_\circ^{-1}([0,1]))} \\
    &\qquad \leq C\Bigl( \|\cL_\eps^*(\hat f^*,\hat j^*)\|_{\Hbext^{1,\alpha^*}(\rho_\circ^{-1}([0,1]))} + \|(\hat f^*,\hat j^*)\|_{(\Hbext^{1,\alpha^*-\delta}\oplus\Hbext^{0,\alpha^*-\delta})(\rho_\circ^{-1}([0,1]))}\Bigr)
  \end{split}
  \end{equation}
  for all $\eps\in[0,1]$. Now the error term is relatively compact, in that the inclusion $(\Hbext^{2,\alpha^*}\oplus\Hbext^{1,\alpha^*})(\rho_\circ^{-1}([0,1]))\hra(\Hbext^{1,\alpha^*-\delta}\oplus\Hbext^{0,\alpha^*-\delta})(\rho_\circ^{-1}([0,1]))$ is compact.

  We now make use of the limiting behavior~\eqref{EqCEAfCokerLim}. Note that the distributional kernel of $\cL_0^*$ is given by $(\hat f^*,\hat j^*)$ where $\hat f^*\in\ker\ubar L_1^*$ and $\hat j^*\in\ker\ubar L_2^*$ are described in Lemma~\ref{LemmaCEAfSpecb}; but elements of the kernel of $\cL_0^*$ on $(\Hb^{2,\alpha^*}\oplus\Hb^{1,\alpha^*})(\rho_\circ^{-1}([0,1]))$ lie in $\Hb^{\infty,\alpha^*}(\rho_\circ^{-1}([0,1]))$ by elliptic regularity, and thus decay (as scattering tensors, i.e.\ with respect to the frame $\dd x^1,\ldots,\dd x^n$) to zero as $|x|\to\infty$, and hence they must be zero. A standard functional analytic argument allows one to drop the relatively compact error term in~\eqref{EqCEAfCokerEstFinal} for $\eps=0$, and then also for all sufficiently small $\eps\geq 0$. Translated back to $L_{\gamma,k}^*$, this means that for sufficiently small $\eps>0$, we have
  \[
    \|(\hat f^*,\hat j^*)\|_{(\Hbext^{2,\alpha^*}\oplus\Hbext^{1,\alpha^*})(\rho_\circ^{-1}([0,\eps]))} \leq C\|L_{\hat\gamma,\hat k}^*(\hat f^*,\hat j^*)\|_{\Hbext^{0,\alpha^*+2}(\rho_\circ^{-1}([0,\eps]))}.
  \]
  Thus, if $L_{\hat\gamma,\hat k}^*(\hat f^*,\hat j^*)$ vanishes in $\rho_\circ\leq\eps$, then so does $(\hat f^*,\hat j^*)$. By Lemma~\ref{LemmaCEProlong}, this then implies the vanishing of $(\hat f^*,\hat j^*)$ in $\Omega$, completing the proof.
\end{proof}

\begin{prop}[Solvability of the linearized constraints]
\label{PropCEAfSolv}
  Let $(\hat\gamma,\hat k)$ be $\cE$-asymp\-tot\-i\-cal\-ly flat. Let $\hat R_0>0$ be such that $\hat K\subset B(0,\hat R_0)$, and put $\hat\rho_2=\hat R_0^{-1}-\hat r^{-1}$ and $\hat w_2=\diag(\hat\rho_2^4,\hat\rho_2^2)$; recall also $\hat w=\diag(\rho_\circ^{-2},\rho_\circ^{-1})$ from Definition~\usref{DefCEAfWeight}. Let $\hat\Omega=\rho_\circ^{-1}([0,\hat R_0^{-1}])$, and define b-00-Sobolev spaces on $\hat\Omega$ (with b-character at $\rho_\circ^{-1}(0)$ and 00-character at $\hat r^{-1}(\hat R_0)$) using a positive b-00-density on $\hat\Omega$. Fix $\alpha_\circ<n-2$.\footnote{The assumption on $\alpha_\circ$ eliminates the cokernel of $L_{\hat\gamma,\hat k}$; note that for Kerr data sets $(\hat\gamma,\hat k)$, the kernel of $L_{\hat\gamma,\hat k}^*$ is indeed non-trivial \cite{MoncriefLinStabI}. If however $\ker L_{\hat\gamma,\hat k}^*$ \emph{is} trivial, then the Proposition holds for any $\alpha_\circ\in\R$.}
  \begin{enumerate}
  \item\label{ItCEAfSolv1}{\rm (Basic solvability.)} Fix $\beta>0$ and $\alpha_2\in\R$. Put
    \begin{equation}
    \label{EqCEAfSolvOp}
      \cL_{\hat\gamma,\hat k} := e^{\beta/\hat\rho_2}\hat\rho_2^{-\alpha_2}\rho_\circ^{-\tilde\alpha}\hat w_2 L_{\hat\gamma,\hat k}\hat w \rho_\circ^{\tilde\alpha}\hat\rho_2^{\alpha_2} e^{-\beta/\hat\rho_2},\qquad \tilde\alpha:=\alpha_\circ-n+2+\frac{n}{2}.
    \end{equation}
    Define adjoints with respect to the volume density and fiber inner products induced by $\hat\gamma$. Then for all $s\in\R$, the operator $\cL_{\hat\gamma,\hat k}\cL_{\hat\gamma,\hat k}^*$ is invertible as a map
    \begin{equation}
    \label{EqCEAfSolvInv}
      \cL_{\hat\gamma,\hat k}\cL_{\hat\gamma,\hat k}^* \colon \rho_\circ^{\frac{n}{2}}H_{\bop,0 0}^{s+2}(\hat\Omega)\oplus\rho_\circ^{\frac{n}{2}}H_{\bop,0 0}^{s+1}(\hat\Omega;\Tsc^*_{\hat\Omega}\hat X) \to \rho_\circ^{\frac{n}{2}}H_{\bop,0 0}^{s-2}(\hat\Omega)\oplus\rho_\circ^{\frac{n}{2}}H_{\bop,0 0}^{s-1}(\hat\Omega;\Tsc^*_{\hat\Omega}\hat X).
    \end{equation}
  \item\label{ItCEAfSolvPhg}{\rm (Polyhomogeneity.)} There exists\footnote{One can in principle compute the set $\hat\cS$ produced in the proof explicitly (the main ingredient being the computation of the set of poles of the operator $\cN(\alpha_\circ,\lambda)^{-1}$ in~\eqref{EqCEAfSolvNormOp} below), though we shall not present the details here as the result depends on $\alpha_\circ$ in a rather complicated manner. We leave open the interesting problem of finding the smallest possible $\hat\cS$; we conjecture that one can ensure $\hat\cS\subset-i(\N_0+n-2)\times\N_0$. See also Remark~\ref{RmkCEPctIndex}.} an index set $\hat\cS\subset\C\times\N_0$ only depending on $\cE$ and $\alpha_\circ$ and satisfying $\Re\hat\cS>\alpha_\circ$ so that the following holds. If $\cG\subset\C\times\N_0$ is an index set with $\Re\cG>\alpha_\circ$, and if $(\hat f,\hat j)\in\cA_\phg^{\cG+2}(\hat X;\ul\R\oplus\Tsc^*\hat X)$ vanishes in a neighborhood of $\hat r\leq\hat R_0$, then there exists
    \[
      (\hat h,\hat q) \in \cA_\phg^{\cG\extcup\hat\cS}(\hat X;S^2\,\Tsc^*\hat X) \oplus \cA_\phg^{(\cG\extcup\hat\cS)+1}(\hat X;S^2\,\Tsc^*\hat X)
    \]
    so that $L_{\hat\gamma,\hat k}(\hat h,\hat q)=(\hat f,\hat j)$, and so that $(\hat h,\hat q)$ vanishes in a neighborhood of $\hat r\leq\hat R_0$.
  \end{enumerate}
\end{prop}

\begin{rmk}[Right inverse on b-00-Sobolev spaces]
\label{RmkCEAfSolvRightInverse}
  Part~\eqref{ItCEAfSolv1} is formulated in the form needed later on. For the orientation of the reader, we note here that this produces a right inverse $S_{\hat\gamma,\hat k}$ of $L_{\hat\gamma,\hat k}$ which is bounded as a map
  \begin{align*}
    S_{\hat\gamma,\hat k} &\colon \rho_\circ^{\alpha_\circ+2}\hat\rho_2^{\alpha_2-4}e^{-\beta/\hat\rho_2}H_{\bop,00}^{s-2}(\hat\Omega) \oplus \rho_\circ^{\alpha_\circ+2}\hat\rho_2^{\alpha_2-2}e^{-\beta/\hat\rho_2}H_{\bop,00}^{s-1}(\hat\Omega;\Tsc^*_{\hat\Omega}\hat X) \\
      &\quad \to \rho_\circ^{\alpha_\circ}\hat\rho_2^{\alpha_2}e^{-\beta/\hat\rho_2}H_{\bop,00}^s(\hat\Omega;S^2\,\Tsc^*_{\hat\Omega}\hat X)\oplus\rho_\circ^{\alpha_\circ+1}\hat\rho_2^{\alpha_2}e^{-\beta/\hat\rho_2}H_{\bop,0 0}^s(\hat\Omega;S^2\,\Tsc^*_{\hat\Omega}\hat X);
  \end{align*}
  such a right inverse is given by $S_{\hat\gamma,\hat k} = e^{-\beta/\hat\rho_2}\hat\rho_2^{\alpha_2}\rho_\circ^{\tilde\alpha}\hat w\cL_{\hat\gamma,\hat k}^*(\cL_{\hat\gamma,\hat k}\cL_{\hat\gamma,\hat k}^*)^{-1}\hat w_2 \rho_\circ^{-\tilde\alpha}\hat\rho_2^{-\alpha_2}e^{\beta/\hat\rho_2}$.
\end{rmk}

\begin{proof}[Proof of Proposition~\usref{PropCEAfSolv}]
  Using Lemma~\ref{LemmaCEAfMap}\eqref{ItCEAfMapLin}, and as in~\S\ref{SsB00}, we have
  \begin{align*}
    \cL_{\hat\gamma,\hat k} &\in \begin{pmatrix} (\CI{+}\cA_\phg^\cE)\Diff_{\bop,0 0}^2(\hat\Omega;S^2\,\Tsc^*\hat X,\ul\R) & \cA_\phg^\cE(\hat\Omega;\Hom(S^2\,\Tsc^*\hat X,\ul\R)) \\ \cA_\phg^\cE\Diff_{\bop,0 0}^1(\hat\Omega;S^2\,\Tsc^*\hat X,\Tsc^*\hat X) & (\CI{+}\cA_\phg^\cE)\Diff_{\bop,0 0}^1(\hat\Omega;S^2\,\Tsc^*\hat X,\Tsc^*\hat X) \end{pmatrix}, \\
    \cL_{\hat\gamma,\hat k}^*&=e^{-\beta/\hat\rho_2} \hat\rho_2^{\alpha_2}\rho_\circ^{\tilde\alpha} \hat w L_{\hat\gamma,\hat k}^*\hat w_2\rho_\circ^{-\tilde\alpha}\hat\rho_2^{-\alpha_2}e^{\beta/\hat\rho_2} \\
      &\in \begin{pmatrix} (\CI{+}\cA_\phg^\cE)\Diff_{\bop,0 0}^2 & \cA_\phg^\cE\Diff_{\bop,0 0}^1 \\ \cA_\phg^\cE & (\CI{+}\cA_\phg^\cE)\Diff_{\bop,0 0}^1 \end{pmatrix}.
  \end{align*}
  Away from the boundary $\hat\rho_2^{-1}(0)$, the operator $\cL_{\hat\gamma,\hat k}^*$ has an injective b-principal symbol by Lemma~\ref{LemmaCEAfbEll}; near $\hat\rho_2=0$ on the other hand, the Douglis--Nirenberg 00-principal symbol at a point $(\hat x,\xi_{0 0})\in{}^{0 0}T^*\hat\Omega$ in $\rho>0$ is $w$ times
  \[
    \upsigma(L_{\hat\gamma,\hat k}^*)(\hat x,\hat\rho_2^2\xi_{0 0}+i\beta\,\dd\hat\rho_2)
  \]
  by~\eqref{EqB00SymbConj}, where we use the notation~\eqref{EqCELstarSymb} and the normalization~\eqref{EqCEDNOrders}. By the same arguments as in the proof of Lemma~\ref{LemmaCELstarInj} upon setting $\xi=\hat\rho_2^2\xi_{0 0}$, this is injective, including for $\xi_{0 0}=0$, since $\hat\rho_2^2\xi_{0 0}+i\beta\,\dd\hat\rho_2\neq 0$ for all (real) 00-covectors $\xi_{0 0}$. Therefore, the b-00-operator
  \[
    \cL_{\hat\gamma,\hat k}\cL_{\hat\gamma,\hat k}^* \in \begin{pmatrix} (\CI+\cA_\phg^\cE)\Diff_{\bop,0 0}^4 & \cA_\phg^\cE\Diff_{\bop,0 0}^3 \\ \cA_\phg^\cE\Diff_{\bop,0 0}^3 & (\CI+\cA_\phg^\cE)\Diff_{\bop,0 0}^2 \end{pmatrix}
  \]
  is elliptic in the Douglis--Nirenberg sense (including at $\hat\rho_2=0$ as a 00-operator), the $(i,j)$-entry having differential order $s_j+s_i$ where $s_1=2$, $s_2=1$; see also the comments following~\eqref{EqCELLstar}. Given $s_0,s\in\R$ with $s_0<s$, the elliptic estimates of Lemma~\ref{LemmaBAnEll} and Proposition~\ref{PropBAnEll00} thus give
  \begin{equation}
  \label{EqCEAfSolvEstPre}
  \begin{split}
    &\|(\hat f^*,\hat j^*)\|_{\rho_\circ^{\frac{n}{2}}H_{\bop,0 0}^{s+2}\oplus \rho_\circ^{\frac{n}{2}}H_{\bop,0 0}^{s+1}} \\
    &\qquad \leq C\Bigl( \|\cL_{\hat\gamma,\hat k}\cL_{\hat\gamma,\hat k}^*(\hat f^*,\hat j^*)\|_{\rho_\circ^{\frac{n}{2}}H_{\bop,0 0}^{s-2}\oplus \rho_\circ^{\frac{n}{2}}H_{\bop,0 0}^{s-1}} + \|(\hat f^*,\hat j^*)\|_{\rho_\circ^{\frac{n}{2}}\hat\rho_2^{-1}H_{\bop,0 0}^{s_0+2}\oplus\rho_\circ^{\frac{n}{2}}\hat\rho_2^{-1}H_{\bop,0 0}^{s_0+1}}\Bigr).
  \end{split}
  \end{equation}
  While this estimate is valid for any weight, the significance of $\frac{n}{2}$ is that it translates between b-densities (used in the function spaces) and asymptotically Euclidean densities (used in the definition of the adjoint), in that $\rho^{\frac{n}{2}}\Hb^0(\ol{\R^n})=L^2(\R^n;|\dd\gamma|)$ (with equivalent norms).

  As in the proof of Lemma~\ref{LemmaCEAfCoker}, we can weaken the second, error, term in~\eqref{EqCEAfSolvEstPre} further provided that the Mellin-transformed b-normal operator of $\cL_{\hat\gamma,\hat k}\cL_{\hat\gamma,\hat k}^*$ is invertible for all Mellin-dual parameters $\lambda\in\C$ with $\Im\lambda=-\frac{n}{2}$. But the invertibility of
  \[
    N(\cL_{\hat\gamma,\hat k}\cL_{\hat\gamma,\hat k}^*,\lambda) = N(\cL_{\hat\gamma,\hat k},\lambda) N(\cL_{\hat\gamma,\hat k}^*,\lambda) \colon H^{s+2}\oplus H^{s+1}\to H^{s-2}\oplus H^{s-1}
  \]
  is equivalent to that of its constant multiple
  \begin{equation}
  \label{EqCEAfSolvNormOp}
    \cN(\alpha_\circ,\lambda) := N(\hat{\ubar L}\hat{\ubar w},\lambda-i\tilde\alpha) N(\hat{\ubar w}\hat{\ubar L}^*,\lambda+i\tilde\alpha).
  \end{equation}
  Note now that $N(\hat{\ubar L}\hat{\ubar w},\mu)=N(\hat{\ubar w}\hat{\ubar L}^*,\bar\mu-i n)^*$; this follows from the fact that the adjoint of $\ubar\rho_\circ^{-i\mu}\hat{\ubar L}\hat{\ubar w}\ubar\rho_\circ^{i\mu}$ with respect to a dilation-invariant b-density on $[0,\infty)_{\ubar\rho_\circ}\times\pa\hat X$---which is thus $\ubar\rho_\circ^n$ times the Euclidean density---is given by $\ubar\rho_\circ^{-n}\ubar\rho_\circ^{-i\bar\mu}\hat{\ubar w}\hat{\ubar L}^*\ubar\rho_\circ^{i\bar\mu}\ubar\rho_\circ^n$. For $\mu=\lambda+i\tilde\alpha$ with $\Im\lambda=-\frac{n}{2}$, thus $\bar\lambda=\lambda+i n$, we have $\bar\mu-i n=\lambda-i\tilde\alpha$, and thus we can write
  \[
    \cN(\alpha_\circ,\lambda)= N(\hat{\ubar w}\hat{\ubar L}^*,\lambda+i\tilde\alpha)^*N(\hat{\ubar w}\hat{\ubar L}^*,\lambda+i\tilde\alpha),\qquad \Im\lambda=-\frac{n}{2}.
  \]
  But since $\Im(\lambda+i\tilde\alpha)=\alpha_\circ-n+2<0$, Lemma~\ref{LemmaCEAfSpecb} implies the invertibility of $\cN(\alpha_\circ,\lambda)$ for $\Im\lambda=-\frac{n}{2}$. We may thus improve~\eqref{EqCEAfSolvEstPre} to
  \begin{align*}
    &\|(\hat f^*,\hat j^*)\|_{\rho_\circ^{\frac{n}{2}}H_{\bop,0 0}^{s+2}\oplus \rho_\circ^{\frac{n}{2}}H_{\bop,0 0}^{s+1}} \\
    &\qquad \leq C\Bigl( \|\cL_{\hat\gamma,\hat k}\cL_{\hat\gamma,\hat k}^*(\hat f^*,\hat j^*)\|_{\rho^{\frac{n}{2}}H_{\bop,0 0}^{s-2}\oplus \rho^{\frac{n}{2}}H_{\bop,0 0}^{s-1}} + \|(\hat f^*,\hat j^*)\|_{\rho_\circ^{\frac{n}{2}-\delta}\hat\rho_2^{-1}H_{\bop,0 0}^{s_0+2}\oplus\rho_\circ^{\frac{n}{2}-\delta}\hat\rho_2^{-1}H_{\bop,0 0}^{s_0+1}}\Bigr)
  \end{align*}
  where $0<\delta<\Re\cE$; the error term is now relatively compact.

  We next argue that the error term can be dropped altogether, which follows if we prove that $\cL_{\hat\gamma,\hat k}\cL_{\hat\gamma,\hat k}^*(\hat f^*,\hat j^*)=0$ for $(\hat f^*,\hat j^*)\in\rho_\circ^{\frac{n}{2}}H_{\bop,0 0}^{s+2}\oplus\rho_\circ^{\frac{n}{2}}H_{\bop,0 0}^{s+1}$ implies $(\hat f^*,\hat j^*)=0$. First, elliptic regularity implies that $(\hat f^*,\hat j^*)$ has infinite b-00-regularity and vanishes to infinite order at $\hat\rho_2^{-1}(0)$. But by the comment following~\eqref{EqCEAfSolvNormOp}, we may then integrate by parts and deduce $0=\la\cL_{\hat\gamma,\hat k}\cL_{\hat\gamma,\hat k}^*(\hat f^*,\hat j^*),(\hat f^*,\hat j^*)\ra_{L^2}=\|\cL_{\hat\gamma,\hat k}^*(\hat f^*,\hat j^*)\|_{L^2}^2$; therefore, $L_{\hat\gamma,\hat k}^*(\tilde f^*,\tilde j^*)=0$ for $(\tilde f^*,\tilde j^*):=\hat w_2\rho_\circ^{-\tilde\alpha}\hat\rho_2^{-\alpha_2}e^{\beta/\hat\rho_2}(\hat f^*,\hat j^*)$. But since $(\hat f^*,\hat j^*)\in\cA^{\frac{n}{2}}([0,R_0^{-1})\times\pa\hat X)$ by Corollary~\ref{CorBAnSob}, we have $(\tilde f^*,\tilde j^*)\in\cA^{\frac{n}{2}-\tilde\alpha}([0,R_0^{-1})\times\pa\hat X)=\cA^{-(\alpha_\circ-n+2)}$. Lemma~\ref{LemmaCEAfCoker} gives $(\tilde f^*,\tilde j^*)=0$, and therefore $(\hat f^*,\hat j^*)=0$.

  The dual space of $\rho_\circ^{\frac{n}{2}}H_{\bop,0 0}^s(\hat\Omega)$ \emph{with respect to the $L^2$-inner product with volume density} $|\dd\hat\gamma|$ (which we stress was also used in the definition of $\cL_{\hat\gamma,\hat k}^*$) is $\rho_\circ^{\frac{n}{2}}H_{\bop,0 0}^{-s}(\hat\Omega)$, i.e.\ the weight remains unchanged. Thus, $\cL_{\hat\gamma,\hat k}\cL_{\hat\gamma,\hat k}^*$ has trivial cokernel too, and we deduce the invertibility of~\eqref{EqCEAfSolvInv}. This proves part~\eqref{ItCEAfSolv1}.

  For part~\eqref{ItCEAfSolvPhg}, we claim that $(\hat h,\hat q)=e^{-\beta/\hat\rho_2}\hat\rho_2^{\alpha_2}\rho_\circ^{\tilde\alpha}\hat w\cL_{\hat\gamma,\hat k}^*(h',q')$ is the desired polyhomogeneous solution, where
  \[
    \cL_{\hat\gamma,\hat k}\cL_{\hat\gamma,\hat k}^*(h',q')=(f',j'):=\hat w_2\rho_\circ^{-\tilde\alpha}\hat\rho_2^{-\alpha_2}e^{\beta/\hat\rho_2}(\hat f,\hat j) \in \cA^{\cG'}_\phg(\hat X),\qquad \cG':=\cG+2-\tilde\alpha.
  \]
  Note that $(f',j')$ vanishes near $r\leq R_0$, and note that $\Re\cG'>\frac{n}{2}$. The extension by $0$ of $h',q'$ to $\hat r\leq\hat R_0$ satisfies $(h',q')\in\rho^{\frac{n}{2}}\Hb^\infty(\hat X)\subset\cA^{\frac{n}{2}}(\hat X)$. In view of Lemma~\ref{LemmaCEAfMap}\eqref{ItCEAfMapLin}, it thus suffices to show that $(h',q')$ is polyhomogeneous with index set $\cG'\extcup\cS'$, where $\cS'\subset\C\times\N_0$ with $\Re\cS'>\frac{n}{2}$ depends only on $\hat\gamma$, $\hat k$, and $\alpha_\circ$. This is a standard result for elliptic b-differential operators such as $\cL:=\cL_{\hat\gamma,\hat k}\cL_{\hat\gamma,\hat k}^*$, see e.g.\ \cite[Proposition~5.61]{MelroseAPS} in the case that $\cL$ has smooth (rather than smooth plus decaying polyhomogeneous) coefficients. Briefly, one rewrites $\cL(h',q')=(f',j')$ in terms of its b-normal operator $N(\cL)$ at $\pa\hat X$ as
  \begin{equation}
  \label{EqCEAfSolvPhgN}
    N(\cL)(h',q')=(f',j')-(\cL-N(\cL))(h',q')
  \end{equation}
  and extracts the polyhomogeneous expansion of $(h',q')$ at $\pa\ol{\R^n}$ iteratively, noting that the coefficients of $\cL-N(\cL)$ are of class $\cA_\phg^\cE$ (and thus decay). Indeed, having obtained an expansion up to conormal errors with decay rate $N>\frac{n}{2}$, the Mellin-transform of $(h',q')$ is meromorphic in $\Im\lambda>-N$ (see the end of~\S\ref{SsBFn}). One then passes to the Mellin transform and inverts $N(\cL,\lambda)$, or equivalently the operator $\cN(\alpha_\circ,\lambda)$ in~\eqref{EqCEAfSolvNormOp}; the set of poles of $\cN(\alpha_\circ,\lambda)^{-1}$ is finite in any strip of bounded $\Im\lambda$ (cf.\ \eqref{EqCEAfCokerEst}). This demonstrates that the Mellin-transform of $(h',q')$ is meromorphic, with finitely many poles, in the larger half space $\Im\lambda>-N-\delta$ (where $0<\delta<\Re\cE$), since this is true for the right hand side of~\eqref{EqCEAfSolvPhgN}. This shows the polyhomogeneity of $(h',q')$ modulo conormal errors with the faster rate $N+\delta$. This finishes the proof, except $(\hat h,\hat q)$ vanishes only in $\hat r\leq\hat R_0$; to get the vanishing in a neighborhood of $\hat r\leq\hat R_0$, one applies the above arguments with $\hat R_0$ replaced by $\hat R_0+\eta$ for $0<\eta\ll 1$.
\end{proof}

\subsection{Manifolds with punctures}
\label{SsCEPct}

We now turn to the analysis of the linearized constraints map on a manifold $X_\circ=[X;\{\fp\}]$ arising via the blow-up of a smooth $n$-dimensional manifold $X$ (without boundary) at a point; here $n\geq 3$.

First, however, let $\gamma,k\in\CI(X;S^2 T^*X)$, with $\gamma$ Riemannian; they do not need to satisfy the constraint equations (though in our application they do).

\begin{definition}[No KIDs]
\label{DefCEPctKID}
  (See~\cite{MoncriefLinStabI}.) Denote by $L_{\gamma,k}$ the linearization of the constraints map $P(-,-;\Lambda)$ (see Lemma~\usref{LemmaCELin}) at $(\gamma,k)$, and let $\cU^\circ\subset X$ be a smoothly bounded precompact connected open set containing $\fp$. We then say that $(X,\gamma,k)$ \emph{has no KIDs} (`Killing Initial Data') in $\cU^\circ$ if the kernel of $L_{\gamma,k}^*$ on $\CI(\cU;\ul\R\oplus T^*X)$ is trivial, where $\cU=\ol{\cU^\circ}$.
\end{definition}

Beig--Chru\'sciel--Schoen showed in \cite[Theorem~1.3]{BeigChruscielSchoenKIDs} that the subset of initial data sets without KIDs is generic (open and dense) for a variety of classes of initial data sets.

Write now $\upbeta_\circ\colon X_\circ\to X$ for the blow-down map. Denote by $\hat\rho\in\CI(X_\circ)$ a defining function of $\pa X_\circ$. In geodesic normal coordinates $x=(x^1,\ldots,x^n)$, $|x|<r_0$, around $\fp\in X$, we can locally take $\hat\rho=r:=|x|$ (Euclidean norm). Setting $\omega=\frac{x}{|x|}\in\Sph^{n-1}$, we identify a collar neighborhood of $\pa X_\circ$ with $[0,r_0)_r\times\Sph^{n-1}_\omega$; we furthermore fix an identification $\upbeta_\circ^*T^*X\cong T^*_\fp X$ in this collar neighborhood which is the identity over $\pa X_\circ$. Put
\begin{equation}
\label{EqCEPctWeights}
  w_\circ = \begin{pmatrix} \hat\rho^2 & 0 \\ 0 & \hat\rho \end{pmatrix},\qquad
  \ubar w_\circ = \begin{pmatrix} r^2 & 0 \\ 0 & r \end{pmatrix},
\end{equation}
where $\ubar w_\circ$ is an operator on $[0,\infty)_r\times\Sph^{n-1}_\omega$ acting on pairs of sections of the pullback of $S^2 T_\fp^*X$ along $(r,\omega)\mapsto\fp$.

\begin{lemma}[Linearized constraints map]
\label{LemmaCEPctLin}
  Denote the lift of $L_{\gamma,k}$ to\footnote{This lift is the differential operator on $X_\circ$ given by the original operator $L_{\gamma,k}$ acting on smooth compactly supported sections of $\upbeta_\circ^*(S^2 T^*X\oplus S^2 T^*X)$ over the interior $(X_\circ)^\circ$ of $X_\circ$.} $X_\circ$ by $L_{\circ,\gamma,k}$. Then
  \begin{equation}
  \label{EqCEPctLin}
    L_{\circ,\gamma,k}w_\circ \in \begin{pmatrix} \Diffb^2(X_\circ;\upbeta_\circ^*S^2 T^*X,\ul\R) & \hat\rho\CI(X_\circ;\Hom(\upbeta_\circ^*S^2 T^*X,\ul\R)) \\ \hat\rho\Diffb^1(X_\circ;\upbeta_\circ^*S^2 T^*X,\upbeta_\circ^*T^*X) & \Diffb^1(X_\circ;\upbeta_\circ^*S^2 T^*X,\upbeta_\circ^*T^*X) \end{pmatrix}.
  \end{equation}
  Letting $\fe=\sum_{j=1}^n(\dd x^j)^2$ in geodesic normal coordinates, the normal operator of $L_{\circ,\gamma,k}w_\circ$ at $\pa X_\circ$ is given by $\ubar L_\circ\ubar w_\circ$, where $\ubar L_\circ=\diag(\ubar L_{\circ,1},\ubar L_{\circ,2})$ with $\ubar L_{\circ,1}=\Delta_\fe\tr_\fe+\delta_\fe\delta_\fe$ and $\ubar L_{\circ,2}=\delta_\fe+\dd\tr_\fe$ (cf.\ \eqref{EqCEAfMapLot}).
\end{lemma}

The weight $w_\circ$ at $r=0$ now encodes the \emph{opposite} relative weighting of the linearized metric and the linearized second fundamental form as compared to the situation at infinity in Lemma~\ref{LemmaCEAfMap}: the latter is one order \emph{more} singular (less decaying) than the former.\footnote{This sign switch arises naturally from the total gluing space perspective, where it is an instance of the typical relationship between small ends of cones ($\pa X_\circ$ from the perspective of $X_\circ$) which emanate from the large end of a cone ($\pa\hat X$ from the perspective of $\hat X$).}

\begin{proof}[Proof of Lemma~\usref{LemmaCEPctLin}]
  Note that $L_{\gamma,k}$ is a smooth coefficient operator on $X$; and since in local coordinates on $X$ near $\fp$, the lift of $r\pa_{x^i}$ to $X_\circ$ is a smooth b-vector field, we have $\pa_{x^i}\in r^{-1}\Vb(X_\circ)$. In view of~\eqref{EqCEDNOrders2}, this gives~\eqref{EqCEPctLin}. Moreover, the only terms contributing to the normal operators of the $(1,1)$ and $(2,2)$ components of $L_{\circ,\gamma,k}w_\circ$ are those with the maximal number of derivatives---so $\Delta_\gamma\tr_\gamma+\delta_\gamma\delta_\gamma$ and $\delta_\gamma+\dd\tr_\gamma$---and we may moreover freeze the metric $\gamma$ at the point $\fp$ (since the difference of $\gamma$ and $\gamma(\fp)$ vanishes at $r=0$). This gives $\ubar L_\circ\ubar w_\circ$, as claimed.
\end{proof}

Note now that the Mellin-transformed normal operator family of $\ubar w_\circ\ubar L_\circ^*$ at $r=0$ satisfies
\[
  N(\ubar w_\circ\ubar L_\circ^*,\lambda) = N(\hat{\ubar w}\hat{\ubar L}^*,-\lambda),
\]
where the right hand side was studied in Lemma~\ref{LemmaCEAfSpecb}. Thus, $N(\ubar w_\circ\ubar L_\circ^*,\lambda)$ is injective for $\lambda\in\C\setminus\{-i,0\}$, and the estimate~\eqref{EqCEAfSpecbEst} holds for the operator $\ubar w_\circ\ubar L_\circ^*$ for $\lambda\in\Lambda\Subset\C\setminus\{-i,0\}$.

\begin{prop}[Solvability of the linearized constraints]
\label{PropCEPctSolv}
  In the notation of Definition~\usref{DefCEPctKID}, suppose that $(X,\gamma,k)$ has no KIDs in $\cU^\circ$. Let $\cU_\circ=\upbeta_\circ^{-1}(\cU)$ (where $\cU=\ol{\cU^\circ}$), and denote by $\rho_2\in\CI(\cU_\circ)$ a defining function of $\pa\cU$; put $w_2:=\diag(\rho_2^4,\rho_2^2)$. Define b-00-Sobolev spaces on $\cU_\circ$ (with b-character at $\pa X_\circ=\hat\rho^{-1}(0)$ and 00-character at $\pa\cU$) using a positive b-00-density on $\cU_\circ$. Let $\hat\alpha,\alpha_2,\beta\in\R$, with $\hat\alpha\neq-n+2,-n+1$.
  \begin{enumerate}
  \item\label{ItCEPctSolv1}{\rm (Basic solvability.)} Put
    \begin{equation}
    \label{EqCEPctSolvOp}
      \cL_{\circ,\gamma,k} := e^{\beta/\rho_2}\rho_2^{-\alpha_2}\hat\rho^{-\tilde\alpha}w_2 L_{\circ,\gamma,k}w_\circ\hat\rho^{\tilde\alpha}\rho_2^{\alpha_2}e^{-\beta/\rho_2},\qquad \tilde\alpha:=\hat\alpha+n-2-\frac{n}{2}.
    \end{equation}
    Define adjoints with respect to the volume density and fiber inner products induced by $\gamma$. Then for all $s\in\R$, the operator $\cL_{\circ,\gamma,k}\cL_{\circ,\gamma,k}^*$ is invertible as a map\footnote{Paralleling Remark~\ref{RmkCEAfSolvRightInverse}, this gives a right inverse of $L_{\circ,\gamma,k}$ which is bounded as a map
    \begin{align*}
      &\hat\rho^{\hat\alpha-2}\rho_2^{\alpha_2-4}e^{-\beta/\rho_2}H_{\bop,0 0}^{s-2}(\cU_\circ) \oplus \hat\rho^{\hat\alpha-2}\rho_2^{\alpha_2-2}e^{-\beta/\rho_2}H_{\bop,0 0}^{s-1}(\cU_\circ;\upbeta_\circ^*T^*_\cU X) \\
      &\quad \to \hat\rho^{\hat\alpha}\rho_2^{\alpha_2}e^{-\beta/\rho_2}H_{\bop,0 0}^s(\cU_\circ;\upbeta_\circ^*S^2 T^*_\cU X) \oplus \hat\rho^{\hat\alpha-1}\rho_2^{\alpha_2}e^{-\beta/\rho_2}H_{\bop,0 0}^s(\cU_\circ;\upbeta_\circ^*S^2 T^*_\cU X).
    \end{align*}}
     \begin{equation}
     \label{EqCEPctSolv1Inv}
     \begin{split}
       \cL_{\circ,\gamma,k}\cL_{\circ,\gamma,k}^* &\colon \hat\rho^{-\frac{n}{2}}H_{\bop,0 0}^{s+2}(\cU_\circ)\oplus\hat\rho^{-\frac{n}{2}}H_{\bop,0 0}^{s+1}(\cU_\circ;\upbeta_\circ^*T^*X) \\
         &\quad\qquad \to \hat\rho^{-\frac{n}{2}}H_{\bop,0 0}^{s-2}(\cU_\circ)\oplus\hat\rho^{-\frac{n}{2}}H_{\bop,0 0}^{s-1}(\cU_\circ;\upbeta_\circ^*T^*X).
     \end{split}
     \end{equation}
  \item\label{ItCEPctSolvPhg}{\rm (Polyhomogeneity.)} There exists an index set $\cS_\circ\subset\C\times\N_0$, only depending on $\hat\alpha$ and satisfying $\Re\cS_\circ>\hat\alpha$ so that the following holds. If $\cG\subset\C\times\N_0$ is an index set with $\Re\cG>\hat\alpha$, and if $(f,j)\in\cA_\phg^{\cG-2}(X_\circ;\ul\R\oplus\upbeta_\circ^*T^*X)$ vanishes in a neighborhood of $\pa\cU$, then there exists
  \[
    (h,q)\in\cA_\phg^{\cG\extcup\cS_\circ}(X_\circ;\upbeta_\circ^*S^2 T^*X)\oplus\cA_\phg^{(\cG\extcup\cS_\circ)-1}(X_\circ;\upbeta_\circ^*S^2 T^*X)
  \]
  with $L_{\circ,\gamma,k}(h,q)=(f,j)$, and so that $(h,q)$ vanishes in a neighborhood of $X_\circ\setminus\cU_\circ$.
  \end{enumerate}
\end{prop}
\begin{proof}
  The proof of part~\eqref{ItCEPctSolv1} is very close to that of Proposition~\ref{PropCEAfSolv}\eqref{ItCEAfSolv1}, except for the analysis of the cokernel. The operator $\cL_{\circ,\gamma,k}$ is a smooth coefficient b-00-differential operator on $\cU_\circ$, and the principal symbol of $\cL_{\circ,\gamma,k}^*$ is injective. For the Douglis--Nirenberg-elliptic operator $\cL_{\circ,\gamma,k}^*\cL_{\circ,\gamma,k}$, we thus have the estimate~\eqref{EqCEAfSolvEstPre} where we replace $\rho_\circ$ and $\frac{n}{2}$ by $\hat\rho$ and $-\frac{n}{2}$, respectively. Improving the error term requires the invertibility of the Mellin-transformed normal operator on the line $\Im\lambda=\frac{n}{2}$, i.e.\ the invertibility of $N(\ubar w_\circ\ubar L_\circ^*,\lambda+i\tilde\alpha)^*N(\ubar w_\circ\ubar L_\circ^*,\lambda+i\tilde\alpha)$, which holds because of $\Im(\lambda+i\tilde\alpha)=\hat\alpha+n-2\notin\{0,-1\}$. This gives the estimate
  \begin{equation}
  \label{EqCEPctSolvOpEst}
  \begin{split}
    \|(f^*,j^*)\|_{\hat\rho^{-\frac{n}{2}}H_{\bop,0 0}^{s+2}\oplus \hat\rho^{-\frac{n}{2}}H_{\bop,0 0}^{s+1}} &\leq C\Bigl(\|\cL_{\circ,\gamma,k}\cL_{\circ,\gamma,k}^*(f^*,j^*)\|_{\hat\rho^{-\frac{n}{2}}H_{\bop,0 0}^{s-2}\oplus \hat\rho^{-\frac{n}{2}}H_{\bop,0 0}^{s-1}} \\
      &\quad\qquad + \|(f^*,j^*)\|_{\hat\rho^{-\frac{n}{2}-1}\rho_2^{-1}H_{\bop,0 0}^{s_0+2}\oplus\hat\rho^{-\frac{n}{2}-1}\rho_2^{-1}H_{\bop,0 0}^{s_0+1}}\Bigr).
  \end{split}
  \end{equation}
  Once we prove the triviality of the kernel of $\cL_{\circ,\gamma,k}^*$, one can drop the relatively compact error term on the right. But if $\cL_{\circ,\gamma,k}^*(f^*,j^*)=0$, then $(f^*,j^*)$ enjoys infinite b-00-regularity and infinite decay at $\rho_2^{-1}(0)$, and in particular it is smooth in $(\cU_\circ)^\circ\setminus\pa X_\circ$ and conormal at $\pa X_\circ$. Moreover, $(\tilde f^*,\tilde j^*):=w_2\hat\rho^{-\tilde\alpha}\rho_2^{-\alpha_2}e^{\beta/\rho_2}(f^*,j^*)$ satisfies $L_{\circ,\gamma,k}^*(\tilde f^*,\tilde j^*)=0$. By Lemma~\ref{LemmaCEProlong}, $(\tilde f^*,\tilde j^*)$ is necessarily smooth down to $\pa\cU$. Near $\pa X_\circ$ on the other hand, where $(\tilde f^*,\tilde j^*)\in\cA^{-\frac{n}{2}-\tilde\alpha}=\cA^{-(\hat\alpha+n-2)}$ is polyhomogeneous (as follows from the same argument, \emph{mutatis mutandis}, that we already used in the proof of Proposition~\ref{PropCEAfSolv}\eqref{ItCEAfSolvPhg}), we write
  \[
    \ubar L_\circ^*(\tilde f^*,\tilde j^*) = -(L_{\circ,\gamma,k}^*-\ubar L_\circ^*)(\tilde f^*,\tilde j^*)
  \]
  and use the injectivity of $N(\ubar w_\circ\ubar L_\circ^*,\lambda)$ for $\lambda\neq -i,0$ to conclude that\footnote{If $-(\hat\alpha+n-2)>1$, we could even conclude that $(\tilde f^*,\tilde j^*)$ vanishes to infinite order.} $(\tilde f^*,\tilde j^*)\in\cA^{-\eps}$ for all $\eps>0$. Thus, $(\tilde f^*,\tilde j^*)$ is the pullback of a distribution $(f^\sharp,j^\sharp)$ on $\cU$ which is smooth away from $\fp$ and satisfies $|D_x^\alpha(f^\sharp,j^\sharp)|\lesssim|x|^{-\eps-|\alpha|}$ near $\fp$ for all $\alpha$. Therefore, $(f^\sharp,j^\sharp)\in H^{\frac{n}{2}-\eps}(\cU)$. But $L_{\gamma,k}^*(f^\sharp,j^\sharp)\in H^{\frac{n}{2}-2-\eps}(\cU)$ vanishes on $\cU\setminus\{\fp\}$ and is therefore a distribution with support in $\{\fp\}$; since $\frac{n}{2}-2>-\frac{n}{2}$, this forces $L_{\gamma,k}^*(f^\sharp,j^\sharp)=0$. Again applying Lemma~\ref{LemmaCEProlong}, this implies that $(f^\sharp,j^\sharp)$ is smooth on $\cU$. Using the absence of KIDs, this finally implies $(f^\sharp,j^\sharp)=0$ and thus $(f^*,j^*)=0$.

  We have now shown the invertibility of~\eqref{EqCEPctSolv1Inv}, completing the proof of part~\eqref{ItCEPctSolv1}. As already remarked above, the proof of part~\eqref{ItCEPctSolvPhg} is the same as that of Proposition~\ref{PropCEAfSolv}\eqref{ItCEAfSolvPhg}.
\end{proof}

\begin{rmk}[Explicit index set of $(h,q)$]
\label{RmkCEPctIndex}
  An alternative argument for Proposition~\ref{PropCEPctSolv}\eqref{ItCEPctSolvPhg} proceeds as follows. One first solves the equation $L_{\circ,\gamma,k}(h,q)=(f,j)$ formally, i.e.\ modulo smooth errors vanishing to infinite order at $\pa X_\circ$, using the fact that $N(\ubar w_\circ\ubar L_\circ,\lambda)$ is surjective for all $\lambda$ except for $\lambda=i(n-1),i n$. (The presence of a cokernel for $\lambda=i(n-1),i n$ necessitates the introduction of logarithmic terms.) To solve away the remaining error, which is the lift of a smooth tensor on $\cU$, one can blow down $\pa X_\circ$, i.e.\ work on $\cU$; the resulting correction to the formal solution is smooth on $\cU$, and thus lifts to be polyhomogeneous with index set $\N_0$ on $X_\circ$. We leave the details to the interested reader.
\end{rmk}

\subsection{Constraints map on the total gluing space}
\label{SsCETot}

We use the notation of~\S\ref{SG}; near $\fp\in X$, we work with geodesic normal coordinates $x=(x^1,\ldots,x^n)$, and we recall the local coordinate charts $(\eps,\hat x)=(\eps,\frac{x}{\eps})$ near $\hat X^\circ$ and $(\hat\rho,\rho_\circ,\omega)=(|x|,\frac{\eps}{|x|},\frac{x}{|x|})$ near $\hat X\cap X_\circ$. We first note:

\begin{lemma}[Scaling]
\label{LemmaCETotScale}
  For $\lambda>0$, we have $P(\lambda^2\gamma,\lambda k;\lambda^{-2}\Lambda)=\diag(\lambda^{-2},\lambda^{-1}) P(\gamma,k;\Lambda)$.
\end{lemma}
\begin{proof}
  Since $\Ric(\lambda^2\gamma)=\Ric(\gamma)$, we have $R_{\lambda^2\gamma}=\lambda^{-2}R_\gamma$. Replacing in $|k|_\gamma^2=\gamma^{i i'}\gamma^{j j'}k_{i j}k_{i' j'}$ the term $\gamma^{i i'}$ by $\lambda^{-2}\gamma^{i i'}$ and $k_{i j}$ by $\lambda k_{i j}$, one similarly picks up an overall factor of $\lambda^{-2}$. For the second component of $P(\lambda^2\gamma,\lambda k;\lambda^{-2}\Lambda)$, we note that $\dd(\tr_{\lambda^2\gamma}\lambda k)=\lambda^{-1}\,\dd\tr_\gamma k$, similarly $\delta_{\lambda^2\gamma}(\lambda k)=\lambda^{-1}\delta_\gamma k$.
\end{proof}

In a collar neighborhood of $\hat X$ and recalling the scaling map $\sfs$ from Definition~\ref{DefGTotScale}, we deduce, for $\lambda=\eps^{-1}$, that
\begin{equation}
\label{EqCETotsfs}
  \eps^2 P(\wt\gamma,\wt k;\Lambda) = \sfs\bigl( P (\sfs^{-1}\wt\gamma,\sfs^{-1}(\eps\wt k); \eps^2\Lambda \bigr);
\end{equation}
note here that the right hand side is $\diag(1,\eps)P(\eps^{-2}\wt\gamma,\eps^{-1}\wt k;\eps^2\Lambda)$. Conversely, for $\lambda=\eps$, we find that $P(\hat\gamma,\hat k;0)=0$ implies $P(\sfs\hat\gamma,\eps^{-1}\sfs\hat k;0)=0$; notice the $\eps^{-1}$ scaling of the second fundamental form.

\begin{definition}[Total families]
\label{DefCETotID}
  Let $\hat\cE,\cE\subset\C\times\N_0$ be nonlinearly closed index sets with $\Re\hat\cE>0$ and $\Re\cE>0$. Let $\hat K\Subset\hat X^\circ$ be a compact (possibly empty) subset, smoothly bounded and with connected complement, and set $\wt K:=\{(\eps,\hat x)\colon\hat x\in\hat K\}\subset\wt X$. Then a pair $(\wt\gamma,\wt k)$ of sections of $S^2\wt T^*\wt X\to\wt X\setminus\wt K^\circ$ is called an \emph{$(\hat\cE,\cE)$-smooth total family} if
  \begin{align*}
    \wt\gamma&=\wt\upbeta^*\gamma+\wt\gamma^{(1)}, \\
    \wt k&=\wt\upbeta^*k+\wt k^{(1)},
  \end{align*}
  where, with the index sets referring to $\hat X$ and $X_\circ$ (in this order),
  \begin{alignat*}{2}
    \gamma&\in\CI(X;S^2 T^*X),&\qquad
    \wt\gamma^{(1)}&\in\cA_\phg^{\N_0\cup\hat\cE,\cE}(\wt X\setminus\wt K^\circ;S^2\wt T^*\wt X), \\
    k&\in\CI(X;S^2 T^*X),&\qquad
    \wt k^{(1)}&\in\cA_\phg^{(\N_0\cup\hat\cE)-1,\cE}(\wt X\setminus\wt K^\circ;S^2\wt T^*\wt X),
  \end{alignat*}
  with $\gamma,k$ regarded as $\eps$-independent elements of $\CI(\wt X';S^2\wt T^*\wt X')$, and where $\wt\gamma|_{\hat X}$ and $\gamma$ (and thus $\wt\gamma|_{X^\circ}=\upbeta_\circ^*\gamma$) are positive definite. The \emph{boundary data} of $(\wt\gamma,\wt k)$ are $(\hat\gamma,\hat k)$ and $(\gamma,k)$, where
  \begin{equation}
  \label{EqCETotRestr}
    (\hat\gamma,\hat k) := \bigl(\hat\sfs^{-1}(\wt\gamma|_{\hat X}),\hat\sfs^{-1}(\eps\wt k)|_{\hat X}\bigr), \qquad
    \upbeta_\circ^*(\gamma,k) = (\wt\gamma,\wt k)|_{X_\circ}.
  \end{equation}
\end{definition}

See Figure~\ref{FigCETotID}.

\begin{figure}[!ht]
\centering
\includegraphics{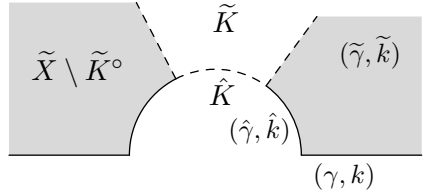}
\caption{Illustration of a total family $(\wt\gamma,\wt k)$ on $\wt X\setminus\wt K^\circ$, with its boundary data $(\gamma,k)$ and $(\hat\gamma,\hat k)$ on $\hat X\setminus\hat K^\circ$.}
\label{FigCETotID}
\end{figure}

Recalling from Lemma~\ref{LemmaGTot}\eqref{ItGTothatX} that $\hat X^\circ$ is a vector space, denote by
\[
  \hat\fe\in\CI(\hat X;S^2\,\Tsc^*\hat X)
\]
the constant (i.e.\ translation-invariant) metric given by $\hat\sfs^{-1}(\wt\upbeta^*\wt\gamma_{(0)}(\fp))$ where $\hat\sfs$ is given in Definition~\ref{DefGTotScale}. Since $\wt\gamma_{(0)}(\fp)=\sum_{j=1}^n(\dd x^j)^2$, we have $\hat\fe=\sum_{j=1}^n(\dd\hat x^j)^2$, so $(\hat X^\circ,\hat\fe)$ is isometric to $\R^n$ with the Euclidean metric. The boundary data of $(\wt\gamma,\wt k)$ at $\hat X$ then satisfy
\[
  (\hat\gamma-\hat\fe,\hat k) \in \cA_\phg^\cE(\hat X\setminus\hat K^\circ;S^2\,\Tsc^*\hat X) \oplus \cA_\phg^{\cE+1}(\hat X\setminus\hat K^\circ;S^2\,\Tsc^*\hat X).
\]
 Thus, $(\hat\gamma,\hat k)$ is $\cE$-asymptotically flat in the sense of Definition~\ref{DefCEAf}. The pair $(\gamma,k)$ on $X$ immediately fits into the setup of~\S\ref{SsCEPct}.

 The goal of this paper is to construct a total family $(\wt\gamma,\wt k)$, with prescribed boundary data $(\hat\gamma,\hat k)$ and $(\gamma,k)$, which satisfies the constraint equations\footnote{The tensor $P(\wt\gamma,\wt k;\Lambda)$ on $\wt X\setminus\wt K^\circ$ is equal to $P(\wt\gamma|_{\wt X_\eps\setminus\wt K^\circ},\wt k|_{\wt X_\eps\setminus\wt K^\circ};\Lambda)$ on $\wt X_\eps\setminus\wt K^\circ$.} $P(\wt\gamma,\wt k;\Lambda)=0$ for small $\eps>0$. Note here that $\wt\gamma$ is automatically positive definite for all sufficiently small $\eps>0$; in the sequel, when we write $P(\wt\gamma,\wt k;\Lambda)$, we shall implicitly always restrict to the parameter range $\eps\leq\eps_0$ where $\eps_0>0$ is so small that $\wt\gamma$ is Riemannian. Our first task is to analyze the constraints map on total families, and the structure and normal operators of its linearization at a total family.

\begin{definition}[Total weight]
\label{DefCETotWeight}
  With $\hat\rho\in\CI(\wt X)$ denoting a defining function of $\hat X$, we set
  \[
    \wt w := \begin{pmatrix} \hat\rho^2 & 0 \\ 0 & \hat\rho \end{pmatrix} \in \CI\bigl(\wt X^\circ;\End(S^2\wt T^*\wt X\oplus S^2\wt T^*\wt X)\bigr).
  \]
\end{definition}

This matches $w_\circ$ in~\eqref{EqCEPctWeights} at $X_\circ$. Since $\hat\rho=\eps\rho_\circ^{-1}$ for a defining function $\rho_\circ\in\CI(\wt X)$ of $X_\circ$, we have $\wt w=\diag(\eps^2\rho_\circ^{-2},\eps\rho_\circ^{-1})=\diag(\eps^2,\eps)\hat w$, which at $\hat X$ thus combines the weight in Definition~\ref{DefCEAfWeight} with the scaling from Lemma~\ref{LemmaCETotScale} (with $\lambda=\eps^{-1}$).

\begin{prop}[Constraints map on the total gluing space]
\label{PropCETot}
  Let $(\wt\gamma,\wt k)$ be a $(\hat\cE,\cE)$-smooth total family, and define $(\hat\gamma,\hat k)$ and $(\gamma,k)$ by~\eqref{EqCETotRestr}.
  \begin{enumerate}
  \item\label{ItCETotPhg}{\rm (Nonlinear constraints.)} We have $P(\wt\gamma,\wt k;\Lambda)=\wt\upbeta^*p + p^{(1)}$, where, in the notation of Definition~\usref{DefCETotID},
    \[
       p:=P(\gamma,k;\Lambda)\in\CI(X;\ul\R\oplus T^*X),\qquad
       p^{(1)}\in\cA_\phg^{(\N_0\cup\hat\cE)-2,\cE}(\wt X\setminus\wt K^\circ;\ul\R\oplus\wt T^*\wt X).
     \]
     Moreover,
     \begin{equation}
     \label{EqCETotPhgLot}
       (\eps^2 P(\wt\gamma,\wt k;\Lambda))|_{\hat X} = (\eps^2 p^{(1)})|_{\hat X} = \hat\sfs\bigl(P(\hat\gamma,\hat k;0)\bigr).
     \end{equation}
  \item\label{ItCETotLin}{\rm (Linearization as a q-differential operator.)} The linearization $L_{\wt\gamma,\wt k}$ of $P(-,-;\Lambda)$ at $(\wt\gamma,\wt k)$ satisfies
  \begin{equation}
  \label{EqCETotLin}
  \begin{split}
    L_{\wt\gamma,\wt k}\wt w &\in(\wt\upbeta^*\CI(X){+}\cA_\phg^{\N_0\cup\hat\cE,\cE}(\wt X\setminus\wt K^\circ)) \\
    &\qquad \times\begin{pmatrix}
      \Diffq^2(\wt X\setminus\wt K^\circ;S^2\wt T^*\wt X,\ul\R) & \CI(\wt X\setminus\wt K^\circ;\Hom(S^2\wt T^*\wt X,\ul\R)) \\
      \Diffq^1(\wt X\setminus\wt K^\circ;S^2\wt T^*\wt X,\wt T^*\wt X) & \Diffq^1(\wt X\setminus\wt K^\circ;S^2\wt T^*\wt X,\wt T^*\wt X)
    \end{pmatrix}.
  \end{split}
  \end{equation}
  Using the notation of~\S\usref{SsGq} and Lemma~\usref{LemmaCEPctLin}, we have
  \begin{equation}
  \label{EqCETotLinNorm}
    \hat\sfs^{-1}\circ N_{\hat X}(L_{\wt\gamma,\wt k}\wt w)\circ\hat\sfs=L_{\hat\gamma,\hat k}\hat w,\qquad
    N_{X_\circ}(L_{\wt\gamma,\wt k}\wt w)=L_{\circ,\gamma,k}w_\circ.
  \end{equation}
  Moreover, omitting the bundles and differential orders,
  \begin{equation}
  \label{EqCETotLinNormDiff}
  \begin{split}
    L_{\wt\gamma,\wt k}\wt w - \hat\chi\sfs\circ L_{\hat\gamma,\hat k}\hat w\circ\sfs^{-1} &\in \cA_\phg^{(\N_0+1)\cup\hat\cE,\N_0\cup\cE}\Diffq(\wt X\setminus\wt K^\circ), \\
    L_{\wt\gamma,\wt k}\wt w-\chi_\circ L_{\circ,\gamma,k}w_\circ &\in \cA_\phg^{\N_0\cup\hat\cE,\cE}\Diffq(\wt X\setminus\wt K^\circ).
  \end{split}
  \end{equation}
  Here, $\hat\chi L_{\hat\gamma,\hat k}$ and $\chi_\circ L_{\circ,\gamma,k}$ are regarded as differential operators on $\wt X$ via Lemma~\usref{LemmaGTotRel}.
  \end{enumerate}
\end{prop}
\begin{proof}
  The arguments for part~\eqref{ItCETotPhg} are similar to those in the proof of Lemma~\ref{LemmaCEAfMap}. Near $\hat X^\circ$, we work in the coordinates $(\eps,\hat x)$. Since $\wt\gamma_{i j}\in\cA_\phg^{\N_0\cup\hat\cE}([0,1)_\eps\times\hat X^\circ)$, we have $\wt\gamma^{i j}\in\cA_\phg^{\N_0\cup\hat\cE}$. Since $\pa_\ell\wt\gamma_{i j}\in\cA_\phg^{(\N_0\cup\hat\cE)-1}$, this gives $\Gamma(\wt\gamma)_{i j}^\ell\in\cA_\phg^{(\N_0\cup\hat\cE)-1}$. We moreover have $\wt k_{i j}\in\cA_\phg^{(\N_0\cup\hat\cE)-1}$ and thus $\pa_\ell\wt k_{i j}\in\cA_\phg^{(\N_0\cup\hat\cE)-2}$; this gives $R(\wt\gamma)$, $|\tr_{\wt\gamma}\wt k|^2$, $|\wt k|_{\wt\gamma}^2$, $\nabla_\ell\wt k_{i j}\in\cA_\phg^{(\N_0\cup\hat\cE)-2}$ (where $\nabla$ is given by $\nabla^{\wt g_\eps}$ on $\wt X_\eps$). Therefore, $P(\wt\gamma,\wt k)\in\cA_\phg^{(\N_0\cup\hat\cE)-2}$. Tracking the leading order terms of $\wt\gamma$, $\wt k$, and their derivatives at $\eps=0$, one finds that the restriction of $\eps^2 P_1(\wt\gamma,\wt k;\Lambda)$, resp.\ $\eps^2 P_2(\wt\gamma,\wt k)_i$ to $\eps=0$ is given by $P_1(\hat\gamma,\hat k;\Lambda)$, resp.\ $P_2(\wt\gamma,\wt k)_{\hat i}$ indeed, where we write $\hat i$ for indices in the $\hat x$-variables; this uses that $\pa_i=\eps^{-1}\pa_{\hat i}$. More conceptually, Lemma~\ref{LemmaCETotScale} gives, in a collar neighborhood of $\hat X$,
  \begin{equation}
  \label{EqCETotPfScale}
  \begin{split}
    \eps^2 P(\wt\gamma,\wt k;\Lambda)&=\diag(1,\eps) P(\eps^{-2}\wt\gamma,\eps^{-1}\wt k;\eps^2\Lambda) \\
      &= \diag(1,\eps)P\bigl(\sfs^{-1}\wt\gamma,\sfs^{-1}(\eps\wt k);\eps^2\Lambda\bigr) \\
      & = \sfs\bigl( P(\sfs^{-1}\wt\gamma,\sfs^{-1}(\eps\wt k);\eps^2\Lambda\bigr)\bigr)
  \end{split}
  \end{equation}
  as sections of $\ul\R\oplus\wt T_{\hat X}^*\wt X$; restriction to $\hat X$ gives~\eqref{EqCETotPhgLot}.

  Near $(X_\circ)^\circ$, note that $\eps\mapsto(\wt\gamma|_{\wt X_\eps},\wt k|_{\wt X_\eps})$, composed with $\wt X_\eps\cong X$, is a equal to an $\eps$-independent smooth symmetric 2-tensor on $X\setminus\{\fp\}$ plus a polyhomogeneous family (with index set $\cE$) of such tensors, and the conclusion follows via arguments much as in the previous paragraph. Near the corner $X_\circ\cap\hat X$ finally, we have a combination of both settings. Indeed,
  \[
    \wt\gamma_{i j}(\hat\rho,\rho_\circ,\omega) = \gamma_{i j}(\hat\rho\omega) + \wt\gamma^{(1)}_{i j}(\hat\rho,\rho_\circ,\omega),\qquad
    \wt k_{i j}(\hat\rho,\rho_\circ,\omega) = k_{i j}(\hat\rho\omega) + \wt k^{(1)}_{i j}(\hat\rho,\rho_\circ,\omega),
  \]
  where $\gamma_{i j}$, $k_{i j}$ are smooth in the blown-down coordinates $x=\hat\rho\omega\in\R^n$, while $\wt\gamma^{(1)}_{i j}\in\cA_\phg^{\N_0\cup\hat\cE,\cE}$ and $\wt k^{(1)}_{i j}\in\cA_\phg^{(\N_0\cup\hat\cE)-1,\cE}$. Thus, $\wt\gamma^{-1}\in\CI(\R^n_x)+\cA_\phg^{\N_0\cup\hat\cE,\cE}(\wt X)$ in this chart. Recall from Lemma~\ref{LemmaGTot}\eqref{ItGTotb} that $\pa_i\in\wt\cV(\wt X)\hra\hat\rho^{-1}\Vb(\wt X)$; therefore, $\pa_\ell\wt\gamma_{i j} \in \CI(\R^n_x) + \cA_\phg^{(\N_0\cup\hat\cE)-1,\cE}(\wt X)$ and $\pa_\ell\wt k_{i j} \in \CI(\R^n_x) + \cA_\phg^{(\N_0\cup\hat\cE)-2,\cE}(\wt X)$. This gives
  \begin{equation}
  \label{EqCETotMem}
    \Gamma(\wt\gamma)_{i j}^\ell\in\CI(\R^n_x)+\cA_\phg^{(\N_0\cup\hat\cE)-1,\cE}(\wt X),\qquad
    \nabla_\ell\wt k_{i j}\in\CI(\R^n_x)+\cA_\phg^{(\N_0\cup\hat\cE)-2,\cE}(\wt X).
  \end{equation}
  We conclude that $P(\wt\gamma,\wt k;\Lambda)\in\CI(\R^n_x)+\cA_\phg^{(\N_0\cup\hat\cE)-2,\cE}(\wt X)$, as claimed.

  For part~\eqref{ItCETotLin}, we use the expression for $L_{\wt\gamma,\wt k}$ given in~\eqref{EqCELin}, in conjunction with~\eqref{EqCETotMem} and $\pa_{x^i}\in\hat\rho^{-1}\Vq(\wt X)$. The normal operators of $L_{\wt\gamma,\wt k}\wt w$ at $\hat X$ and $X_\circ$ can be computed by evaluating the limit of $L_{\wt\gamma,\wt k}\wt w$ at the interiors $\hat X^\circ$ and $(X_\circ)^\circ$, respectively. Since near $(X_\circ)^\circ$, $(\wt\gamma,\wt k)$ is a polyhomogeneous (with index set $\N_0\cup\cE$) family of symmetric 2-tensors (see also Lemma~\ref{LemmaGTotRel}) with leading term $(\gamma,k)$, we immediately obtain the second equation in~\eqref{EqCETotLinNorm}. For the first equation, we use the scaling relation~\eqref{EqCETotPfScale} and $\hat\rho=\eps\rho_\circ^{-1}$; thus,
  \begin{align*}
    L_{\wt\gamma,\wt k}(\hat\rho^2\wt h,\hat\rho\wt q) &= \eps^{-2}\frac{\dd}{\dd t}\eps^2 P(\wt\gamma+t\eps^2\rho_\circ^{-2}\wt h,\wt k+t\eps\rho_\circ^{-1}\wt q;\Lambda)\Big|_{t=0} \\
      &= \frac{\dd}{\dd\tilde t} \Bigl(\sfs P(\sfs^{-1}\wt\gamma+\tilde t\rho_\circ^{-2}\sfs^{-1}\wt h,\sfs^{-1}(\eps\wt k)+\tilde t\rho_\circ^{-1}\sfs^{-1}\wt q;\eps^2\Lambda)\Bigr)\Big|_{\tilde t=0} \\
      &= \sfs\Bigl(L_{\sfs^{-1}\wt\gamma,\sfs^{-1}(\eps\wt k)}\bigl(\hat w(\sfs^{-1}\wt h,\sfs^{-1}\wt q)\bigr)\Bigr),
  \end{align*}
  where in the second line we introduced $\tilde t=t\eps^2$. Since near $\hat X^\circ$ the pair $(\sfs^{-1}\wt\gamma,\sfs^{-1}(\eps\wt k))$ is a polyhomogeneous (with index set $\N_0\cup\hat\cE$) family of symmetric 2-tensors on $\hat X^\circ$ with leading term $(\hat\gamma,\hat k)$, we obtain the first equation in~\eqref{EqCETotLinNorm}.

  Since $\wt\upbeta^*\CI(X)+\cA_\phg^{\N_0\cup\hat\cE,\cE}(\wt X)\subset\cA_\phg^{\N_0\cup\hat\cE,\N_0\cup\cE}(\wt X)$, we immediately obtain the first membership in~\eqref{EqCETotLinNormDiff}. For the second membership, we note that the $\eps$-independent extension of $\hat\chi L_{\circ,\gamma,k}w_\circ$ to an operator on $\wt X\setminus x^{-1}(0)$ has coefficients in $\hat\chi\wt\upbeta^*\CI(X_\circ)$, and it is equal to $L_{\wt\gamma,\wt k}\wt w$ at $\eps=0$; therefore, its deviation from $\hat\chi L_{\wt\gamma,\wt k}\wt w$ is given by the polyhomogeneous term in~\eqref{EqCETotLin}.
\end{proof}

This result will suffice for the construction of a formal solution of the gluing problem. In order to correct a formal solution to a true solution of the nonlinear constraint equations, we need a solvability theory for $L_{\wt\gamma,\wt k}$ on spaces of tensors with finite regularity:

\begin{prop}[Uniform estimates for the linearized constraints map]
\label{PropCETotSolv}
  Let $(\wt\gamma,\wt k)$ be a $(\hat\cE,\cE)$-smooth total family on $\wt X\setminus\wt K^\circ$. Let $\hat R_0>0$ be such that $\hat K\subset B(0,\hat R_0)$, and denote by $\hat\rho_2\in\CI(\wt X)$ a function which vanishes simply at $\hat r^{-1}(\hat R_0)$ and is positive for $\hat r>\hat R_0$ (e.g.\ one can take $\hat\rho_2=\hat r^{-1}-\hat R_0^{-1}$ near $\hat X$). Denote by $\rho_2\in\CI(\wt X)$ a function which vanishes simply at $[0,1)\times\pa\cU$ and is positive on $[0,1)\times\cU^\circ$; put $\tilde\rho_2=\hat\rho_2\rho_2$. Put $\wt w_2:=\diag(\tilde\rho_2^4,\tilde\rho_2^2)$, and recall $\wt w=\diag(\hat\rho^2,\hat\rho)$. Define q-00-Sobolev spaces on $\wt\Omega:=\{\hat\rho_2\geq 0,\ \rho_2\geq 0\}\subset\wt X\setminus\wt K^\circ$ (with 00-behavior at $\hat r=\hat R_0$ and $\pa\cU$) using a positive q-00-density; write $\wt\Omega_\eps=\wt\Omega\cap\wt X_\eps$. Let $\alpha_2\in\R$ and $\beta>0$, and fix $d\in(-\infty,n-2)$. Then there exist $\eps_0>0$ so that for all $s\in\R$ and $\hat\alpha,\alpha_\circ\in\R$ satisfying $\alpha_\circ-\hat\alpha=d$, and setting
  \begin{equation}
  \label{EqCETotSolvConj}
    \cL_{\wt\gamma,\wt k} := e^{\beta/\tilde\rho_2}\tilde\rho_2^{-\alpha_2}\hat\rho^{-\tilde\alpha}\rho_\circ^{-\alpha_\circ}\wt w_2 L_{\wt\gamma,\wt k}\wt w\rho_\circ^{\alpha_\circ}\hat\rho^{\tilde\alpha}\tilde\rho_2^{\alpha_2}e^{-\beta/\tilde\rho_2},\qquad
    \tilde\alpha:=\hat\alpha+n-2-\frac{n}{2},
  \end{equation}
  the map
  \begin{equation}
  \label{EqCETotSolv}
    \cL_{\wt\gamma,\wt k}\cL_{\wt\gamma,\wt k}^* \colon \hat\rho^{-\frac{n}{2}}H_{\qop,0 0,\eps}^{s+2}(\wt\Omega_\eps)\oplus\hat\rho^{-\frac{n}{2}}H_{\qop,0 0,\eps}^{s+1}(\wt\Omega_\eps;\wt T^*\wt X) \to \hat\rho^{-\frac{n}{2}}H_{\qop,0 0,\eps}^{s-2}(\wt\Omega_\eps)\oplus\hat\rho^{-\frac{n}{2}}H_{\qop,0 0,\eps}^{s-1}(\wt\Omega_\eps;\wt T^*\wt X)
  \end{equation}
  is invertible and has a uniformly bounded inverse for $\eps<\eps_1$ for all $\eps_1<\eps_0$.\footnote{This implies the existence of a right inverse of $L_{\wt\gamma,\wt k}$ which is uniformly bounded as a map
  \begin{align*}
    &\hat\rho^{\hat\alpha-2}\rho_\circ^{\alpha_\circ}\tilde\rho_2^{\alpha_2-4}e^{-\beta/\tilde\rho_2}H_{\qop,0 0,\eps}^{s-2}(\wt\Omega_\eps) \oplus \hat\rho^{\hat\alpha-2}\rho_\circ^{\alpha_\circ}\tilde\rho_2^{\alpha_2-2}e^{-\beta/\tilde\rho_2}H_{\qop,0 0,\eps}^{s-1}(\wt\Omega_\eps;\wt T^*\wt X) \\
    & \to \hat\rho^{\hat\alpha}\rho_\circ^{\alpha_\circ}\tilde\rho_2^{\alpha_2}e^{-\beta/\tilde\rho_2}H_{\qop,0 0,\eps}^s(\wt\Omega_\eps;S^2\wt T^*\wt X) \oplus \hat\rho^{\hat\alpha-1}\rho_\circ^{\alpha_\circ}\tilde\rho_2^{\alpha_2}e^{-\beta/\tilde\rho_2}H_{\qop,0 0,\eps}^s(\wt\Omega_\eps;S^2\wt T^*\wt X).
  \end{align*}} Here, adjoints are defined with respect to the volume density and fiber inner products induced by $\wt\gamma$.
\end{prop}

The definition~\eqref{EqCETotSolvConj} combines the conjugations and re-weightings~\eqref{EqCEAfSolvOp} and \eqref{EqCEPctSolvOp} performed in the analysis of the two normal operators. 

\begin{proof}[Proof of Proposition~\usref{PropCETotSolv}]
  Since $\eps/\rho_\circ$ is a defining function of $\hat X$, we may arrange that $\hat\rho\rho_\circ=\eps$. Since $L_{\wt\gamma,\wt k}$ commutes with multiplication by functions of $\eps=\hat\rho\rho_\circ$, we may therefore replace $(\hat\alpha,\alpha_\circ)$ by $(\hat\alpha-c,\alpha_\circ-c)$ for any $c$; for $c=\alpha_\circ$, we may thus reduce to the case $\alpha_\circ=0$, which we assume henceforth. In particular, $\hat\alpha<n-2$ and thus also $-\hat\alpha\neq -n+2,-n+1$.

  \pfstep{Symbolic estimate.} Note that $\cL_{\wt\gamma,\wt k}$ is a q-00-operator with coefficients in $\cA_\phg^{\N_0\cup\hat\cE,\N_0\cup\cE}(\wt X)$, and the q-00-principal symbol of $\cL_{\wt\gamma,\wt k}^*$ is injective in the Douglis--Nirenberg sense: this follows either by inspection of the form of $\wt w_2 L_{\wt\gamma,\wt k}\wt w$, or for small $\eps>0$ from the injectivity of the b-00-principal symbols of the normal operators of $\cL_{\wt\gamma,\wt k}$. Note that the $X_\circ$-normal operator is
  \[
    N_{X_\circ}(\cL_{\wt\gamma,\wt k})=c^L_\circ\circ\cL_{\circ,\gamma,k}\circ c^R_\circ,
  \]
  where $c_\circ^L$, resp.\ $c_\circ^R$ is a smooth bundle isomorphism on $\ul\R\oplus\upbeta_\circ^*T^*X$, resp.\ $\upbeta_\circ^*S^2 T^*X\oplus\upbeta_\circ^*S^2 T^*X$; here $\cL_{\circ,\gamma,k}$ is given by~\eqref{EqCEPctSolvOp}, and the presence of $c_\circ^{L/R}$ arises from the fact that $\rho_2|_{X_\circ}$ is not necessarily equal to, but certainly a smooth positive multiple of, $\tilde\rho_2|_{X_\circ}$. Similarly, writing $\hat\rho=\eps\rho_\circ^{-1}$ in~\eqref{EqCETotSolvConj}, one finds using~\eqref{EqCETotLinNorm} that the $\hat X$-normal operator satisfies
  \[
    \hat\sfs^{-1}\circ N_{\hat X}(\cL_{\wt\gamma,\wt k})\circ\hat\sfs = \hat c^L\circ\cL_{\hat\gamma,\hat k}\circ\hat c^R,
  \]
  with $\hat c^L$, resp.\ $\hat c^R$ a smooth bundle isomorphism of $\ul\R\oplus\Tsc^*\hat X$, resp.\ $S^2\,\Tsc^*\hat X\oplus S^2\,\Tsc^*\hat X$, and where $\cL_{\hat\gamma,\hat k}$ is defined by~\eqref{EqCEAfSolvOp}, with $\alpha_\circ$ there equal to $-\hat\alpha$ and thus $\tilde\alpha$ there equal to $-\tilde\alpha$ in present notation.

  Elliptic estimates for the Douglis--Nirenberg elliptic operator
  \begin{equation}
  \label{EqCETotSolvLLstar}
    \cL_{\wt\gamma,\wt k}\cL_{\wt\gamma,\wt k}^*\in\cA_\phg^{\N_0\cup\hat\cE,\N_0\cup\cE}\begin{pmatrix} \Diff_{\qop,0 0}^4 & \Diff_{\qop,0 0}^3 \\ \Diff_{\qop,0 0}^3 & \Diff_{\qop,0 0}^2 \end{pmatrix}
  \end{equation}
  give uniform (in $\eps$) estimates
  \begin{equation}
  \label{EqCETotSolvEstPre}
  \begin{split}
    &\|(\tilde f^*,\tilde j^*)\|_{\hat\rho^{-\frac{n}{2}}H_{\qop,0 0,\eps}^{s+2}(\wt\Omega_\eps)\oplus\hat\rho^{-\frac{n}{2}}H_{\qop,0 0,\eps}^{s+1}(\wt\Omega_\eps)} \\
    &\qquad \leq C\Bigl( \|\cL_{\wt\gamma,\wt k}\cL_{\wt\gamma,\wt k}^*(\tilde f^*,\tilde j^*)\|_{\hat\rho^{-\frac{n}{2}}H_{\qop,0 0,\eps}^{s-2}\oplus\hat\rho^{-\frac{n}{2}}H_{\qop,0 0,\eps}^{s-1}} + \|(\tilde f^*,\tilde j^*)\|_{\hat\rho^{-\frac{n}{2}}H_{\qop,0 0,\eps}^{s_0+2}\oplus\hat\rho^{-\frac{n}{2}}H_{\qop,0 0,\eps}^{s_0+1}}\Bigr)
  \end{split}
  \end{equation}
  for any fixed $s_0<s$. (The weights here can be chosen arbitrarily, but with foresight we choose them to match those in~\eqref{EqCEPctSolvOpEst}.) One could weaken the error term by estimating it in spaces with weight $\tilde\rho_2^{-1}(0)$ (due to the 00-ellipticity of $\cL_{\wt\gamma,\wt k}\cL_{\wt\gamma,\wt k}^*$ there), but we shall not need this improvement here.

  \pfstep{Normal operator argument at $\hat X$.} In the second, error, term on the right hand side of~\eqref{EqCETotSolvEstPre}, we insert a partition of unity $1=\hat\chi+(1-\hat\chi)$ and apply the triangle inequality; in the term involving $1-\hat\chi$, the weight at $\hat X$ can then be changed arbitrarily since $\hat\rho>0$ on $\supp(1-\hat\chi)$. In the term involving $\hat\chi$, the first norm equivalence in~\eqref{EqGqHqNormEquivPre} gives the uniform estimate
  \begin{equation}
  \label{EqCETotSolvEsthatXPre}
  \begin{split}
    &\|(\hat\chi\tilde f^*,\hat\chi\tilde j^*)\|_{\hat\rho^{-\frac{n}{2}}H_{\qop,0 0,\eps}^{s_0+2}(\wt\Omega_\eps)\oplus\hat\rho^{-\frac{n}{2}}H_{\qop,0 0,\eps}^{s_0+1}(\wt\Omega_\eps;\wt T^*\wt X)} \\
    &\qquad \leq C\eps^{\frac{n}{2}}\bigl\|\bigl(\sfs^{-1}(\hat\chi\tilde f^*),\sfs^{-1}(\hat\chi\tilde j^*)\bigr)\bigr\|_{\rho_\circ^{\frac{n}{2}}H_{\bop,0 0}^{s_0+2}(\hat X)\oplus\rho_\circ^{\frac{n}{2}}H_{\bop,0 0}^{s_0+1}(\hat X;\Tsc^*\hat X)}.
  \end{split}
  \end{equation}
  Observe now that
  \[
    \hat\sfs\circ N_{\hat X}(\cL_{\wt\gamma,\wt k}\cL_{\wt\gamma,\wt k}^*)\circ\hat\sfs^{-1}=\hat c^L\cL_{\hat\gamma,\hat k}\hat c^R(\hat c^R)^*\cL_{\hat\gamma,\hat k}^*(\hat c^L)^*
  \]
  is invertible as a map~\eqref{EqCEAfSolvInv}. We can thus estimate the right hand side of~\eqref{EqCETotSolvEsthatXPre} by an $\eps$-independent constant times
  \[
    \eps^{\frac{n}{2}}\|\hat\sfs^{-1}N_{\hat X}(\cL_{\wt\gamma,\wt k}\cL_{\wt\gamma,\wt k}^*)(\hat\chi\tilde f^*,\hat\chi\tilde j^*)\|_{\rho_\circ^{\frac{n}{2}}H_{\bop,0 0}^{s_0-2}(\hat X;\Tsc^*\hat X)\oplus\rho_\circ^{\frac{n}{2}}H_{\bop,0 0}^{s_0-1}(\hat X;\Tsc^*\hat X)}.
  \]
  Using~\eqref{EqGqHqNormEquiv} to pass back to q-00-Sobolev spaces on $\wt X$, we may replace the normal operator here by the operator $\cL_{\wt\gamma,\wt k}\cL_{\wt\gamma,\wt k}^*$ itself, and as a consequence of Proposition~\ref{PropCETot}\eqref{ItCETotLin} pick up an error term whose weight at $\hat X$ is reduced by an amount $0<\eta<\min(1,\min\Re\hat\cE)$. This gives
  \begin{align*}
    &\|(\hat\chi\tilde f^*,\hat\chi\tilde j^*)\|_{\hat\rho^{-\frac{n}{2}}H_{\qop,0 0,\eps}^{s_0+2}(\wt\Omega_\eps)\oplus\hat\rho^{-\frac{n}{2}}H_{\qop,0 0,\eps}^{s_0+1}(\wt\Omega_\eps;\wt T^*\wt X)} \\
    &\qquad \leq C\Bigl( \|\cL_{\wt\gamma,\wt k}\cL_{\wt\gamma,\wt k}^*(\hat\chi\tilde f^*,\hat\chi\tilde j^*)\|_{\hat\rho^{-\frac{n}{2}}H_{\qop,0 0,\eps}^{s_0-2}(\wt\Omega_\eps)\oplus\hat\rho^{-\frac{n}{2}}H_{\qop,0 0,\eps}^{s_0-1}(\wt\Omega_\eps;\wt T^*\wt X)} \\
    &\qquad \hspace{4em} + \|(\hat\chi\tilde f^*,\hat\chi\tilde j^*)\|_{\hat\rho^{-\frac{n}{2}-\eta}H_{\qop,0 0,\eps}^{s_0+2}(\wt\Omega_\eps)\oplus\hat\rho^{-\frac{n}{2}-\eta}H_{\qop,0 0,\eps}^{s_0+1}(\wt\Omega_\eps;\wt T^*\wt X)}\Bigr).
  \end{align*}
  We now drop the cutoff in the second term on the right; note that multiplication by $\hat\chi$ is uniformly bounded on every q-00-Sobolev space. Since the coefficients $[\cL_{\wt\gamma,\wt k}\cL_{\wt\gamma,\wt k}^*,\hat\chi]$ are supported away from $\hat X$, we can omit the cutoff $\hat\chi$ in the first term as well since the contribution from the commutator is controlled by a constant times the second term (with $\hat\chi$ dropped).

  As a consequence, we can now strengthen the estimate~\eqref{EqCETotSolvEstPre} by replacing the norm of the error term by a weaker norm at $\hat X$:
  \begin{equation}
  \label{EqCETotSolvEstPre2}
  \begin{split}
    &\|(\tilde f^*,\tilde j^*)\|_{\hat\rho^{-\frac{n}{2}}H_{\qop,0 0,\eps}^{s+2}(\wt\Omega_\eps)\oplus\hat\rho^{-\frac{n}{2}}H_{\qop,0 0,\eps}^{s+1}(\wt\Omega_\eps)} \\
    &\qquad \leq C\Bigl( \|\cL_{\wt\gamma,\wt k}\cL_{\wt\gamma,\wt k}^*(\tilde f^*,\tilde j^*)\|_{\hat\rho^{-\frac{n}{2}}H_{\qop,0 0,\eps}^{s-2}\oplus\hat\rho^{-\frac{n}{2}}H_{\qop,0 0,\eps}^{s-1}} \\
    &\qquad\qquad\quad + \|(\tilde f^*,\tilde j^*)\|_{\hat\rho^{-\frac{n}{2}-\eta}H_{\qop,0 0,\eps}^{s_0+2}(\wt\Omega_\eps)\oplus\hat\rho^{-\frac{n}{2}-\eta}H_{\qop,0 0,\eps}^{s_0+1}(\wt\Omega_\eps;\wt T^*\wt X)}\Bigr).
  \end{split}
  \end{equation}

  \pfstep{Normal operator argument at $X_\circ$.} We improve the estimate~\eqref{EqCETotSolvEstPre2} further by using the $X_\circ$-normal operator. To wit, inserting a partition of unity $1=\chi_\circ+(1-\chi_\circ)$ in the second term on the right and applying the triangle inequality, the weight at $X_\circ$ of the term with $1-\chi_\circ$ can be changed arbitrarily; for the term with $\chi_\circ$ on the other hand, we use the second norm equivalence in~\eqref{EqGqHqNormEquivPre} to get an upper bound by an $\eps$-independent constant times
  \begin{equation}
  \label{EqCETotSolvEstXcircPre}
    \|(\chi_\circ\tilde f^*,\chi_\circ\tilde j^*)\|_{\hat\rho^{-\frac{n}{2}-\eta}H_{\bop,0 0}^{s_0+2}(X_\circ)\oplus\hat\rho^{-\frac{n}{2}-\eta}H_{\bop,0 0}^{s_0+1}(X_\circ;\upbeta_\circ^*T^*X)}.
  \end{equation}
  We now use that, analogously to~\eqref{EqCEPctSolv1Inv}, the normal operator
  \[
    N_{X_\circ}(\cL_{\wt\gamma,\wt k}\cL_{\wt\gamma,\wt k}^*)=c_\circ^L\cL_{\circ,\gamma,k}c_\circ^R(c_\circ^R)^*\cL_{\circ,\gamma,k}^*(c_\circ^L)^*
  \]
  is invertible as a map
  \[
    \hat\rho^{-\frac{n}{2}-\eta}H_{\bop,0 0}^{s_0+2}\oplus\hat\rho^{-\frac{n}{2}-\eta}H_{\bop,0 0}^{s_0+1} \to \hat\rho^{-\frac{n}{2}-\eta}H_{\bop,0 0}^{s_0-2}\oplus\hat\rho^{-\frac{n}{2}-\eta}H_{\bop,0 0}^{s_0-1}
  \]
  if $\eta>0$ is sufficiently small. (Note that~\eqref{EqCEPctSolv1Inv} gives this for $\eta=0$; furthermore, the requirements on the $\hat\rho$-weight $-\frac{n}{2}$ used in the proof of~\eqref{EqCEPctSolv1Inv} are \emph{open} in the weight, and therefore are satisfied for $-\frac{n}{2}-\eta$ as well when $\eta>0$ is small enough.) We can thus estimate~\eqref{EqCETotSolvEstXcircPre} by an $\eps$-independent constant times
  \begin{align*}
    \|N_{X_\circ}(\cL_{\wt\gamma,\wt k}\cL_{\wt\gamma,\wt k}^*)(\chi_\circ\tilde f^*,\chi_\circ\tilde j^*)\|_{\hat\rho^{-\frac{n}{2}-\eta}H_{\bop,0 0}^{s_0-2}(X_\circ;\upbeta_\circ^*T^*X)\oplus\hat\rho^{-\frac{n}{2}-\eta}H_{\bop,0 0}^{s_0-1}(X_\circ;\upbeta_\circ^*T^*X)}.
  \end{align*}
  Passing back to q-00-Sobolev spaces on $\wt X$ via~\eqref{EqGqHqNormEquivPre} and replacing the normal operator by $\cL_{\wt\gamma,\wt k}\cL_{\wt\gamma,\wt k}^*$ produces an error with reduced weight at $X_\circ$.

  The grand total is the uniform estimate
  \begin{align*}
    &\|(\tilde f^*,\tilde j^*)\|_{\hat\rho^{-\frac{n}{2}}H_{\qop,0 0,\eps}^{s+2}(\wt\Omega_\eps)\oplus\hat\rho^{-\frac{n}{2}}H_{\qop,0 0,\eps}^{s+1}(\wt\Omega_\eps)} \\
    &\qquad \leq C\Bigl( \|\cL_{\wt\gamma,\wt k}\cL_{\wt\gamma,\wt k}^*(\tilde f^*,\tilde j^*)\|_{\hat\rho^{-\frac{n}{2}}H_{\qop,0 0,\eps}^{s-2}\oplus\hat\rho^{-\frac{n}{2}}H_{\qop,0 0,\eps}^{s-1}} \\
    &\qquad \hspace{4em} + \|(\tilde f^*,\tilde j^*)\|_{\hat\rho^{-\frac{n}{2}-\eta}\rho_\circ^{-\eta}H_{\qop,0 0,\eps}^{s_0+2}(\wt\Omega_\eps)\oplus\hat\rho^{-\frac{n}{2}-\eta}\rho_\circ^{-\eta}H_{\qop,0 0,\eps}^{s_0+1}(\wt\Omega_\eps;\wt T^*\wt X)}\Bigr)
  \end{align*}
  for some $\eta>0$. But the second term on the right is bounded by $C\eps^\eta$ times the left hand side and thus can be absorbed for sufficiently small $\eps>0$. This gives a uniform coercive estimate for $\cL_{\wt\gamma,\wt k}\cL_{\wt\gamma,\wt k}^*$. The same estimate applies on the dual function spaces as well, and we conclude the uniformly invertibility of~\eqref{EqCETotSolvConj} for $\eps<\eps_0$ when $\eps_0>0$ is sufficiently small (depending only on the choice of $s,s_0$). Fixing $\eps_0$ for the choice $s=0$, $s_0=-1$, say, the uniform invertibility follows for any other choice of $s\in\R$ by elliptic regularity (when $s>0$) and duality (when $s<0$).
\end{proof}

We shall need to understand the regularity properties in $\eps$ of the inverse of~\eqref{EqCETotSolv}. Here, regularity is meant in the following sense:

\begin{definition}[Regularity of families of elements of q-00-Sobolev spaces]
\label{DefCETotReg}
  Define the domain $\wt\Omega$ as in Proposition~\usref{PropCETotSolv}; let $\hat\Omega=\{\hat\rho_2\geq 0\}\subset\hat X$ and $\Omega_\circ=\{\rho_2\geq 0\}\subset X_\circ$ as in Propositions~\usref{PropCEAfSolv} and~\usref{PropCEPctSolv}. Fix cutoffs $\hat\chi,\chi_\circ\in\CI(\wt X)$ as in~\eqref{EqGTotCutoffs} with the additional property that $\hat\rho_2>0$ on $\supp\chi_\circ$, and $\rho_2>0$ on $\supp\hat\chi$. Recall the maps $\hat\phi$, $\phi_\circ$ from~\eqref{EqGqPhis}. Let $\eps_0>0$ and $s\in\R$.
  \begin{enumerate}
  \item{\rm (Continuity.)} The space
    \[
      \cC^0([0,\eps_0),H_{\qop,0 0,\eps}^s(\wt\Omega_\eps))
    \]
    consists of all families $(u_\eps)_{\eps\in(0,\eps_0)}$ such that $\|u_\eps\|_{H_{\qop,0 0,\eps}^s(\wt\Omega_\eps)}$ is uniformly bounded for $\eps\in(0,\eps_1)$ for all $\eps_1<\eps_0$, and so that $\hat\phi^*(\hat\chi u_\eps)\in H_{\bop,0 0}^s(\hat\Omega)$ and $\phi_\circ^*(\chi_\circ u_\eps)\in H_{\bop,0 0}^s(\Omega_\circ)$ depend continuously on $\eps$.
  \item\label{ItCETotRegCon}{\rm (Conormal regularity.)} For $m\in\N_0$, define $\cA_m([0,\eps_0),H_{\qop,0 0,\eps}^s(\wt\Omega_\eps))$ to consist of all $u\in\cC^0([0,\eps_0),H_{\qop,0 0,\eps}^s(\wt\Omega_\eps))$ so that $(\eps\pa_\eps)^j \hat\phi^*(\hat\chi u_\eps)\in H_{\bop,0 0}^{s-j}(\hat\Omega)$ and $(\eps\pa_\eps)^j \phi_\circ^*(\chi_\circ u_\eps)\in H_{\bop,0 0}^{s-j}(\Omega_\circ)$ are continuous and uniformly bounded for all $j\leq m$ and $\eps\in(0,\eps_1)$, $\eps_1<\eps_0$. (In the first expression, $\pa_\eps$ is the coordinate derivative in $(\eps,\hat x)$-coordinates; in the second expression, it is defined as the derivative along the first factor of $[0,1)_\eps\times X_\circ$.)
  \end{enumerate}
  Spaces of families of elements of weighted q-00-Sobolev spaces are defined analogously (cf.\ the relationship~\eqref{EqGqHqNormEquiv}).
\end{definition}

Thus, $\cA_0=\cC^0$. We also note that the derivatives in part~\eqref{ItCETotRegCon} are well-defined as distributional derivatives due to the a priori assumption that $u\in\cC^0$. The regularity condition in part~\eqref{ItCETotRegCon} can be rewritten as follows. Let $\cR\in\Vb(\wt X)$ denote a vector field whose push-forward to $[0,1)_\eps$ is given by $\eps\pa_\eps$ and which is tangent to $\pa\wt\Omega$. One possible choice for $\cR$, in $(\eps,x)$-coordinates, is $\eps\pa_\eps+\hat\chi x\pa_x$; indeed, this equals $\eps\pa_\eps-(1-\hat\chi)\hat x\pa_{\hat x}$ in $(\eps,\hat x)$-coordinates. Then $u\in\cC^0([0,\eps_0),H_{\qop,0 0,\eps}^s(\wt\Omega_\eps))$ lies in $\cA_m$ iff $\cR^j u\in\cC^0([0,\eps_0),H_{\qop,0 0,\eps}^{s-j}(\wt\Omega_\eps))$ for $j=1,\ldots,m$.

\begin{lemma}[Regularity of the inverse]
\label{LemmaCETotReg}
  Let $m\in\N_0$. In the notation of Proposition~\usref{PropCETotSolv}, the inverse of the map~\eqref{EqCETotSolv} is bounded as a map
  \begin{align*}
    (\cL_{\wt\gamma,\wt k}\cL_{\wt\gamma,\wt k}^*)^{-1} &\colon \cA_m\bigl([0,\eps_0),\hat\rho^{-\frac{n}{2}}H_{\qop,0 0,\eps}^{s-2}(\wt\Omega_\eps)\bigr) \oplus \cA_m\bigl([0,\eps_0),\hat\rho^{-\frac{n}{2}}H_{\qop,0 0,\eps}^{s-1}(\wt\Omega_\eps;\wt T^*\wt X)\bigr) \\
      &\qquad \to \cA_m\bigl([0,\eps_0),\hat\rho^{-\frac{n}{2}}H_{\qop,0 0,\eps}^{s+2}(\wt\Omega_\eps)\bigr) \oplus \cA_m\bigl([0,\eps_0),\hat\rho^{-\frac{n}{2}}H_{\qop,0 0,\eps}^{s+1}(\wt\Omega_\eps;\wt T^*\wt X)\bigr).
  \end{align*}
\end{lemma}
\begin{proof}
  Suppose that $\cL_{\wt\gamma,\wt k}\cL_{\wt\gamma,\wt k}^*u=f$ with $f\in\cA_0=\cC^0$. To prove the continuity of $u=u(\eps)$ at some value $\eps=\eps_\infty\in(0,\eps_0)$, note that the weak compactness of the unit ball in the $H_{\bop,0 0}$-spaces on $\hat\Omega$ and $\Omega_\circ$ implies that for any given sequence $\eps_i\to\eps_\infty$, we can pass to a subsequence so that $\eps^{\frac{n}{2}}\hat\phi^*(\hat\chi u(\eps_i))$ and $\phi_\circ^*(\chi_\circ u(\eps_i))$ converge weakly to some
  \[
    \hat u\in\rho_\circ^{\frac{n}{2}}H_{\bop,0 0}^{s+2}(\hat\Omega)\oplus\rho_\circ^{\frac{n}{2}}H_{\bop,0 0}^{s+1}(\hat\Omega;\hat\phi^*\wt T^*\wt X),\qquad
    u_\circ\in\hat\rho^{-\frac{n}{2}}H_{\bop,0 0}^{s+2}(\Omega_\circ)\oplus\hat\rho^{-\frac{n}{2}}H_{\bop,0 0}^{s+1}(\Omega_\circ;\phi_\circ^*\wt T^*\wt X).
  \]
  Moreover, $\eps^{-\frac{n}{2}}\hat\phi_*(\hat u)=(\phi_\circ)_*(u_\circ)$ on $\hat\chi^{-1}(1)\cap\chi_\circ^{-1}(1)$. Thus, there exists $u_\infty\in(H_{0 0}^{s+2}\oplus H_{0 0}^{s+1})(\wt\Omega\cap\wt X_{\eps_\infty})$ with $\eps^{\frac{n}{2}}\hat\phi^*(\hat\chi u_\infty)=\hat u$ and $\phi_\circ^*(\chi_\circ u_\infty)=u_\circ$. Since the coefficients of $\cL_{\wt\gamma,\wt k}\cL_{\wt\gamma,\wt k}^*$ on $\wt\Omega^\circ\cap\{\eps>0\}$ depend smoothly on $\eps$, we find, upon taking the limit $i\to\infty$, that $\cL_{\wt\gamma,\wt k}\cL_{\wt\gamma,\wt k}^*u_\infty=f(\eps_\infty)$. The invertibility of $\cL_{\wt\gamma,\wt k}\cL_{\wt\gamma,\wt k}^*$ implies that $u_\infty=u(\eps_\infty)$, as desired.

  When $f\in\cA_k$ and $m=1$, we note that
  \[
    \cL_{\wt\gamma,\wt k}\cL_{\wt\gamma,\wt k}^*(\cR u) = \cR f + [\cL_{\wt\gamma,\wt k}\cL_{\wt\gamma,\wt k}^*,\cR]u.
  \]
  The first term lies in $\cA_0$; furthermore, the operator acting on $u$ in the second term is a q-00-operator of class~\eqref{EqCETotSolvLLstar}, and thus the second term lies in $\cA_0$ as well. Inverting $\cL_{\wt\gamma,\wt k}\cL_{\wt\gamma,\wt k}^*$ therefore implies $\cR u\in\cA_0$. The case of $m\geq 2$ follows by induction.
\end{proof}

\begin{rmk}[Relaxing the regularity of $\wt\gamma,\wt k$ in $\eps$]
\label{RmkCETotSolvCont}
  Later, we shall use slight extension of Proposition~\ref{PropCETot}, in which we work with $\wt\gamma+\wt h,\wt k+\wt q$ where $\wt h,\wt q$ are continuous families (in $\eps$) of fiberwise smooth tensors on $\wt X$ which decay rapidly as $\eps\searrow 0$. Lemma~\ref{LemmaCETotReg} similarly remains valid for $m=0$. Note indeed that in the proof of Proposition~\ref{PropCETot}, the regularity of $\wt\gamma,\wt h$ in $\eps$ is not used except when arguing that $\cL_{\wt\gamma,\wt k}\cL_{\wt\gamma,\wt k}^*$ differs from its two normal operators by error terms which decay, as q-00-differential operators, at the respective boundary hypersurface of $\wt X$. This decay is still guaranteed due to the rapid decay of $\wt h,\wt q$ as $\eps\searrow 0$.
\end{rmk}

Since in the gluing problem the boundary data~\eqref{EqCETotRestr} will be given, and satisfy the constraint equations, our task (in view of Proposition~\ref{PropCETot}\eqref{ItCETotPhg}) is to find subleading corrections so as produce a solution of the constraint equations. At the subleading level, the nonlinear constraints map is well approximated by its linearization:

\begin{lemma}[Accuracy of the linearization]
\label{LemmaCETotAcc}
  Let $(\wt\gamma,\wt k)$ be a $(\hat\cE,\cE)$-smooth total family. Put
  \begin{equation}
  \label{EqCETotAccQ}
    Q(\wt\gamma,\wt k;\wt h,\wt q) := P(\wt\gamma+\wt h,\wt k+\wt q;\Lambda) - P(\wt\gamma,\wt k;\Lambda) - L_{\wt\gamma,\wt k}(\wt h,\wt q)
  \end{equation}
  for all $\wt h,\wt q$ for which the right hand side is defined for $\eps<\eps_0$ with some $\eps_0>0$ (depending on $\wt\gamma,\wt k,\wt h,\wt q$).
  \begin{enumerate}
  \item\label{ItCETotAccPhg}{\rm (Polyhomogeneous version.)} Suppose $\hat\cE'\subset\hat\cE$ and $\cE'\subset\cE$ are index sets. Put $\hat\cF'=2\hat\cE'+(\N_0\cup\hat\cE)$ and $\cF'=2\cE'+(\N_0\cup\cE)$. Let $(\tilde h,\tilde q)\in\cA_\phg^{\hat\cE',\cE'}(\wt X\setminus\wt K^\circ;S^2\wt T^*\wt X)\oplus\cA_\phg^{\hat\cE'-1,\cE'}(\wt X\setminus\wt K^\circ;S^2\wt T^*\wt X)$. Then\footnote{For all sufficiently small $\eps>0$, the sum $\wt\gamma+\wt h$ is a positive definite section of $S^2\wt T^*\wt X$, and therefore the left hand side is well-defined.}
    \begin{equation}
    \label{EqCETotAccPhg}
      Q(\wt\gamma,\wt k;\wt h,\wt q) \in \cA_\phg^{\hat\cF'-2,\cF'}(\wt X;\ul\R\oplus\wt T^*\wt X).
    \end{equation}
  \item\label{ItCETotAccq}{\rm (q-00-Sobolev version.)} In the notation of Proposition~\usref{PropCETotSolv}, suppose that $\alpha_0,\hat\alpha>0$ and $s>\frac{n}{2}+2$. Let $\alpha'_2\in\R$. Then we have a uniform estimate
    \begin{equation}
    \label{EqCETotAccq}
    \begin{split}
      &\|Q(\wt\gamma,\wt k;\wt h_1,\wt q_1)-Q(\wt\gamma,\wt k;\wt h_2,\wt q_2)\|_{\hat\rho^{2\hat\alpha-2}\rho_\circ^{2\alpha_\circ}\tilde\rho_2^{\alpha'_2}e^{-\beta/\tilde\rho_2}H_{\qop,0 0,\eps}^{s-2}\oplus\hat\rho^{2\hat\alpha-2}\rho_\circ^{2\alpha_\circ}\tilde\rho_2^{\alpha'_2}e^{-\beta/\tilde\rho_2}H_{\qop,0 0,\eps}^{s-1}} \\
      &\qquad \leq C\|(\wt h_1,\wt q_1)-(\wt h_2,\wt q_2)\|_{\hat\rho^{\hat\alpha}\rho_\circ^{\alpha_\circ}\tilde\rho_2^{\alpha_2}e^{-\beta/\tilde\rho_2}H_{\qop,0 0,\eps}^s\oplus\hat\rho^{\hat\alpha-1}\rho_\circ^{\alpha_\circ}\tilde\rho_2^{\alpha_2}e^{-\beta/\tilde\rho_2}H_{\qop,0 0,\eps}^s} \\
      &\qquad\hspace{3.12em} \times \max_{j=1,2} \|(\wt h_j,\wt q_j)\|_{\hat\rho^{\hat\alpha}\rho_\circ^{\alpha_\circ}\tilde\rho_2^{\alpha_2}e^{-\beta/\tilde\rho_2}H_{\qop,0 0,\eps}^s\oplus\hat\rho^{\hat\alpha-1}\rho_\circ^{\alpha_\circ}\tilde\rho_2^{\alpha_2}e^{-\beta/\tilde\rho_2}H_{\qop,0 0,\eps}^s}
    \end{split}
    \end{equation}
    for all $(\tilde h_j,\tilde q_j)$, $j=1,2$, whose norm in the function space $\hat\rho^{\hat\alpha}\rho_\circ^{\alpha_\circ}\tilde\rho_2^{\alpha_2}e^{-\beta/\tilde\rho_2}H_{\qop,0 0,\eps}^{s_0}\oplus\hat\rho^{\hat\alpha-1}\rho_\circ^{\alpha_\circ}\tilde\rho_2^{\alpha_2}e^{-\beta/\tilde\rho_2}H_{\qop,0 0,\eps}^{s_0}$ is sufficiently small depending on $s_0>\frac{n}{2}+2$, $\alpha_2\in\R$, $\beta>0$, and $\alpha_\circ-\hat\alpha$ (as well as on $\wt\gamma,\wt k$). In particular, using the same norms as in~\eqref{EqCETotAccq},
    \[
      \|Q(\wt\gamma,\wt k;\wt h,\wt q)\|\leq C\|(\wt h,\wt q)\|^2.
    \]
  \end{enumerate}
\end{lemma}

In part~\eqref{ItCETotAccPhg}, note that $L_{\wt\gamma,\wt k}=L_{\wt\gamma,\wt k}\wt w\circ\wt w^{-1}$ maps $(\wt h,\wt q)$ into
\begin{equation}
\label{EqCETotAccLot}
  L_{\wt\gamma,\wt k}\wt w(\cA_\phg^{\hat\cE'-2,\cE'})\subset\cA_\phg^{\hat\cE'+(\N_0\cup\hat\cE)-2,\cE'+(\N_0\cup\cE)}(\wt X\setminus\wt K^\circ;\ul\R\oplus\wt T^*\wt X).
\end{equation}
Thus,~\eqref{EqCETotAccPhg} shows that $P(\wt\gamma+\wt h,\wt k+\wt q;\Lambda)-P(\wt\gamma,\wt k;\Lambda)$ is given by $L_{\wt\gamma,\wt k}(\wt h,\wt q)$ modulo quadratically small error terms; similarly for part~\eqref{ItCETotAccq}.

\begin{proof}[Proof of Lemma~\usref{LemmaCETotAcc}]
  Consider the schematic form of $P(\wt\gamma,\wt k;\Lambda)$ in~\eqref{EqCEOp}. In local coordinates on $X$, lifted to (singular) coordinates on $\wt X$, the Christoffel symbols of $\wt\gamma$ are of the form $\Gamma(\wt\gamma)\sim\wt\gamma^{-1}D\wt\gamma$ where $D$ is a coordinate derivative, and where $\wt\gamma,\wt\gamma^{-1}$ denote components of the (inverse) metric in these coordinates. The scalar curvature is thus of the form
  \[
    R_{\wt\gamma}\sim D(\wt\gamma^{-1}D\wt\gamma)+(\wt\gamma^{-1}D\wt\gamma)^2\sim\wt\gamma^{-1}D^2\wt\gamma+\wt\gamma^{-2}(D\wt\gamma)^2.
  \]
  Now, in the neighborhood of any point on $X$ and for sufficiently small $|\wt h|$ (maximum norm of the components of $\wt h$), we have $(\wt\gamma+\wt h)^{-1}=\sum_{j=0}^\infty(-1)^j\wt\gamma^{-1}(\wt h\wt\gamma^{-1})^j$; therefore, for twice continuously differentiable $\wt h$ with small $|\wt h|$, we can expand $R_{\wt\gamma+\wt h}\sim(\wt\gamma+\wt h)^{-1}D^2(\wt\gamma+\wt h)+(\wt\gamma+\wt h)^{-2}(D\wt\gamma+D\wt h)^2$ into a convergent (in $\cC^0$) power series around $\wt h=0$, giving
  \begin{align*}
    &R_{\wt\gamma+\wt h} = R_{\wt\gamma} + D_{\wt\gamma}R(\wt h) + Q_1(\wt\gamma,\wt h), \\
    &\qquad Q_1(\wt\gamma;\wt h) \sim \sum_{j\geq 0} c_j(\wt\gamma)\tilde h^j \bigl( (D^2\wt\gamma)\tilde h^2 + \tilde h(D^2\tilde h) + (D\wt\gamma)^2\wt h^2 + (D\wt\gamma)\wt h(D\wt h) + (D\wt h)^2\bigr).
  \end{align*}
  Here, we have $c_j(\wt\gamma)\in\cA_\phg^{\N_0\cup\hat\cE,\N_0\cup\cE}$ since $\wt\gamma\in\cA_\phg^{\N_0\cup\hat\cE,\N_0\cup\cE}$. In a similar vein, since $(\tr_{\wt\gamma}\wt k)^2\sim(\wt\gamma^{-1}\wt k)^2$, the expression $(\tr_{\wt\gamma+\wt h}(\wt k+\wt q))^2$ is equal to the sum of the nonlinear leading term $(\tr_{\wt\gamma}\wt k)^2$, its linearization in $(\wt\gamma,\wt k)$ evaluated at $(\wt h,\wt q)$, and a nonlinear error term
  \[
    Q_2(\wt\gamma,\wt k;\wt h,\wt q) \sim \sum_{j\geq 0} c_j(\wt\gamma)\wt h^j\bigl(\wt k^2\wt h^2+\wt k\wt h\wt q+\wt q^2\bigr);
  \]
  the term $|\wt k|_{\wt\gamma}^2\sim(\wt\gamma^{-1}\wt k)^2$ has a completely analogous description.

  For the second component of $P(\wt\gamma,\wt k;\Lambda)$, the terms $\delta_{\wt\gamma}\wt k$ and $\dd\tr_{\wt\gamma}\wt k$ have the schematic forms $\wt\gamma^{-1}(D\wt k+(\wt\gamma^{-1}D\wt\gamma)\wt k)$ and $D(\wt\gamma^{-1}\wt k)$, respectively, which are both schematically equal to $\wt\gamma^{-1}D\wt k+\wt\gamma^{-2}(D\wt\gamma)\wt k$. Thus, the quadratic and higher order (in $(\wt h,\wt q)$) terms of $\delta_{\wt\gamma+\wt h}(\wt k+\wt q)$ and $\dd\tr_{\wt\gamma+\wt h}(\wt k+\wt q)\sim(\wt\gamma+\wt h)^{-1}D(\wt k+\wt q)+(\wt\gamma+\wt h)^{-2}(D(\wt\gamma+\wt h))(\wt k+\wt q)$ are of the schematic form
  \[
    Q_3(\wt\gamma,\wt k;\wt h,\wt q) \sim \sum_{j\geq 0} c_j(\wt\gamma)\wt h^j \bigl( (D\wt k)\wt h^2+\wt h(D\wt q) + (D\wt\gamma)\wt k\wt h^2 + (D\wt\gamma)\wt q\wt h+\wt k(D\wt h)\wt h+(D\wt h)\wt q\bigr).
  \]
  Using
  \[
    P(\wt\gamma+\wt h,\wt k+\wt q;\Lambda)-P(\wt\gamma,\wt k;\Lambda)-L_{\wt\gamma,\wt k}(\wt h,\wt q) \sim Q_1(\wt\gamma;\wt h)+Q_2(\wt\gamma,\wt k;\wt h,\wt q),Q_3(\wt\gamma,\wt k;\wt h,\wt q),
  \]
  and the fact that $D\in\hat\rho^{-1}\Vq(\wt X)$, we can now prove the Lemma.

  For part~\eqref{ItCETotAccPhg}, consider the first term of $Q_1$, which is $\wt h^j(D^2\wt\gamma)\wt h^2\in\cA_\phg^{(\N_0\cup\hat\cE)-2,\N_0\cup\cE}\cdot\cA_\phg^{(2+j)\hat\cE',(2+j)\cE'}$; since $\hat\cE'\subset\hat\cE$, we have $(\N_0\cup\hat\cE)+(2+j)\hat\cE'\subset(\N_0\cup\hat\cE)+2\hat\cE'$, similarly for the un-hatted index sets. Thus, $\wt h^j(D^2\wt\gamma)\wt h^2\in\cA_\phg^{\hat\cF'-2,\cF'}$ indeed. Structurally, note that this term, and indeed all terms comprising $Q_1$ and $Q_2$, contain a total of exactly $2$ factors involving $D,\wt k,\wt q$. Each of these factors reduces the index set at $\hat X$ by $1$ compared to $\wt\gamma,\wt h$; therefore, $Q_1,Q_2\in\cA_\phg^{\hat\cF'-2,\cF'}$. The terms comprising $Q_3$ have the same structure (for instance, $\wt h^{2+j}(D\wt k)$ involves $D$ and $\wt k$, while $\wt h^j(D\wt h)\wt q$ involves $D$ and $\wt q$), and thus also $Q_3\in\cA_\phg^{\hat\cF'-2,\cF'}$.

  Turning to part~\eqref{ItCETotAccq} and considering the case $(\wt h_2,\wt q_2)=0$ first, we use the multiplicative properties of q-00-Sobolev spaces (defined relative to a positive q-00-density), which for unweighted spaces are the same as those of standard Sobolev spaces on $\R^n$ (as follows from the bounded geometry perspective), while for weighted spaces the weights are additive. In particular, $H_{\qop,0 0,\eps}^s$ is an algebra under pointwise multiplication for $s>\frac{n}{2}$, and there exists an $\eps$-independent constant $C$ so that $\|u v\|_{H_{\qop,0 0,\eps}^s}\leq C\|u\|_{H_{\qop,0 0,\eps}^s}\|v\|_{H_{\qop,0 0,\eps}^s}$. Thus, for example, using the membership $D^2\wt\gamma\in\cA^{-2,0}$, we have
  \[
    \|\tilde h^j(D^2\wt\gamma)\tilde h^2\|_{\hat\rho^{(2+j)\hat\alpha-2}\rho_\circ^{(2+j)\alpha_\circ}\tilde\rho_2^{(2+j)\alpha_2}e^{-(2+j)\beta/\tilde\rho_2}H_{\qop,0 0,\eps}^s} \leq C^{2+j}\|\tilde h\|_{\hat\rho^{\hat\alpha}\rho_\circ^{\alpha_\circ}\tilde\rho_2^{\alpha_2}e^{-\beta/\tilde\rho_2}H_{\qop,0 0,\eps}^s}^{2+j}.
  \]
  Since for $\beta>0$ multiplication by $e^{-\beta/\tilde\rho_2}\tilde\rho_2^\alpha$ is uniformly bounded on any fixed weighted q-00-Sobolev space for any $\alpha$, we can replace the weight $\tilde\rho_2^{(2+j)\alpha_2}e^{-(2+j)\beta/\tilde\rho_2}$ on the left hand sides by $C_2^j\tilde\rho_2^{\alpha'_2}e^{-\beta/\tilde\rho_2}$ with $C_2>0$. Turning to the term $c_j(\wt\gamma)\wt h^{j+1}(D^2\wt h)$ in $Q_1(\wt\gamma,\wt h)$, observe that $D$, relative to a q-00-derivative, loses two powers of $\tilde\rho_2$ in addition to one power of $\hat\rho$; we thus have
  \[
    \|\tilde h^{1+j}(D^2\tilde h)\|_{\hat\rho^{(2+j)\hat\alpha-2}\rho_\circ^{(2+j)\alpha_\circ}\tilde\rho_2^{(2+j)\alpha_2-4}e^{-(2+j)\beta/\tilde\rho_2}H_{\qop,0 0,\eps}^s} \leq C^{2+j}\|\tilde h\|_{\hat\rho^{\hat\alpha}\rho_\circ^{\alpha_\circ}\tilde\rho_2^{\alpha_2}e^{-\beta/\tilde\rho_2}H_{\qop,0 0,\eps}^s}^{2+j}.
  \]
  Arguing as before, the $\tilde\rho_2$-weight on the left can be replaced by $\tilde\rho_2^{\alpha'_2}e^{-\beta/\tilde\rho_2}$ upon suitably enlarging the constant $C$. The analysis of the remaining terms comprising $Q_1,Q_2,Q_3$ is analogous. In order to estimate the difference $Q(\wt\gamma,\wt k;\wt h_1,\wt q_1)-Q(\wt\gamma,\wt k;\wt h_2,\wt q_2)$, one uses identities such as $\wt h_1^j-\wt h_2^j=(\wt h_1-\wt h_2)\Sigma$ where $\Sigma:=\sum_{k=0}^{j-1}\wt h_1^k\wt h_2^{j-1-k}$; and the norm of $\Sigma$ is bounded by $C j$ times the $(j-1)$-st power of the maximum of the norms of $\tilde h_1$ and $\tilde h_2$.
\end{proof}

For the construction of a formal solution to the gluing problem, we only need:

\begin{cor}[Accuracy of the normal operators]
\label{CorCETotAccNorm}
  Let $(\wt\gamma,\wt k)$ be a $(\hat\cE,\cE)$-smooth total family with boundary data $(\hat\gamma,\hat k)$ and $(\gamma,k)$; let further $\hat\cE'\subset\hat\cE$ and $\cE'\subset\cE$ be index sets, and let $(\wt h,\wt q)\in\cA_\phg^{\hat\cE',\cE'}(\wt X\setminus\wt K^\circ;S^2\wt T^*\wt X)\oplus\cA_\phg^{\hat\cE'-1,\cE'}(\wt X\setminus\wt K^\circ;S^2\wt T^*\wt X)$. Using the notation of~\S\usref{SsGq}, we then have
  \begin{subequations}
  \begin{alignat}{2}
  \label{EqCETotAccNormHat}
    &P(\wt\gamma+\wt h,\wt k+\wt q;\Lambda)-P(\wt\gamma,\wt k;\Lambda)-\eps^{-2}\hat\chi\sfs\Bigl(L_{\hat\gamma,\hat k}\bigl(\sfs^{-1}\wt h,\sfs^{-1}(\eps\wt q)\bigr)\Bigr) &&\in \cA_\phg^{\hat\cE'+((\N_0+1)\cup\hat\cE)-2,\cE'+(\N_0\cup\cE)}, \\
  \label{EqCETotAccNormCirc}
    &P(\wt\gamma+\wt h,\wt k+\wt q;\Lambda)-P(\wt\gamma,\wt k;\Lambda)-\chi_\circ L_{\circ,\gamma,k}(\wt h,\wt q) &&\in \cA_\phg^{\hat\cE'+(\N_0\cup\cE)-2,\cE'+\cE}.
  \end{alignat}
  \end{subequations}
\end{cor}
\begin{proof}
  This follows from Lemma~\ref{LemmaCETotAcc} upon using~\eqref{EqCETotLinNormDiff}. Indeed, acting with
  \[
    L_{\wt\gamma,\wt k}\wt w-\hat\chi\sfs\circ L_{\hat\gamma,\hat k}\hat w\circ \sfs^{-1} \in \cA_\phg^{(\N_0+1)\cup\hat\cE,\N_0\cup\cE}\Diffq(\wt X)
  \]
  on $\wt w^{-1}(\wt h,\wt q)\in\cA_\phg^{\hat\cE'-2,\cE'}$ gives an element of the space in~\eqref{EqCETotAccNormHat}; note also that acting on the pair $(\wt h,\wt q)$ of symmetric 2-tensors, we have
  \[
    \hat w\sfs^{-1}\wt w^{-1}=\diag(\rho_\circ^{-2},\rho_\circ^{-1})\sfs^{-1}\diag(\hat\rho^{-2},\hat\rho^{-1})=\eps^{-2}\sfs^{-1}\diag(1,\eps).
  \]
  Moreover, the quadratic error term~\eqref{EqCETotAccPhg} lies in the space~\eqref{EqCETotAccNormHat} as well, since in the notation of Lemma~\ref{LemmaCETotAcc} we have $\hat\cF'\subset\hat\cE'+\hat\cE$ and $\cF'\subset\cE'+\cE$. The argument for~\eqref{EqCETotAccNormCirc} is completely analogous, now using that $w_\circ\wt w^{-1}=\Id$.
\end{proof}

\section{Gluing construction}
\label{SGl}

The data for our gluing problem are as follows.

\begin{definition}[Gluing data]
\label{DefGlData}
  Let $X$ be an $n$-dimensional smooth manifold, $n\geq 3$; let $\Lambda\in\R$. Let $\cE\subset\C\times\N_0$ be an index set with $\Re\cE>0$; let $\hat K\subset\R^n$ be compact with smooth boundary and connected complement. Then \emph{$\cE$-gluing data (with cosmological constant $\Lambda$)} consist of a point $\fp\in X$ and two pairs $(\hat\gamma,\hat k)$, $(\gamma,k)$ with the following properties: $(\hat\gamma,\hat k)$ is $\cE$-asymptotically flat on $\R^n\setminus\hat K^\circ$ (see Definition~\usref{DefCEAf}); $\gamma,k$ are smooth symmetric 2-tensors on $X$, with $\gamma$ Riemannian; and both pairs satisfy the nonlinear constraint equations
  \[
    P(\hat\gamma,\hat k;0)=0,\qquad
    P(\gamma,k;\Lambda)=0.
  \]
\end{definition}

Given $X$, we recall from Definition~\ref{DefGTot} the notation $\wt X$ for the total gluing space and $\hat X\cong\ol{\R^n}$ and $X_\circ=[X;\{\fp\}]$ for the boundary hypersurfaces of $\wt X$.

\begin{thm}[Main result, detailed version]
\label{ThmGlT}
  Let $\fp\in X$ and $(\hat\gamma,\hat k)$, $(\gamma,k)$ be $\cE$-gluing data (with cosmological constant $\Lambda$) as in Definition~\usref{DefGlData}. Suppose $\cU^\circ\subset X$ is a smoothly bounded connected open set containing $\fp$ with compact closure $\cU=\ol{\cU^\circ}$, and suppose that $(X,\gamma,k)$ has no KIDs in $\cU^\circ$ (see Definition~\usref{DefCEPctKID}). Let $0<\delta<\min(\Re\cE,n-2)$. Then there exist index sets $\hat\cE_\sharp,\cE_\sharp\subset\C\times\N_0$ with $\Re\hat\cE_\sharp>0$ and $\Re\cE_\sharp>\delta$, and a $(\hat\cE_\sharp,\cE_\sharp)$-smooth total family $(\wt\gamma,\wt k)$ (defined on $\wt X\setminus\wt K^\circ$, with $\wt K^\circ$ given in Definition~\usref{DefCETotID}) with the following properties:
  \begin{enumerate}
  \item\label{ItGlTBdy} the boundary data of $(\wt\gamma,\wt k)$ are equal to $(\hat\gamma,\hat k)$ and $(\gamma,k)$;
  \item\label{ItGlTLoc} we have $(\wt\gamma,\wt k)=(\gamma,k)$ near $\wt X\setminus([0,1)\times\cU^\circ)$, and $(\sfs^{-1}\wt\gamma,\sfs^{-1}(\eps\wt k))=(\hat\gamma,\hat k)$ near $\hat r\leq\hat R_0$ where $\hat R_0>0$ is any fixed constant such that $B(0,\hat R_0)\supset\hat K$;
  \item\label{ItGlTSolv} there exists $\eps_\sharp>0$ so that $(\wt\gamma,\wt k)$ solves the constraint equations in $\eps<\eps_\sharp$, i.e.\ $P(\wt\gamma,\wt k;\Lambda)=0$ on every $\eps$-level set of $\wt X\setminus\wt K^\circ$ for $\eps<\eps_\sharp$.
  \end{enumerate}
\end{thm}

\begin{rmk}[$\cE$ and the Positive Mass Theorem]
\label{RmkGlTPMT}
  If $\Re\cE>n-2$, then the ADM energy $E$ and ADM linear momentum $P=(P_1,P_2,P_3)$ of the data set $(\hat\gamma,\hat k)$ are equal to $0$ (which implies that the ADM mass vanishes); see e.g.\ \cite[\S2]{EichmairHuangLeeSchoenPMT} for the definitions. This conjecturally implies that $(\hat X,\hat\gamma,\hat k)$ can be isometrically embedded as a spacelike hypersurface in $(n+1)$-dimensional Minkowski space. This is known for $n\leq 7$ \cite{EichmairJangReduction,HuangLeeSpacetimePMTRigidity}, following earlier work in the case of spin manifolds \cite{WittenPMT,BeigChruscielPMT,ChruscielMaertenPMT}. See also \cite{SchoenYauPMT,EichmairHuangLeeSchoenPMT,SchoenYauPMTHigh,LohkampPMT,HirschZhangSpacetimePMT}. We also remark that while we assume strong regularity of $\hat\gamma,\hat k$ at infinity, the decay may be too weak to allow for a definition of the ADM energy and linear momentum.
\end{rmk}

\subsection{Formal solution}
\label{SsGlFo}

By appropriately alternating the solution operators for the linearized constraints map of the two given initial data sets, we shall now prove:

\begin{prop}[Polyhomogeneous formal solution of the gluing problem]
\label{PropGlFo}
  Under the assumptions of Theorem~\usref{ThmGlT} and using its notation, there exists a $(\hat\cE_\sharp,\cE_\sharp)$-smooth total family $(\wt\gamma,\wt k)$ satisfying properties~\eqref{ItGlTBdy}, \eqref{ItGlTLoc}, but only:
  \begin{enumerate}
  \item[{\rm(3')}] $(\wt\gamma,\wt k)$ is a formal solution of the constraint equations, i.e.\ $P(\wt\gamma,\wt k;\Lambda)\in\CIdot(\wt X\setminus\wt K^\circ;\ul\R\oplus\wt T^*\wt X)$ vanishes to infinite order at $\hat X$ and $X_\circ$.
  \end{enumerate}
\end{prop}
\begin{proof}[Proof of Proposition~\usref{PropGlFo}]
  Recall that $\hat\rho$ and $\rho_\circ\in\CI(\wt X)$ are boundary defining functions of $\hat X$ and $X_\circ$, respectively; for convenience, we choose them so that $\hat\rho\rho_\circ=\eps$. Moreover, by replacing $\cE$ with $\cE\cup(\N_0+n-2)$, we may assume that $\min\Re\cE\leq n-2$; and we replace $\cE$ by its nonlinear closure $\bigcup_{j\in\N}j\cE$. Lastly, we choose the localizer $\hat\chi\in\CI(\wt X)$ to a collar neighborhood of $\hat X$ so that $\supp\hat\chi\subset\wt\upbeta^{-1}([0,1)\times\cU^\circ)$.

  \pfstep{Naive gluing.} Fix geodesic normal coordinates $x\in\R^n$, $|x|<r_0$, around $\fp$ with respect to the metric $\gamma$, and set $\fe=\sum_{j=1}^n(\dd x^j)^2$; thus, $\gamma=\fe$ at $\fp$. We regard $\hat x=\frac{x}{\eps}$ as a map from a neighborhood of $\hat X\subset\wt X$ to $\ol{\R^n}$. Writing then
  \[
    \hat\gamma = \hat\gamma_{\hat i\hat j}(\hat x)\,\dd\hat x^i\,\dd\hat x^j,
  \]
  we regard $\sfs(\hat\gamma)=\hat\gamma_{\hat i\hat j}(\hat x)\,\dd x^i\,\dd x^j$ (in the notation of Definition~\ref{DefGTotScale}) as a section of $S^2\wt T^*\wt X\to\wt X$ in a collar neighborhood of $\hat X$; since $\hat\gamma_{\hat i\hat j}-\delta_{i j}$ is polyhomogeneous, with index set $\cE$, in the defining function $|\hat x|^{-1}$ of $X_\circ$, we have
  \[
    \hat\chi\bigl(\sfs(\hat\gamma)-\fe\bigr) \in \cA_\phg^{\N_0,\cE}(\wt X\setminus\wt K^\circ;S^2\wt T^*\wt X),
  \]
  where $\wt K=\{(\eps,\hat x)\colon\hat x\in\hat K\}$ as in Definition~\ref{DefCETotID}, and where $\hat\chi\in\CI(\wt X)$ is $1$ near $\hat X$ and is supported in $|x|<r_0$. Similarly,
  \[
    \eps^{-1}\hat\chi\sfs(\hat k) \in \cA_\phg^{\N_0-1,\cE}(\wt X\setminus\wt K^\circ;S^2\wt T^*\wt X).
  \]
  Therefore, the naive gluing
  \begin{equation}
  \label{EqFoNaive}
    (\wt\gamma_0,\wt k_0) := \bigl(\wt\upbeta^*\gamma + \hat\chi(\sfs(\hat\gamma)-\fe),\ \wt\upbeta^*k+\eps^{-1}\hat\chi\sfs(\hat k)\bigr)
  \end{equation}
  of the two pairs of initial data is an $(\hat\cE_0,\cE_0)$-smooth total family in the sense of Definition~\ref{DefCETotID}, where we set $(\hat\cE_0,\cE_0)=(\N_0+1,\cE)$; its boundary data are $(\hat\gamma,\hat k)$ and $(\gamma,k)$. By Proposition~\ref{PropCETot}\eqref{ItCETotPhg}, we have
  \[
    P(\wt\gamma_0,\wt k_0;\Lambda) \in \cA_\phg^{\hat\cF_0-2,\cF_0}(\wt X\setminus\wt K^\circ;\ul\R\oplus\wt T^*\wt X), \qquad \hat\cF_0:=\N_0+1,\quad \cF_0:=\cE.
  \]
  The absence of the $\cO(\hat\rho^{-2})$ leading order term of the error $P(\wt\gamma_0,\wt k_0;\Lambda)$ at $\hat X$ and the strong vanishing at $X_\circ$ are due to the fact that $(\hat\gamma,\hat k)$ and $(\gamma,k)$ satisfy the constraint equations. This is the starting point for an iterative construction of a formal solution to the gluing problem. We record here that
  \[
    \min\Re\cF_0 - \min\Re\hat\cF_0 = \min\Re\cE-1 > \delta-1.
  \]

  \pfstep{Solving away to leading order at $\hat X$.} Fix any $0<\eta<\min(1,\delta)$. Suppose $(\wt\gamma_\ell,\wt k_\ell)$ (for some $\ell\in\N_0$, starting with $\ell=0$) is an $(\hat\cE_\ell,\cE_\ell)$-smooth total family, with boundary data equal to $(\hat\gamma,\hat k)$ and $(\gamma,k)$, and so that
  \[
    P(\wt\gamma_\ell,\wt k_\ell;\Lambda) \in \cA_\phg^{\hat\cF_\ell-2,\cF_\ell}(\wt X\setminus\wt K^\circ;\ul\R\oplus\wt T^*\wt X),
  \]
  where the index sets satisfy
  \begin{equation}
  \label{EqFoIndInd}
  \begin{alignedat}{2}
    \Re\hat\cE_\ell&>1-\eta,&\qquad
    \hat\cF_\ell+(\N_0\cup\hat\cE_\ell)&\subset\hat\cF_\ell\subset\hat\cE_\ell, \\
    \Re\cE_\ell&>\delta-\eta, &\qquad
    \cF_\ell+(\N_0\cup\cE_\ell)&\subset\cF_\ell\subset\cE_\ell, \\
    &\min\Re\cF_\ell-\min\Re\hat\cF_\ell > \delta-1.\hspace{-10em}&&&
  \end{alignedat}
  \end{equation}
  Moreover, we assume that $(\wt\gamma_\ell,\wt k_\ell)=(\gamma,k)$ in a neighborhood of $[0,1)\times(X_\circ\setminus\cU)$, and $(\sfs^{-1}\wt\gamma_\ell,\sfs^{-1}(\eps\wt k_\ell))=(\hat\gamma,\hat k)$ in a neighborhood of $\hat r\leq\hat R_0$.

  Consider $(z,m)\in\hat\cF_\ell$ so that $(z,m+1)\notin\hat\cF_\ell$ and $\Re z=\min\Re\hat\cF_\ell$. Then there exists $(\wt f,\wt j)\in\cA_\phg^{\N_0,\cF_\ell}(\wt X\setminus\wt K^\circ;\ul\R\oplus\wt T^*\wt X)$, vanishing near $\hat r\leq\hat R_0$, so that
  \[
    P(\wt\gamma_\ell,\wt k_\ell;\Lambda) \equiv \hat\rho^{z-2}(\log\hat\rho)^m(\wt f,\wt j) \bmod \cA_\phg^{\hat\cF_\ell^\flat-2,\cF_\ell},\qquad
    \hat\cF_\ell^\flat:=\hat\cF_\ell\setminus\{(z,m)\}.
  \]
  Our present aim is to correct $(\wt\gamma_\ell,\wt k_\ell)$ so as to remove the term $\hat\rho^{z-2}(\log\hat\rho)^m(\wt f,\wt j)$ from $P(\wt\gamma_\ell,\wt k_\ell;\Lambda)$. To this end, we write $\hat\rho=\eps\rho_\circ^{-1}$ and thus
  \[
    \hat\rho^{z-2}(\log\hat\rho)^m(\wt f,\wt j)=\sum_{i=0}^m{m\choose i}(-1)^i\eps^z(\log\eps)^{m-i}\cdot\eps^{-2}\rho_\circ^{-(z-2)}(\log\rho_\circ)^i(\wt f,\wt j).
  \]
  Consider the $i$-th summand; we have
  \[
    (\hat f_i,\hat j_i):=\sfs^{-1}\bigl(\rho_\circ^{-(z-2)}(\log\rho_\circ)^i(\wt f,\wt j)|_{\hat X}\bigr)\in\cA_\phg^{\cG_i+2}(\hat X\setminus\hat K^\circ;\ul\R\oplus\Tsc^*\hat X)
  \]
  where $\cG_i=\cF_\ell-z+(0,i)$. By~\eqref{EqFoIndInd}, we have $\Re\cG_i>\delta-1$. We can thus apply Proposition~\ref{PropCEAfSolv}\eqref{ItCEAfSolvPhg} (with $\alpha_\circ:=\delta-1<n-2$ indeed) to produce $(\hat h_i,\hat q_i)\in\cA_\phg^{\cG_i\extcup\hat\cS}\oplus\cA_\phg^{(\cG_i\extcup\hat\cS)+1}$, vanishing near $\hat r\leq\hat R_0$, with
  \[
    L_{\hat\gamma,\hat k}(\hat h_i,\hat q_i)=-(\hat f_i,\hat j_i);
  \]
  the index set $\hat\cS$ here only depends on $\delta$ and satisfies $\Re\hat\cS>\delta-1$. Put then
  \begin{align*}
    (\tilde h,\tilde q) &:= \hat\chi\sum_{i=0}^m {m\choose i}(-1)^i\eps^z(\log\eps)^{m-i}\bigl(\sfs(\hat h_i),\eps^{-1}\sfs(\hat q_i)\bigr) \\
    &\hspace{4em} \in \cA_\phg^{(z,m),\cF^\sharp_\ell}(\wt X;S^2\wt T^*\wt X)\oplus\cA_\phg^{(z,m)-1,\cF^\sharp_\ell}(\wt X;S^2\wt T^*\wt X),
  \end{align*}
  where
  \begin{equation}
  \label{EqFoFlSharp}
    \cF^\sharp_\ell=\bigl((\cF_\ell-z)\extcup\hat\cS\bigr)+(z,m).
  \end{equation}
  Note that $\cF^\sharp_\ell\supset\cF_\ell$ and $\min\Re\cF^\sharp_\ell=\min(\min\Re\cF_\ell,\min\Re\hat\cF_\ell+\min\Re\hat\cS)$. Therefore, $\min\Re\cF^\sharp_\ell-\min\Re\hat\cF_\ell>\delta-1$, and $\Re\cF_\ell^\sharp>\min(\delta-\eta,(1-\eta)+(\delta-1))=\delta-\eta$; we may thus replace $\cE_\ell$ by the nonlinear closure of $\cE_\ell\cup(\cF_\ell^\sharp+(\N_0\cup\cE_\ell))$, and replace $\cF_\ell$ by $\cF^\sharp_\ell+(\N_0\cup\cE_\ell)$, without affecting the validity of~\eqref{EqFoIndInd}.

  Having adjusted the index sets in this fashion, we now apply Corollary~\ref{CorCETotAccNorm} (concretely, the membership~\eqref{EqCETotAccNormHat}), with $(\hat\cE,\cE)=(\hat\cE_\ell,\cE_\ell)$ and $(\hat\cE',\cE')=((z,m),\cF_\ell)$. This gives
  \begin{align*}
    P(\wt\gamma_\ell+\wt h,\wt k_\ell+\wt q;\Lambda) &\equiv P(\wt\gamma_\ell,\wt k_\ell;\Lambda) - \hat\rho^{z-2}(\log\hat\rho)^m(\wt f,\wt j) \bmod \cA_\phg^{\hat\cF'_\ell-2,\cF_\ell}, \\
    &\hspace{6em} \hat\cF'_\ell=(z,m)+\bigl((\N_0+1)\cup\hat\cE_\ell\bigr).
  \end{align*}
  Note here that $\hat\cF'_\ell\subsetneq\hat\cF_\ell$, and $\min\Re\hat\cF'_\ell>\min\Re\hat\cF_\ell+(1-\eta)$.

  Proceeding in this manner for all $(z,m)\in\hat\cF_\ell$ with the property that $\min\Re\hat\cF_\ell\leq\Re z\leq\min\Re\hat\cF_\ell+(1-\eta)$ (and adjusting $\cE_\ell$ and $\cF_\ell$ as above) produces $(\tilde h,\tilde q)\in\cA_\phg^{\hat\cF_\ell,\cF_\ell}\oplus\cA_\phg^{\hat\cF_\ell-1,\cF_\ell}$ so that for the $(\hat\cE_\ell,\cE_\ell)$-smooth total family
  \[
    (\wt\gamma^\ell,\wt k^\ell):=(\wt\gamma_\ell,\wt k_\ell)+(\wt h,\wt q),
  \]
  we have
  \begin{equation}
  \label{EqFoSolvHatX}
    P(\wt\gamma^\ell,\wt k^\ell;\Lambda) \in \cA_\phg^{\hat\cF^\ell-2,\cF_\ell}(\wt X\setminus\wt K^\circ;\ul\R\oplus\wt T^*\wt X),\qquad
    \Re\hat\cF^\ell>\min\Re\hat\cF_\ell+(1-\eta).
  \end{equation}
  Thus, the remaining error has improved decay at $\hat X$ by more than $1-\eta$ orders as compared to $P(\wt\gamma_\ell,\wt k_\ell;\Lambda)$, while the adjustments of $\cE_\ell$ and $\cF_\ell$ were mild enough so as to preserve~\eqref{EqFoIndInd}.

  If $\min\Re\cF_\ell-\min\Re\hat\cF^\ell>\delta-1$ still, we may repeat this procedure (with $\hat\cF^\ell$ in place of $\hat\cF_\ell$), thus producing a further correction (with yet stronger decay at $\hat X$) to $(\wt\gamma^\ell,\wt k^\ell)$ whose remaining error has yet another $1-\eta$ orders of vanishing at $\hat X$; and so on. Upon repeating this procedure the maximal number of times,\footnote{Since $\cF_\ell\neq\emptyset$, cf.\ \eqref{EqFoFlSharp}, this maximal number is finite for any fixed $\ell$.} and by a minor abuse of notation again calling the final approximate solution $(\wt\gamma^\ell,\wt k^\ell)$, we thus still have~\eqref{EqFoSolvHatX} and $(\wt\gamma^\ell,\wt k^\ell)=(\gamma,k)$ near $[0,1)\times(X_\circ\setminus\cU)$, and furthermore
  \begin{equation}
  \label{EqFoSolvHatXFinal}
  \begin{gathered}
    \Re\hat\cE_\ell>1-\eta, \qquad \hat\cF^\ell+(\N_0\cup\hat\cE_\ell)\subset\hat\cF^\ell\subset\hat\cE_\ell, \\
    \Re\cE_\ell>\delta-\eta, \qquad \cF_\ell+(\N_0\cup\cE_\ell)\subset\cF_\ell\subset\cE_\ell, \\
    \min\Re\cF_\ell-\min\Re\hat\cF^\ell \in (\delta-2+\eta,\delta-1].
  \end{gathered}
  \end{equation}

  \pfstep{Solving away to leading order at $X_\circ$.} The arguments near $X_\circ$ are very similar to (and somewhat simpler, due to the absence of scaling considerations, than) those near $\hat X$. We thus consider $(z,m)\in\cF_\ell$ with $(z,m+1)\notin\cF_\ell$ and $\Re z=\min\Re\cF_\ell$, and $(\wt f,\wt j)\in\cA_\phg^{\hat\cF^\ell-2,\N_0}(\wt X\setminus\wt K^\circ;\ul\R\oplus\wt T^*\wt X)$ with the property that
  \[
    P(\wt\gamma^\ell,\wt k^\ell;\Lambda) \equiv \rho_\circ^z(\log\rho_\circ)^m(\wt f,\wt j) \bmod \cA_\phg^{\hat\cF^\ell-2,\cF_\ell^\flat},\qquad
    \cF_\ell^\flat:=\cF_\ell\setminus\{(z,m)\}.
  \]
  In the expansion
  \[
    \rho_\circ^z(\log\rho_\circ)^m(\wt f,\wt j) = \sum_{i=0}^m {m\choose i}(-1)^i \eps^z(\log\eps)^{m-i}\cdot \hat\rho^{-z}(\log\hat\rho)^i(\wt f,\wt j),
  \]
  we consider the $i$-th summand. Note that
  \[
    (f_i,j_i) := \hat\rho^{-z}(\log\hat\rho)^i(\wt f,\wt j)|_{X_\circ} \in \cA_\phg^{\cG_i-2}(X_\circ;\ul\R\oplus\upbeta_\circ^*T^*X),\qquad \cG_i=\hat\cF^\ell-z+(0,i),
  \]
  and $(f_i,j_i)$ has support in $\upbeta_\circ^{-1}(\cU^\circ)$; moreover, $\min\Re\cG_i\geq-\delta+1$ by~\eqref{EqFoSolvHatXFinal}. Proposition~\ref{PropCEPctSolv}\eqref{ItCEPctSolvPhg} (with $\hat\alpha=-\delta+1-\eta>-n+2$) then produces
  \[
    (h_i,q_i) \in \cA_\phg^{\cG_i\extcup\cS_\circ}\oplus\cA_\phg^{(\cG_i\extcup\cS_\circ)-1},\qquad L_{\circ,\gamma,k}(h_i,q_i)=-(f_i,j_i),
  \]
  with $\supp h_i$, $\supp q_i\subset\cU^\circ$, and where the index set $\cS_\circ$ only depends on $\delta,\eta$ and satisfies $\Re\cS_\circ>-\delta+1-\eta$. We then put
  \begin{align*}
    (\wt h,\wt q) &:= \chi_\circ\sum_{i=0}^m {m\choose i}(-1)^i\eps^z(\log\eps)^{m-i} (h_i,q_i) \\
    &\hspace{4em} \in \cA_\phg^{\hat\cF^{\ell\sharp},(z,m)}(\wt X;S^2\wt T^*\wt X)\oplus\cA_\phg^{\hat\cF^{\ell\sharp}-1,(z,m)}(\wt X;S^2\wt T^*\wt X),
  \end{align*}
  where
  \[
    \hat\cF^{\ell\sharp} = \bigl( (\hat\cF^\ell-z)\extcup\cS_\circ\bigr) + (z,m).
  \]
  We have $\hat\cF^{\ell\sharp}\supset\hat\cF^\ell$ and $\min\Re\hat\cF^{\ell\sharp}=\min(\min\Re\hat\cF^\ell,\min\Re\cF_\ell+\min\Re\cS_\circ)$. Therefore, $\min\Re\cF_\ell-\min\Re\hat\cF^{\ell\sharp}<\delta-1+\eta$, and $\Re\hat\cF^{\ell\sharp}>\min(1-\eta,(\delta-2+\eta)+(2-\eta)+(-\delta+1-\eta))=1-\eta$. Thus, we can replace $\hat\cE_\ell$ by the nonlinear closure of $\hat\cE_\ell\cup(\hat\cF^{\ell\sharp}+(\N_0\cup\hat\cE_\ell))$, and replace $\hat\cF^\ell$ by $\hat\cF^{\ell\sharp}+(\N_0\cup\hat\cE_\ell)$, without affecting the validity of the first two lines of~\eqref{EqFoSolvHatXFinal}.

  Corollary~\ref{CorCETotAccNorm} (concretely, the membership~\eqref{EqCETotAccNormCirc}), with $(\hat\cE,\cE)=(\hat\cE_\ell,\cE_\ell)$ and $(\hat\cE',\cE')=(\hat\cF^\ell,(z,m))$ gives
  \begin{align*}
    P(\wt\gamma^\ell+\wt h,\wt k^\ell+\wt q;\Lambda) &\equiv P(\wt\gamma^\ell,\wt k^\ell;\Lambda) - \rho_\circ^z(\log\rho_\circ)^m(\wt f,\wt j) \bmod \cA_\phg^{\hat\cF^\ell-2,\cF'_\ell}, \qquad \cF'_\ell=(z,m)+\cE_\ell.
  \end{align*}
  Thus, $\Re\cF_\ell'>\min\Re\cF_\ell+(\delta-\eta)$. In this fashion, we solve away all terms of $P(\wt\gamma^\ell,\wt k^\ell;\Lambda)$ at $X_\circ$ corresponding to $(z,m)\in\cF_\ell$ with $\min\Re\cF_\ell\leq\Re z\leq\min\Re\cF_\ell+(\delta-\eta)$ at once. If subsequently we still have $\min\Re\cF_\ell'-\min\Re\hat\cF^\ell\in(\delta-2+\eta,\delta-1]$, we repeat this procedure. We thus produce $(\wt h,\wt q)\in\cA_\phg^{\hat\cF^\ell,\cF_\ell}\oplus\cA_\phg^{\hat\cF^\ell-1,\cF_\ell}$ (with a new index set $\hat\cE_\ell$) so that $(\wt\gamma_{\ell+1},\wt k_{\ell+1})=(\wt\gamma^\ell,\wt k^\ell)+(\wt h,\wt q)$ is a $(\hat\cE_{\ell+1},\cE_{\ell+1})$-smooth total family (with $\hat\cE_{\ell+1},\cE_{\ell+1}$ denoting the adjustments of the index sets $\hat\cE_\ell,\cE_\ell$ performed in the course of the argument) for which
  \[
    P(\wt\gamma_{\ell+1},\wt k_{\ell+1};\Lambda) \in \cA_\phg^{\hat\cF_{\ell+1}-2,\cF_{\ell+1}},
  \]
  where $\Re\hat\cF_{\ell+1}>\min\Re\hat\cF_\ell+(1-\eta)$ and $\Re\cF_{\ell+1}>\min\Re\cF_\ell+(\delta-\eta)$, and~\eqref{EqFoIndInd} holds for $\ell+1$ in place of $\ell$. This completes the inductive step. 

  \pfstep{Asymptotic summation.} We now have a sequence $(\wt\gamma_\ell,\wt k_\ell)$, $\ell=0,1,2,\ldots$, with the property that $P(\wt\gamma_\ell,\wt k_\ell;\Lambda)$ is polyhomogeneous on $\wt X$ with the index sets at $\hat X$ and $X_\circ$ having minimal real parts tending to $+\infty$, and so that $(\wt h_\ell,\wt q_\ell):=(\wt\gamma_{\ell+1},\wt k_{\ell+1})-(\wt\gamma_\ell,\wt k_\ell)$ is polyhomogeneous with index sets having real parts larger than any desired constant $C$ for all $\ell\geq\ell_0(C)$. The latter property enables us to take $(\wt\gamma,\wt k)$ to be an asymptotic sum
  \[
    (\wt\gamma,\wt k) \sim (\wt\gamma_0,\wt k_0) + \sum_{\ell=0}^\infty (\wt h_\ell,\wt q_\ell),
  \]
  which is an $(\hat\cE_\sharp,\cE_\sharp)$-smooth family for some index sets $\hat\cE_\sharp,\cE_\sharp$ which satisfy $\Re\hat\cE_\sharp>1-\eta$ and $\Re\cE_\sharp>\delta-\eta$ (from~\eqref{EqFoIndInd}). The former property implies, in view of Lemma~\ref{LemmaCETotAcc}\eqref{ItCETotAccPhg} and~\eqref{EqCETotAccLot}, that $P(\wt\gamma,\wt k;\Lambda)$ is polyhomogeneous on $\wt X$ with index sets which have arbitrarily large minimal real parts; thus, $P(\wt\gamma,\wt k;\Lambda)\in\CIdot(\wt X\setminus\wt K^\circ)$, as desired. The proof is complete.
\end{proof}

\begin{rmk}[Better regularity at $\hat X$]
\label{RmkFoBetter}
  In Proposition~\ref{PropGlFo}, one can arrange, for any desired value $N\in\N$, that $\Re\hat\cE_\sharp>N$; thus, $(\wt\gamma,\eps\wt k)$ is of class $\cC^N$ down to $\hat X$. Indeed, starting with the naive gluing $(\wt\gamma_0,\wt k_0)$ in~\eqref{EqFoNaive}, one solves away the error $P(\wt\gamma_0,\wt k_0;\Lambda)$ to leading order at $\hat X$ (where it has index set $(\N_0+1)-2$); arguing iteratively, once the error $P(\wt\gamma_\ell,\wt k_\ell;\Lambda)$ has index set $(\N_0+\ell+1)-2$ at $\hat X$, and an index set $\cE_\ell$ at $X_\circ$ with $\Re\cE_\ell>\Re\cE-\eta$, one solves away the leading term at $\hat X$ using Proposition~\ref{PropCEAfSolv}\eqref{ItCEAfSolvPhg} with $\alpha_\circ=(\Re\cE-\eta)-(\ell+1)$ (which satisfies the bound $\alpha_\circ<n-2$ by an increasingly wide margin as $\ell$ increases). Stopping this procedure at step $\ell\gg N$, one then solves away the error term to a high order at $X_\circ$ by means of Proposition~\ref{PropCEPctSolv}\eqref{ItCEPctSolvPhg} with large values of $\hat\alpha$ (so that the corrections vanish to high order at $\hat X$). Once one has thus constructed a $(\hat\cE_\flat,\cE_\flat)$-smooth total family $(\wt\gamma_\flat,\wt k_\flat)$, with $\Re\hat\cE_\flat>N$, and with $P(\wt\gamma_\flat,\wt k_\flat;\Lambda)$ vanishing to order $\gg N$ at $\hat X$ and $X_\circ$, one can correct $(\wt\gamma_\flat,\wt k_\flat)$ to a formal solution by solving away the error in turns at $\hat X$ and $X_\circ$ as in the proof of Proposition~\ref{PropGlFo}; an inspection of the proof shows that these corrections are polyhomogeneous and vanish to order $\gg N$ at $\hat X$ and $X_\circ$.
\end{rmk}

\subsection{True solution}
\label{SsGlT}

In this section, we prove Theorem~\ref{ThmGlT}. We start with the formal solution $(\wt\gamma,\wt k)$ from Proposition~\ref{PropGlFo}. Theorem~\ref{ThmGlT} is then a consequence of the following result:

\begin{prop}[Correction of a formal solution]
\label{PropGlT}
  With $(\wt\gamma,\wt k)$ as in Proposition~\usref{PropGlFo}, there exists $(\wt h,\wt q)\in\CIdot(\wt X\setminus\wt K^\circ;S^2\wt T^*\wt X\oplus S^2\wt T^*\wt X)$ with the following properties:
  \begin{enumerate}
  \item we have $(\wt h,\wt q)=0$ near $\wt X\setminus([0,1)\times\cU^\circ)$ and near $\{(\eps,x)\colon \eps x\in B(0,\hat R_0)\}\subset\wt X$;
  \item for $(\wt\gamma',\wt k')=(\wt\gamma+\wt h,\wt k+\wt q)$, we have $P(\wt\gamma',\wt k';\Lambda)=0$ in $\eps<\eps_\sharp$ for some $\eps_\sharp>0$.
  \end{enumerate}
\end{prop}
\begin{proof}
  We choose $\hat\rho$ and $\rho_\circ$ so that $\hat\rho\rho_\circ=\eps$. Fix $\alpha_2\in\R$, $\beta>0$, and $\hat\alpha,\alpha_\circ>0$ with $\alpha_\circ-\hat\alpha<n-2$. Recall $\wt w$ and $\wt w_2$ from Proposition~\ref{PropCETotSolv}. Set $\tilde\alpha=\hat\alpha+n-2-\frac{n}{2}$ and write
  \[
    w_L=e^{\beta/\tilde\rho_2}\tilde\rho_2^{-\alpha_2}\hat\rho^{-\tilde\alpha}\rho_\circ^{-\alpha_\circ}\wt w_2,\qquad
    w_R=\wt w\rho_\circ^{\alpha_\circ}\hat\rho^{\tilde\alpha}\tilde\rho_2^{\alpha_2}e^{-\beta/\tilde\rho_2},
  \]
  for the weights appearing in the definition~\eqref{EqCETotSolvConj} of $\cL_{\wt\gamma,\wt k}=w_L L_{\wt\gamma,\wt k}w_R$. We shall construct $(\wt h,\wt q)$ of the form
  \[
    (\wt h,\wt q)=w_R^2 L_{\wt\gamma,\wt k}^*w_L(\wt f,\wt j)=w_R\cL_{\wt\gamma,\wt k}^*(\wt f,\wt j),
  \]
  where $(\wt f,\wt j)$ solves the equation
  \begin{equation}
  \label{EqGlTPDE}
    w_L P\bigl((\wt\gamma,\wt k)+w_R\cL_{\wt\gamma,\wt k}^*(\wt f,\wt j)\bigr) = 0.
  \end{equation}
  Recalling the notation $Q(\wt\gamma,\wt k;\wt h,\wt q)$ for the quadratic error~\eqref{EqCETotAccQ}, this equation is equivalent to
  \[
    \cL(\wt f,\wt j) = -w_L P(\wt\gamma,\wt k;\Lambda) - w_L Q\bigl(\wt\gamma,\wt k; w_R\cL_{\wt\gamma,\wt k}^*(\wt h,\wt j)\bigr),
  \]
  where $\cL:=\cL_{\wt\gamma,\wt k}\cL_{\wt\gamma,\wt k}^*$ is an elliptic unweighted q-00-operator (see~\eqref{EqCETotSolvLLstar}) which is uniformly invertible as a map~\eqref{EqCETotSolv}. This lends itself to a fixed point formulation using the map
  \begin{equation}
  \label{EqGlTFixed}
    T \colon (\wt f,\wt j) \mapsto \cL^{-1}\Bigl( -w_L P(\wt\gamma,\wt k;\Lambda) - w_L Q\bigl(\wt\gamma,\wt k;w_R\cL_{\wt\gamma,\wt k}^*(\wt f,\wt j)\bigr) \Bigr),
  \end{equation}
  as we proceed to demonstrate.

  \pfstep{Existence of a solution with fixed regularity.} Fix $s>\frac{n}{2}+2$; let $N>0$ be arbitrary. We claim that there exist $\eps_0>0$ and $\delta>0$ so that $T$ is a uniform contraction on the $\delta$-ball in the space
  \[
    \cX_\eps^{s,N} := \eps^N H_{\qop,0 0,\eps}^{s+2}(\wt\Omega_\eps) \oplus \eps^N H_{\qop,0 0,\eps}^{s+1}(\wt\Omega_\eps;\wt T^*\wt X)
  \]
  for $\eps<\eps_0$, where $\wt\Omega\subset\wt X$ is the domain defined in Proposition~\ref{PropCETotSolv}. To show this, note first that the norm of $w_L P(\wt\gamma,\wt k;\Lambda)$ in the space
  \[
    \cY_\eps^{s,N} := \eps^N H_{\qop,0 0,\eps}^{s-2}(\wt\Omega_\eps) \oplus \eps^N H_{\qop,0 0,\eps}^{s-1}(\wt\Omega_\eps;\wt T^*\wt X)
  \]
  is bounded by $C_1\eps$ since $P(\wt\gamma,\wt k;\Lambda)$ lies in $\CIdot(\wt X)$ and vanishes near $\tilde\rho_2=0$. (Here, $C_1$ depends on $s,N$, and tacitly of course on the fixed constants $\alpha_2,\beta,\alpha_\circ,\hat\alpha$ and on the tensors $\wt\gamma,\wt k$.) Next, when $\|(\wt f,\wt j)\|_{\cX_\eps^{s,N}}<\delta$ and thus
  \[
    \|w_R\cL_{\wt\gamma,\wt k}^*(\wt f,\wt j)\|_{\hat\rho^{\hat\alpha}\rho_\circ^{\alpha_\circ}\tilde\rho_2^{\alpha_2}e^{-\beta/\tilde\rho_2}H_{\qop,0 0,\eps}^s\oplus\hat\rho^{\hat\alpha-1}\rho_\circ^{\alpha_\circ}\tilde\rho_2^{\alpha_2}e^{-\beta/\tilde\rho_2}H_{\qop,0 0,\eps}^s} < C'_2\delta,
  \]
  an application of Lemma~\ref{LemmaCETotAcc}\eqref{ItCETotAccq} gives, a fortiori,
  \[
    \bigl\|w_L Q\bigl(\wt\gamma,\wt k;w_R\cL_{\wt\gamma,\wt k}^*(\wt f,\wt j)\bigr)\bigr\|_{\cY_\eps^{s,N}} \leq C_2\|(\wt f,\wt j)\|_{\cX_\eps^{s,N}}^2 < C_2\delta^2.
  \]
  Here, $C_2$ and $C'_2$ are constants depending only on $s,N$. Finally, write $C_3$ for an upper bound on the ($N$-independent) operator norm of $\cL^{-1}\colon\cY_\eps^{s,N}\to\cX_\eps^{s,N}$ for $\eps<\eps_0$, where $\eps_0$ is given by Proposition~\ref{PropCETotSolv}. Fix $\delta_0>0$ so that $C_3 C_2\delta_0^2<\frac{\delta_0}{2}$; for all $\delta\leq\delta_0$, there then exists $\eps_1\in(0,\eps_0)$ so that $C_3(C_1\eps+C_2\delta^2)<\delta$ for $\eps<\eps_1$. The map $T$ thus maps the $\delta$-ball in $\cX_\eps^{s,N}$ into itself.

  Next, consider the difference
  \begin{equation}
  \label{EqGlTDiff}
    T(\wt f_1,\wt j_1)-T(\wt f_2,\wt j_2) = \cL^{-1}\Bigl( w_L Q\bigl(\wt\gamma,\wt k;w_R\cL_{\wt\gamma,\wt k}^*(\wt f_2,\wt j_2)\bigr)-w_L Q\bigl(\wt\gamma,\wt k;w_R\cL_{\wt\gamma,\wt k}^*(\wt f_1,\wt j_1)\bigr)\Bigr)
  \end{equation}
  for $(\wt h_j,\wt q_j)\in\cX_\eps^{s,N}$ of norm less than $\delta$. Using the estimate~\eqref{EqCETotAccq}, the $\cX_\eps^{s,N}$-norm of~\eqref{EqGlTDiff} is bounded by $C_3 C\delta\|(\wt f_1,\wt j_1)-(\wt f_2,\wt j_2)\|_{\cX_\eps^{s,N}}$. If we choose $\delta$ so that $C_3 C\delta<\frac12$ (and then $\eps_1$ as above), we conclude that $T$ is a contraction on the $\delta$-ball of $\cX_\eps^{s,N}$, uniformly for $\eps<\eps_1$.

  We therefore obtain a solution $(\wt f,\wt j)$ of the equation~\eqref{EqGlTPDE} in $\eps<\eps_1=\eps_1(s,N)$ which is uniformly bounded in the space $\cX_\eps^{s,N}$. Moreover, the solution produced by the contraction mapping principle is independent of the choices of $N$ and $s$, though the constant $\eps_1$ given by the above arguments may depend on these choices. But note that whether or not a family of tensors is uniformly bounded in $\cX_\eps^{s,N}$ depends on $N$ only in an arbitrarily small neighborhood of $\eps=0$; therefore, we can take $\eps_1=\eps_1(s)$ to be independent of $N$. Furthermore, since the construction of $(\wt f,\wt j)$ is independent of $s$, we automatically have
  \begin{equation}
  \label{EqGlTReg0}
    (\wt f,\wt j)\in\cX_\eps^{s,N},\qquad \eps<\eps_1(s),\quad N\in\R.
  \end{equation}

  \pfstep{Higher q-00-regularity.} We proceed to show that, in fact, $(\wt f,\wt j)\in\cX_\eps^{s,N}$ for all $s,N\in\R$ for $\eps<\eps_1(s_0)$ where we fix $s_0>\frac{n}{2}+2$ (e.g.\ $s_0=\frac{n}{2}+3$). For any fixed $s$, we only need to prove this for $\eps\in[\eps_1(s),\eps_1(s_0))$ in view of~\eqref{EqGlTReg0}, but we give an argument which directly works in the full range $\eps<\eps_1(s_0)$. To wit, consider the localization of $(\wt f,\wt j)$ to a q-00-unit cell on $\wt\Omega$, or equivalently, in the notation of~\eqref{EqGTotCutoffs} and~\eqref{EqGqPhis}, the localization of $\hat\phi^*(\hat\chi\wt f,\hat\chi\wt j)$ to a b-00-unit cell on $\{\hat r\geq\hat R_0\}\subset\hat X$ or of $\hat\phi_\circ^*(\chi_\circ\wt f,\chi_\circ\wt j)$ to a b-00-unit cell on $\upbeta_\circ^{-1}(\cU)\subset X_\circ$. Identifying such a unit cell with the unit ball in $\R^n$ via coordinate charts of the bounded geometry perspective, this localization has $\cC^{4,\alpha}\oplus\cC^{3,\alpha}$-norms (with $\alpha\in(0,s-\frac{n}{2}-2)$) uniformly bounded by $C_N\eps^N$ for all $N$; we use this bound for $N=1$. The ellipticity of $\cL$ as an unweighted q-00-operator, and this $C_1\eps$-bound of the localizations of $(\wt f,\wt j)$ implies that for sufficiently small $\eps$ (independently of $s$) the PDE~\eqref{EqGlTPDE} in such charts is a quasilinear Douglis--Nirenberg elliptic system (with uniform dependence on the unit cell); such systems are discussed in Appendix~\ref{SDN}. By Proposition~\ref{PropDNEll}, this implies $\cO(\eps^N)$ bounds in $\cC^{k+2,\alpha}\oplus\cC^{k+1,\alpha}$ on the localizations of $(\wt f,\wt j)$ for all $k$, as claimed.

  \pfstep{Conormality in $\eps$.} Recalling Definition~\ref{DefCETotReg}, we may run the above contraction mapping argument on the space $\cC^0([0,\eps_1(s_0)),\cX_\eps^{s_0,N})$ using Lemma~\ref{LemmaCETotReg} and deduce that $(\wt f,\wt j)$ is in fact continuous with values in $\cX_\eps^{s_0,N}$ (for all $N$); continuity in $\cX_\eps^{s,N}$ for all $s\geq s_0$ follows from a parameter-dependent variant of Proposition~\ref{PropDNEll} (the proof of which requires choosing the small parameter $\delta$ and the elliptic parametrix $Q$ there in a continuous fashion). We conclude that
  \begin{align*}
    (\wt h,\wt q) = w_R\cL_{\wt\gamma,\wt k}^*(\wt f,\wt j) &\in \eps^N\hat\rho^{\hat\alpha}\rho_\circ^{\alpha_\circ}\tilde\rho_2^{\alpha_2}e^{-\beta/\tilde\rho_2}H_{\qop,0 0,\eps}^s(\wt X\setminus\wt K^\circ;S^2\wt T^*\wt X) \\
      &\quad\qquad\qquad \oplus\eps^N\hat\rho^{\hat\alpha-1}\rho_\circ^{\alpha_\circ}\tilde\rho_2^{\alpha_2}e^{-\beta/\tilde\rho_2}H_{\qop,0 0,\eps}^s(\wt X\setminus\wt K^\circ;S^2\wt T^*\wt X)
  \end{align*}
  for all $s,N$, with continuous dependence on $\eps\in(0,\eps_\sharp)$;

  We can improve this further: for some $\eps_\sharp<\eps_1(s_0)$ depending only on $s_0$, we have $(\wt f,\wt j)\in\cA_1([0,\eps_\sharp),\cX_\eps^{s_0,N})$. Indeed, we may run the contraction mapping argument on the space $\cA_1([0,\eps_\sharp),\cX_\eps^{s_0,N})$ by using Lemma~\ref{LemmaCETotReg} and direct differentiation of the argument of $\cL^{-1}$ in~\eqref{EqGlTFixed} along the testing operator $\cR$ (see the paragraph preceding Lemma~\ref{LemmaCETotReg}). Proposition~\ref{PropDNEll}, with added $\cC^1$ dependence on the parameter $\eps$, improves this to
  \begin{equation}
  \label{EqGlTA1}
    (\wt f,\wt j)\in\cA_1([0,\eps_\sharp),\cX_\eps^{s,N})\qquad \forall\, s,N\in\R.
  \end{equation}

  We claim that $(\wt f,\wt j)\in\cA_k([0,\eps_\sharp),\cX_\eps^{s,N})$ for \emph{all} $k\in\N$, with $\eps_\sharp$ fixed. Using the testing operator $\cR\in\Vb(\wt X)$ again, this follows most easily by direct differentiation of~\eqref{EqGlTPDE}. We explain this schematically, for simplicity of notation. The PDE~\eqref{EqGlTPDE} is of the form $\cP(\eps,u(\eps);\Lambda)=0$, so differentiating once along $\eps\pa_\eps$ (which is justified in view of~\eqref{EqGlTA1}), a stand-in for $\cR$, gives $\cL_{\eps,u(\eps)}(\eps\pa_\eps u)=-(\eps\pa_\eps\cP)(\eps,u(\eps))$ where $\cL_{\eps,u(\eps)}$ is the linearization of $\cP$ in the second argument. Differentiating this once more gives an equation of the form $\cL_{\eps,u(\eps)}((\eps\pa_\eps)^2 u)=f$ where $f$ depends only on $\eps,u(\eps),\pa_\eps u(\eps)$ and in particular depends continuously on $\eps$. While a priori $(\eps\pa_\eps)^2 u$ is defined only as a distribution in $\eps$ with values in $\cX_\eps^{s,N}$, the invertibility of $\cL$ (which in our concrete setting is the operator $\cL_{\wt\gamma+\wt h,\wt k+\wt q}\cL_{\wt\gamma+\wt h,\wt k+\wt q}^*$) implies that $(\eps\pa_\eps)^2 u\in\cC^0$ after all. (See Remark~\ref{RmkCETotSolvCont}.) This can be iterated, giving the desired conormal regularity of $(\wt f,\wt j)$.

 By Sobolev embedding (Corollary~\ref{CorBAn00Sob}), we can now conclude that $(\wt h,\wt q)\in\eps^N\CI(\wt X\setminus\wt K^\circ;S^2\wt T^*\wt X\oplus S^2\wt T^*\wt X)$ (for all $N$) vanishes in $\hat r<\hat R_0$ and outside of $[0,1)\times\cU$. The proof is complete.
\end{proof}

\section{Applications}
\label{SA}

As an immediate application of Theorem~\ref{ThmGlT}, we explain in~\S\ref{SsAB} how to glue exact Kerr initial data into a given generic initial data set. On the other hand, all known explicit initial data sets are contained in spacetimes with symmetries, and thus the initial data do admit KIDs; we thus demonstrate in~\S\ref{SsAK} how to extend our methods so as to glue any asymptotically flat data set into the initial data of an exact Kerr--de~Sitter (or Kerr, or Kerr--anti de~Sitter) black hole. Combining these two applications, we can thus produce initial data sets by gluing small Kerr black holes into Kerr (or Kerr--de~Sitter, or Kerr--anti de~Sitter) initial data.

For $\Lambda\in\R$ and parameters $\bhm\in\R$, $\bha\in\R$, the Kerr (for $\Lambda=0$), Kerr--de~Sitter (for $\Lambda>0$), and Kerr--anti de~Sitter (for $\Lambda<0$) metric $g_{\Lambda,\bhm,\bha}$ describes a rotating black hole with mass $\bhm$ and specific angular momentum $\bha$; in the case $\bha=0$, the metric $g_{\Lambda,\bhm,0}$ reduces to the Schwarzschild or Schwarzschild--(anti) de~Sitter metric \cite{SchwarzschildPaper,KerrKerr,CarterHamiltonJacobiEinstein}. In Boyer--Lindquist coordinates, and using the notation from \cite[\S1]{HintzKdSMS}, we have
\begin{align}
\label{EqAKerr}
\begin{split}
  g_{\Lambda,\bhm,\bha} &:= -\frac{\mu_{\Lambda,\bhm,\bha}(r)}{b_{\Lambda,\bhm,\bha}^2\varrho_{\Lambda,\bhm,\bha}^2(r,\theta)}(\dd t-\bha\,\sin^2\theta\,\dd\phi)^2 + \varrho_{\Lambda,\bhm,\bha}^2(r,\theta)\Bigl(\frac{\dd r^2}{\mu_{\Lambda,\bhm,\bha}(r)} + \frac{\dd\theta^2}{c_{\Lambda,\bhm,\bha}(\theta)}\Bigr) \\
    &\qquad + \frac{c_{\Lambda,\bhm,\bha}(\theta)\sin^2\theta}{b_{\Lambda,\bhm,\bha}^2\varrho_{\Lambda,\bhm,\bha}^2(r,\theta)}\bigl( (r^2+\bha^2)\dd\phi - \bha\,\dd t\bigr)^2,
\end{split} \\
  &\hspace{-1em}b_{\Lambda,\bhm,\bha}:=1+\frac{\Lambda\bha^2}{3},\quad
    c_{\Lambda,\bhm,\bha}(\theta):=1+\frac{\Lambda\bha^2}{3}\cos^2\theta,\quad
    \varrho_{\Lambda,\bhm,\bha}^2(r,\theta):=r^2+\bha^2\cos^2\theta. \nonumber
\end{align}
We call the parameters $(\Lambda,\bhm,\bha)$ \emph{subextremal} if $\mu_{\Lambda,\bhm,\bha}$ has four distinct real roots. The metrics $g_{\Lambda,\bhm,\bha}$ are stationary and axisymmetric: the vector fields $\pa_t$ and $\pa_\phi$ are Killing vector fields. When $\bha=0$, the metric is spherically symmetric. (For $\bhm=0$, we obtain the Minkowski, de~Sitter, and anti--de~Sitter metrics, which are maximally symmetric.) We shall henceforth write K(AdS) for Kerr, Kerr--de~Sitter, and Kerr--anti~de Sitter (depending on a fixed choice of $\Lambda$).

For a discussion of the geometry of K(AdS) spacetimes, we refer the reader to \cite{DafermosRodnianskiLectureNotes,VasyMicroKerrdS,HolzegelSmuleviciKAdSKG}. We mention here only that the above metric of a subextremal K(AdS) black hole can be extended analytically past the larger two roots $0<r_{\Lambda,\bhm,\bha}^e<r_{\Lambda,\bhm,\bha}^c$ of $\mu_{\Lambda,\bhm,\bha}(r)$, with the level set $r=r_{\Lambda,\bhm,\bha}^e$ containing the future event horizon.

\subsection{Gluing Kerr initial data into a given data set}
\label{SsAB}

Fix the parameters $\bhm,\bha\in\R$ of a Kerr black hole (not necessarily subextremal). Fix $r^e>0$ large enough so that the Kerr metric, in Boyer--Lindquist coordinates, is defined on the manifold
\[
  M := \R_t \times [r^e,\infty) \times \Sph^2_{\theta,\phi},
\]
We may take $r^e$ to be any number larger than $\bhm+\sqrt{\bhm^2-\bha^2}$ in the subextremal case $|\bha|<\bhm$. We introduce the usual Cartesian spatial coordinates $x=(x^1,x^2,x^3)$ on $[r^e,\infty)\times\Sph^2_{\theta,\phi}$. Fix now the generators $B_j\in\mathfrak{so}(1,3)\subset\R^{4\times 4}$, $j=1,2,3$, of Lorentz boosts in the three spatial coordinate directions $x^1,x^2,x^3$ to be $(B_j)_{p q}=\delta_{0 p}\delta_{j q}+\delta_{0 q}\delta_{j p}$; thus, for example, $\exp(w_1 B_1)$ is a Lorentz boost in the $x^1$-direction with rapidity $w_1$. Let $\bfB=(B_1,B_2,B_3)$; given a rapidity vector $\bfw\in\R^3$, we then write $t_\bfw=\exp(-\bfw\cdot\bfB)^*t$ and $x_\bfw=\exp(-\bfw\cdot\bfB)^*x$ for the time and spatial coordinates of an observer that is boosted with rapidity $\bfw$. Note that since $\exp(-\bfw\cdot\bfB)\in SO(1,3)$, the pullback of the Minkowski metric $-\dd t^2+\dd x^2$ is equal to $-\dd t_\bfw^2+\dd x_\bfw^2$.

\begin{lemma}[Boosted Kerr initial data]
\label{LemmaAB}
  Let $\Sigma\subset M$ denote a spacelike hypersurface which is a graph over $[r^e,\infty)\times\Sph^2$, and which for large $r$ is equal to a level set of $t_\bfw$. Let $\hat K\subset\R^3$ denote the closed ball of radius $r^e$, and identify $\Sigma\cong\R^3\setminus\hat K^\circ$ via a diffeomorphism which is given by $x$, resp.\ $x_\bfw$ near $\hat K$, resp.\ for large $r$. Then the initial data at $\Sigma$ of the metric $g_{0,\bhm,\bha}$ on $M$ are $\cE$-asymptotically flat on $\R^3\setminus\hat K^\circ$, where $\cE=\N_0+1$.
\end{lemma}

See also the closely related discussion in \cite[Appendices~E and F]{ChruscielDelayMapping}.

\begin{proof}[Proof of Lemma~\usref{LemmaAB}]
  For $\Lambda=0$, the coefficients $b_{\Lambda,\bhm,\bha}$ and $c_{\Lambda,\bhm,\bha}$ in~\eqref{EqAKerr} are equal to $1$. Let $\rho=r^{-1}$. One can then check that $g_{0,\bhm,\bha}=-\dd t^2+\dd x^2+g'$ where $g'\in\rho\CI(\ol{\R^3}\setminus\hat K^\circ;S^2\,\Tsc^*\ol{\R^3})$ is independent of $t$; that is, the coefficients of $g'$ with respect to $\pa_t,\pa_{x^1},\pa_{x^2},\pa_{x^3}$ are smooth functions of $r^{-1}$ and $\omega=\frac{x}{|x|}$ and vanish at $r^{-1}=0$. A fortiori, $g_{0,\bhm,\bha}$ is thus for large $r$ of the form
  \begin{equation}
  \label{EqABClass}
    {-}\dd t^2+\dd x^2 + g',\qquad g'\in\rho\CI_{\rm ext},
  \end{equation}
  where we write $\rho\CI_{\rm ext}$ for the space of smooth symmetric 2-tensors whose coefficients, for large $r$ and for $|\frac{t}{r}|<1$, are smooth functions of $\frac{t}{r}$, $r^{-1}$, and $\omega$. The Lemma now follows from the following two claims:
  \begin{enumerate}
  \item\label{ItABData} the initial data of any metric $g$ of the form~\eqref{EqABClass} at $t=0$ are $\cE$-asymptotically flat for large $r$;
  \item\label{ItABStable} the class of metrics of the form~\eqref{EqABClass} is preserved under pullback along a Lorentz boost.
  \end{enumerate}

  Regarding claim~\eqref{ItABData}, note that the metric induced by $g$ on $t^{-1}(0)$ satisfies
  \[
    \gamma\in\dd x^2+\rho\CI(\ol{\R^3}\setminus B(0,R);S^2\,\Tsc^*\ol{\R^3})
  \]
  for some large $R$ indeed. The future unit normal vector field $n$ at $t^{-1}(0)$ is $n=\pa_t+n'$, $n'\in\rho\CI(\ol{\R^3}\setminus B(0,R);\Tsc\ol{\R^4})$ (i.e.\ a linear combination of $\pa_t$, $\pa_{x^1}$, $\pa_{x^2}$, $\pa_{x^3}$ with coefficients of class $\rho\CI$). But the Christoffel symbols of $g$ (in the coordinates $t,x$) lie in $\rho^2\CI_{\rm ext}$ (the space of smooth functions of $\frac{t}{r}$, $r^{-1}$, $\omega$ in $r\gg 1$ which vanish quadratically at $r^{-1}=0$), since $\pa_t,\pa_{x^j}\colon\rho\CI_{\rm ext}\to\rho^2\CI_{\rm ext}$. Therefore, the coefficients of the second fundamental form satisfy $k_{i j}=g(\nabla_i\pa_j,n)\in\rho^2\CI(\ol{\R^3}\setminus B(0,R))$.

  Regarding claim~\eqref{ItABStable}, note that $\phi:=\exp(-\bfw\cdot\bfB)$ is an invertible linear map on $\R^4$, and thus extends by continuity to a diffeomorphism of $\ol{\R^4}$ (cf.\ \cite[Lemma~5.1.1]{MelroseDiffOnMwc}) preserving the boundary at infinity. As a consequence, push-forward along $\phi$ or $\phi^{-1}$ maps $\Vb(\ol{\R^4})$ into itself, and also maps functions of class $\rho\CI_{\rm ext}$ into itself, and therefore $\phi^*$ preserves symmetric 2-tensors of class $\rho\CI_{\rm ext}$. Furthermore, as already noted, we have $\phi^*(-\dd t^2+\dd x^2)=-\dd t^2+\dd x^2$. This completes the proof that the pullback of~\eqref{EqABClass} under $\phi^*$ is again of the form~\eqref{EqABClass}.
\end{proof}

Therefore, we may apply Theorem~\ref{ThmGlT} to glue a Kerr black hole (whether subextremal or not) into the neighborhood of a point $\fp$ in any given initial data set $(X,\gamma,k)$ (with or without cosmological constant), provided $X$ does not have KIDs in a neighborhood of $\fp$.

\subsection{Gluing a given data set into subextremal Kerr and Kerr--(anti) de~Sitter initial data}
\label{SsAK}

Fix the parameters $\bhm,\bha\in\R$ of a subextremal K(AdS) black hole. We write $r^e=r_{\Lambda,\bhm,\bha}^e$ for the radius of the event horizon. For sufficiently small $\eta>0$, the K(AdS) metric can then be analytically extended to the manifold
\[
  M := \R_\ft \times [r^e-2\eta,\infty) \times \Sph^2_{\theta,\varphi},
\]
where $\ft=t-T(r)$ and $\varphi=\phi-\Phi(r)$ for suitable functions $T,\Phi$; see e.g.\ \cite[Equation~(1.5)]{HintzKdSMS} in the Kerr--de~Sitter case. We require that the level sets of $\ft$ are spacelike and transversal to the future event horizon. For $\delta\in(0,2\eta]$, let
\[
  X_\delta = \ft^{-1}(0) \cap \{ r > r^e-\delta \} \subset M,
\]
and denote by
\[
  \gamma,k\in\CI(X_\delta;S^2\,\Tsc^* X)
\]
the initial data on $X_\delta$ induced by $g_{\Lambda,\bhm,\bha}$. Suppose we are given an $\cE$-asymptotically flat initial data set $(\hat\gamma,\hat k)$ on $\R^3\setminus\hat K^\circ$ as in Definition~\ref{DefCEAf}, which satisfies the constraint equations,. Since $(M,g_{\Lambda,\bhm,\bha})$ has global Killing vector fields, the data set $(X_\delta,\gamma,k)$ \emph{does} have KIDs, i.e.\ the kernel of $L_{\gamma,k}^*$ is nontrivial on any nonempty open subset of $X_\delta$ \cite{MoncriefLinStabI}. Therefore, Theorem~\ref{ThmGlT} is not directly applicable. Note however that in the spacetime evolving from a putative gluing $(\wt\gamma,\wt k)$ of $(\hat\gamma,\hat k)$ into a neighborhood of a point $\fp\in X_\delta$, the exterior domain is independent of the glued data in $r<r^e-\eta$, provided $\eps>0$ is small enough, in view of finite speed of propagation for solutions of the Einstein field equations. Thus, we shall relax the gluing requirements by tolerating violations of the constraint equations in $r<r^e-\eta$:

\begin{thm}[Gluing into K(AdS)]
\label{ThmAK}
  Let $(\hat\gamma,\hat k)$ be $\cE$-asymptotically flat data on $\R^3\setminus\hat K^\circ$ with $P(\hat\gamma,\hat k;0)=0$, and let $(\gamma,k)$ be K(AdS) data on $X_{2\eta}$ as above. Let $\fp\in X_0$, and let $\cU^\circ\subset X_{2\eta}$ be a smoothly bounded connected open set containing $\fp$ with compact closure $\cU=\ol{\cU^\circ}$ disjoint from $\{r=r^e-2\eta\}$, and so that $\cU^\circ\cap(X_{2\eta}\setminus X_\eta)\neq\emptyset$. Then the conclusions of Theorem~\usref{ThmGlT} hold for $n=3$ and $\wt X=\wt X_\eta$.
\end{thm}

See Figure~\ref{FigAK}.

\begin{figure}[!ht]
\centering
\includegraphics{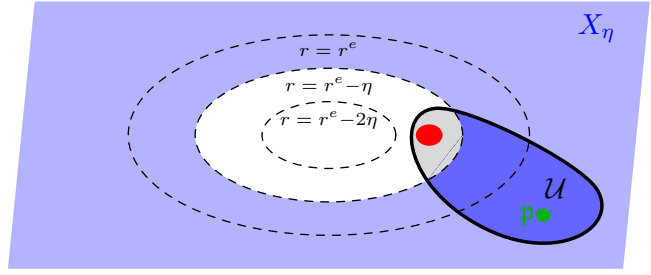}
\caption{Illustration of Theorem~\ref{ThmAK} and its proof; shown is a part of $X_{2\eta}$. We permit violations of the constraint equations in $r<r^e-\eta$; in the gluing procedure, which modifies the data $(\gamma,k)$ in $\cU$, we add suitable error terms (which we may take to be supported in the red region) to the constraints into $X_{2\eta}\setminus X_\eta$.}
\label{FigAK}
\end{figure}

\begin{rmk}[Mass]
\label{RmkAKDisc}
  If we start with Kerr data $(\gamma,k)$, then the glued initial data $(\wt\gamma,\wt k)$, for all $\eps>0$ (in the notation of Theorem~\ref{ThmGlT}), agree with $(\gamma,k)$ for large radii, and therefore their ADM mass is always equal to $\bhm$. Note however that $(\wt\gamma,\wt k)$ differs from $(\gamma,k)$ near the Kerr event horizon, and thus Bartnik-type notions of (quasilocal) mass (see e.g.\ \cite{BartnikQuasilocal,BartnikEnergyLectures,AndersonJaureguiQuasiLocal}) which depend on properties of minimal surfaces or MOTS depend on $\eps$. Heuristically, the proof below thus takes mass away from $(\gamma,k)$ and uses it to glue in the data set $(\hat\gamma,\hat k)$.
\end{rmk}

\begin{rmk}[Other gluing procedures]
  Another possible avenue for accounting for the cokernel of $L_{\gamma,k}$ is to modulate the parameters of the K(AdS) black hole in the gluing procedure, as has been done in related contexts for $\Lambda=0$ in \cite{CorvinoScalar,CorvinoSchoenAsymptotics}, and for $\Lambda>0$ in \cite{ChruscielPollackKottler,CortierKdSGluing}; we shall not pursue this possibility in this paper. Conjecturally, in such a gluing procedure in the case $\Lambda=0$, in which one moreover leaves $(\gamma,k)$ unchanged near the event horizon, the ADM mass of the glued data set is $\bhm+\eps\hat\bhm+o(\eps)$ where $\hat\bhm$ is the ADM mass of $(\hat\gamma,\hat k)$.
\end{rmk}

\begin{proof}[Proof of Theorem~\usref{ThmAK}]
  The main task is to find a suitable replacement for Proposition~\ref{PropCEPctSolv}. Let $\cU_\circ=\upbeta_\circ^*(\cU)$ where $\cU=\ol{\cU^\circ}$. Let $N<\infty$ denote the dimension of the kernel $K^*$ of $L_{\gamma,k}^*$ on $\CI(\cU^\circ)$. By Lemma~\ref{LemmaCEProlong}, elements of $K^*$ have full support on $\cU^\circ$, and therefore we may select $(f_i^*,j_i^*)\in\CIc(\cU^\circ\setminus X_\eta;\ul\R\oplus T^* X_{2\eta})$, $i=1,\ldots,N$, so that
  \[
    E_2\colon \hat\rho^{-\frac{n}{2}}H_{\bop,0 0}^{s+2}(\cU_\circ)\oplus\hat\rho^{-\frac{n}{2}}H_{\bop,0 0}^{s+1}(\cU_\circ;\upbeta^*T^*X_{2\eta})\ni(f^*,j^*)\mapsto(\la(f^*,j^*),(f_i^*,j_i^*)\ra)_{i=1,\ldots,N}\in\C^N
  \]
  restricts to a linear isomorphism $\cK^*\cong\C^N$ where $\cK^*$ is the kernel of $\cL_{\circ,\gamma,k}^*$ on the domain of the map~\eqref{EqCEPctSolv1Inv}. (Note here that elements of $\cK^*$ are rescaled versions of restrictions of pullbacks of elements of $K^*$ to $\cU_\circ$ along $\upbeta_\circ$.) Define $E_1\colon\C^N\to\CIc(\cU^\circ\setminus X_\eta;\ul\R\oplus T^*X_{2\eta})$ by $(c_i)_{i=1,\ldots,N}\mapsto\sum_{i=1}^N c_i(f_i^*,j_i^*)$. Since $\ran E_1$ spans the cokernel of $\cL_{\circ,\gamma,k}\cL_{\circ,\gamma,k}^*$, we conclude that the map
  \[
    \begin{pmatrix}
      \cL_{\circ,\gamma,k}\cL_{\circ,\gamma,k}^* & E_1 \\
      E_2 & 0
    \end{pmatrix}
    \colon \bigl(\hat\rho^{-\frac{n}{2}}H_{\bop,0 0}^{s+2}\oplus\hat\rho^{-\frac{n}{2}}H_{\bop,0 0}^{s+1}\bigr) \oplus \C^N \to \bigl(\hat\rho^{-\frac{n}{2}}H_{\bop,0 0}^{s-2}\oplus\hat\rho^{-\frac{n}{2}}H_{\bop,0 0}^{s-1}\bigr) \oplus \C^N
  \]
  is invertible. Regarding part~\eqref{ItCEPctSolvPhg} of Proposition~\ref{PropCEPctSolv}, we may now uniquely solve, for any given polyhomogeneous $(f,j)$, the system\footnote{The choice $0\in\C^N$ of the right hand side of the equation involving $E_2$ is arbitrary.}
  \[
    L_{\circ,\gamma,k}(h,q)=(f,j)+E_1(c),\qquad
    E_2(h,q)=0,
  \]
  for $c\in\C^N$ and polyhomogeneous $(h,q)$. Since $\supp E_1(c)\subset\cU^\circ\setminus X_\eta$, this implies that $L_{\circ,\gamma,k}(h,q)=(f,j)$ on $X_\eta$.

  We similarly modify Proposition~\ref{PropCETotSolv}: we may take the same $(f_i^*,j_i^*)$, regarded as $\eps$-independent tensors on $\wt X'$, and consider the map
  \begin{equation}
  \label{EqAKNewMap}
    \begin{pmatrix} \cL_{\wt\gamma,\wt k}\cL_{\wt\gamma,\wt k}^* & E_1 \\ E_2 & 0 \end{pmatrix}
  \end{equation}
  between the direct sums of the spaces in~\eqref{EqCETotSolv} with $\C^N$. This can be shown to be invertible with uniformly bounded inverse for all sufficiently small $\eps>0$ by following the proof of Proposition~\ref{PropCETotSolv}: for the normal operator argument at $\hat X_{2\eta}$, one works only with $\cL_{\wt\gamma,\wt k}\cL_{\wt\gamma,\wt k}^*$ and obtains an improvement of the weight at $\hat X_{2\eta}$, while the normal operator argument at $(X_{2\eta})_\circ$ provides the estimate for the $\C^N$-summand, and gives the improvement of the weight at $(X_{2\eta})_\circ$ as before. Lemma~\ref{LemmaCETotReg} remains valid, \emph{mutatis mutandis}, for the map~\eqref{EqAKNewMap}.

  The arguments of~\S\ref{SGl} now apply with minor modifications; for instance, instead of~\eqref{EqGlTPDE}, one solves the system
  \[
    w_L P\bigl((\wt\gamma,\wt k)+w_R\cL_{\wt\gamma,\wt k}^*(\wt f,\wt j)\bigr) + E_1(c) = 0, \qquad
    E_2\bigl((\wt\gamma,\wt k)+w_R\cL_{\wt\gamma,\wt k}^*(\wt f,\wt j)\bigr) = 0.
  \]
  We leave the details to the reader.
\end{proof}

If we take $(\hat\gamma,\hat k)$ to be (boosted) Kerr initial data, then Theorem~\ref{ThmAK} produces initial data sets in which a small mass Kerr black hole is glued into a given K(AdS) initial data set. These are natural initial data sets which (depending on the choice of location and boost parameter of the small black hole, which determine the geodesic along which the small Kerr black hole moves) conjecturally evolve into spacetimes describing extreme mass ratio inspirals; see~\S\ref{SsIO}.

\appendix
\section{Quasilinear elliptic Douglis--Nirenberg systems}
\label{SDN}

We work on $\R^n$, $n\geq 1$, with points in $\R^n$ denoted by $x$. Let $N\in\N$. Consider $u(x)=(u_k(x))_{k=1,\ldots,N}$, with $u_k$ valued in $\R^{d_k}$. Let $t_1,\ldots,t_N,s_1,\ldots,s_N\in\N_0$. We consider a system $(P_j(u))_{j=1,\ldots,N}=0$ of $N$ nonlinear partial differential equations for $u$ which is quasilinear in the following Douglis--Nirenberg sense: writing
\[
  D^{<m}u := (u,D u,\ldots,D^{m-1}u)
\]
for the vector of all (mixed) $x$-derivatives of $u$ up to order $m-1$, we have
\begin{align*}
  P_j(x,u) &= \sum_{k=1}^N \sum_{|\alpha|=t_k+s_j}L_{j k,\alpha}\bigl(x,(D^{<t_\ell+s_j}u_\ell(x))_{\ell=1,\ldots,N}\bigr) D^\alpha u_k(x) \\
    &\qquad + \tilde P_j\bigl(x,(D^{<t_\ell+s_j}u_\ell(x))_{\ell=1,\ldots,N}\bigr),
\end{align*}
where $L_{j k,\alpha}$ and $\tilde P_j$ are smooth functions of their arguments. We furthermore assume that $P$ is elliptic at $u$ in the sense that the linear operator
\[
  L(x,(D^{<t_\ell+s_j}u_\ell),D) := \Bigl( \sum_{|\alpha|=t_k+s_j} L_{j k,\alpha}(x,(D^{<t_\ell+s_j}u_\ell))D^\alpha \Bigr)_{j,k=1,\ldots,N} \in \bigl( \Diff^{t_k+s_j} \bigr)_{j,k=1,\ldots,N}
\]
is elliptic in the Douglis--Nirenberg sense.

\begin{prop}[Bootstrap]
\label{PropDNEll}
  Let $\bar s=\max s_j$. Suppose that $u\in\bigoplus_{k=1}^N\cC^{t_k+\bar s,\alpha}(\R^n;\R^{d_k})$, $\alpha\in(0,1)$, is a solution of $P(x,u)=0$. Then $u\in\CI$.
\end{prop}
\begin{proof}
  This is a standard scaling and linear approximation argument, but we give a proof here for completeness. Since the claim is local, it suffices to work near $x=0$. Assume that $u_k\in\cC^{t_k+\bar s+m,\alpha}$ with $m\in\N_0$. For $\delta>0$ to be determined later, consider the rescaled coordinate $\hat x=\frac{x}{\delta}$, and write $\hat D=\delta^{-1}D$ for its coordinate derivatives. Setting $\hat u=\diag(\delta^{-t_k})u$, the rescaling $(\delta^{s_j}P_j(x,u))_{j=1,\ldots,N}=0$ then reads
  \[
    L(\delta\hat x,(D^{<t_\ell+s_j}u_\ell(\delta\hat x)),\hat D) \hat u = f(\delta) := -\bigl(\delta^{s_j}\tilde P_j(\delta\hat x,(D^{<t_\ell+s_j}u_\ell(\delta\hat x)))\bigr)_{j=1,\ldots,N}.
  \]
  Working in $|\hat x|\leq 2$, we have $f(\delta)\in\bigoplus_j \cC^{m+1+\bar s-s_j,\alpha}$ (with norm depending on $\delta$). Moreover, since the norm of $D^{<t_\ell+s_j}u_\ell(\delta\hat x)-D^{<t_\ell+s_j}u_\ell(0)\in\cC^{m+1+\bar s-s_j,\alpha}$ is of size $\cO(\delta)$, we can rewrite this further as
  \[
    \cL\hat u + \tilde\cL\hat u = f(\delta),
  \]
  where $\cL=L(0,(D^{<t_\ell+s_j}u_\ell(0)),\hat D)$ is Douglis--Nirenberg elliptic, and where the remainder term $\tilde\cL\in(\Diff^{t_k+s_j})_{j,k=1,\ldots,N}$ has coefficients of size $\cO(\delta)$ in $|\hat x|\leq 2$. Let $\chi\in\CIc(B(0,2))$ be identically $1$ on $B(0,1)$; then
  \[
    \cL(\chi\hat u) + \tilde\cL(\chi\hat u) = \tilde f(\delta),\qquad
    \tilde f(\delta):=\chi f(\delta) + [\cL+\tilde\cL,\chi]\hat u \in \bigoplus_j \cC^{m+1+\bar s-s_j,\alpha}.
  \]
  Let $Q\in(\Psi^{-s_k-t_j})_{j,k=1,\ldots,N}$ denote a parametrix of $\cL$ near $\supp\chi$, with Schwartz kernel compactly supported in $B(0,2)\times B(0,2)$, and so that $Q\cL=I+R$ where $R\circ\chi\in\Psi^{-1}$. We then have
  \[
    Q\tilde f(\delta) = (Q\cL + Q\tilde\cL)(\chi\hat u) = (I+Q\tilde\cL)(\chi\hat u) + R(\chi\hat u).
  \]
  But $Q\tilde f(\delta),R(\chi\hat u)\in\bigoplus_k \cC^{m+1+\bar s+t_k,\alpha}$, while the operator norm of $Q\tilde\cL$ on $\bigoplus_k \cC^{m+\eta+\bar s+t_k,\alpha}$ is of size $\cO(\delta)$ uniformly for $\eta\in[0,1]$. (We use here the boundedness of ps.d.o.s on H\"older spaces, see e.g.\ \cite[\S13.8]{TaylorPDE3}.) Thus, for sufficiently small $\delta$, we can invert $I+Q\tilde\cL$ (with the choice $\eta=0$) via a Neumann series, giving
  \[
    \chi\hat u = (I+Q\tilde\cL)^{-1}\bigl(Q\tilde f(\delta)-R(\chi\hat u)\bigr).
  \]
  But the right hand side is now also defined for the choice $\eta=1$ (upon shrinking $\delta$ further if necessary). This improves the regularity of $\chi\hat u$, and thus of $\chi(x/\delta)u(x)$, by $1$ order relative to the inductive hypothesis. The proof is complete.
\end{proof}

\bibliographystyle{alphaurl}


\end{document}